\documentclass[12pt]{amsart}

\usepackage{amsmath} 
\usepackage{amsthm} 
\usepackage{amssymb} 

\usepackage{microtype} 
\usepackage{pinlabel} 

\newcommand{\sqdiamond}{\tikz [x=1.2ex,y=1.85ex,line width=.1ex,line join=round, yshift=-0.285ex] \draw  (0,.5) -- (0.75,1) -- (1.5,.5) -- (0.75,0) -- (0,.5) -- cycle;}%
\renewcommand{\Diamond}{$\sqdiamond$} 

\usepackage[scaled=0.9]{sourcecodepro} 

\usepackage{enumitem}
\usepackage{color}
\usepackage{subfig, caption}

\usepackage{wrapfig}
\usepackage{pdflscape}
\usepackage{array}
\usepackage{tikz-cd}
\usepackage{longtable, booktabs}

\usepackage[hidelinks, pagebackref]{hyperref}

\makeatletter
\define@key{href}{font}{#1}
\makeatother
\usepackage{xpatch}
\newcommand\hrefdefaultfont{\ttfamily}
\xpatchcmd\href{\setkeys{href}{#1}}{\setkeys{href}{font=\hrefdefaultfont,#1}}{}{\fail}

\renewcommand*{\backref}[1]{}
\renewcommand*{\backrefalt}[4]{
  \ifcase #1 
  [No citations.]
  \or [#2]
  \else [#2]
  \fi }

\usepackage{cite}
\makeatletter
\newcommand{\citecomment}[2][]{\citenum{#2}#1\citevar}
\newcommand{\citeone}[1]{\citecomment{#1}}
\newcommand{\citetwo}[2][]{\citecomment[,~#1]{#2}}
\newcommand{\citevar}{\@ifnextchar\bgroup{,~\citeone}{\@ifnextchar[{,~\citetwo}{]}}}
\newcommand{\citefirst}{\@ifnextchar\bgroup{\citeone}{\@ifnextchar[{\citetwo}{]}}}
\newcommand{\cites}{[\citefirst}
\makeatother

\let\originalleft\left
\let\originalright\right
\renewcommand{\left}{\mathopen{}\mathclose\bgroup\originalleft}
\renewcommand{\right}{\aftergroup\egroup\originalright}

\newcommand{\calA}{\mathcal{A}}
\newcommand{\calB}{\mathcal{B}}

\newcommand{\calD}{\mathcal{D}}

\newcommand{\calL}{\mathcal{L}}
\newcommand{\calM}{\mathcal{M}}

\newcommand{\calO}{\mathcal{O}}
\newcommand{\calP}{\mathcal{P}}

\newcommand{\calR}{\mathcal{R}}
\newcommand{\calS}{\mathcal{S}}
\newcommand{\calT}{\mathcal{T}}
\newcommand{\calU}{\mathcal{U}}
\newcommand{\calV}{\mathcal{V}}
\newcommand{\calW}{\mathcal{W}}

\newcommand{\CC}{\mathbb{C}}

\newcommand{\HH}{\mathbb{H}}

\newcommand{\NN}{\mathbb{N}}

\newcommand{\RR}{\mathbb{R}}

\newcommand{\ZZ}{\mathbb{Z}}

\newcommand{\st}{\mathbin{\mid}} 
\newcommand{\from}{\colon} 
 
\newcommand{\homeo}{\mathrel{\cong}} 

\newcommand{\isom}{\cong} 

\newcommand*{\medcap}{\mathbin{\scalebox{1.5}{\ensuremath{\cap}}}}
\newcommand*{\medcup}{\mathbin{\scalebox{1.5}{\ensuremath{\cup}}}}

\newcommand{\cover}[1]{{\widetilde{#1}}}

\newcommand{\closure}[1]{{\overline{#1}}}

\newcommand{\bdy}{\partial} 
 
\newcommand{\link}{\operatorname{link}}

\newcommand{\CP}{\mathbb{CP}} 

\newcommand{\Stab}{\operatorname{Stab}}

\newcommand{\Isom}{\operatorname{Isom}} 

\input{header_article.tex}

\theoremstyle{plain}
\newtheorem{XXXtheoremQED}[equation]{Theorem} 
  {\pushQED{\qed}\begin{XXXtheoremQED}}
  {\popQED\end{XXXtheoremQED}}

\newcommand{\fakeenv}{} 

\newenvironment{restate}[2]  
{ 
 \renewcommand{\fakeenv}{#2} 
 \theoremstyle{plain} 
 \newtheorem*{\fakeenv}{#1~\ref{#2}} 
 \begin{\fakeenv}
}
{
 \end{\fakeenv}
}

\newenvironment{restated}[2]  
{ 
 \renewcommand{\fakeenv}{#2} 
 \theoremstyle{definition} 
 \newtheorem*{\fakeenv}{#1~\ref{#2}} 
 \begin{\fakeenv}
}
{
 \end{\fakeenv}
}

\usepackage{subfig}
\usepackage[margin=0.5cm]{caption}
\usepackage{xcolor}

\usepackage{graphicx}

\newcommand{\preacw}{\scalebox{0.7}{$\curvearrowleft$}}
\newcommand{\acw}{{\rotatebox[origin=c]{-90}{\preacw}}}
\newcommand{\sign}{\operatorname{sign}}

\renewcommand{\link}{\mathsf{L}}
\newcommand{\mat}{\textrm{mat}}
\newcommand{\bdyi}{\bdy_\infty}  

\newcommand{\Circle}{\calS^1(\calV)}
\newcommand{\Sphere}{\calS^2(\calV)}
\newcommand{\Link}{\link(\calV)}
\newcommand{\Disc}{\calD(\calV)}
\newcommand{\Mobius}{\calM(\calV)}

\newcommand{\acts}{\curvearrowright}

\newcommand{\subgroup}[1]{\langle #1 \rangle}

\newcommand{\Hull}{\textsf{Hull}}

\newsavebox{\FigEightVeer}
\newsavebox{\fibredSnappy}
\newsavebox{\fibredVeer}
\newsavebox{\NonfibredSnappy}
\newsavebox{\NonfibredVeer}

\begin{lrbox}{\FigEightVeer}
\small
\verb+cPcbbbiht_12+
\end{lrbox}
\begin{lrbox}{\fibredSnappy}
\small
\verb+m232+
\end{lrbox}
\begin{lrbox}{\fibredVeer}
\small
\verb+fLLQcbecdeehhnkei_12001+
\end{lrbox}
\begin{lrbox}{\NonfibredSnappy}
\small
\verb+s227+
\end{lrbox}
\begin{lrbox}{\NonfibredVeer}
\small
\verb+gLLAQbecdfffhhnkqnc_120012+
\end{lrbox}

\usepackage{capt-of}  

\makeatletter
\newcommand{\titlefigure}[1]{\def\@titlefigure{#1}}
\let\@titlefigure\@empty

\let\old@setabstract\@setabstract
\renewcommand{\@setabstract}{%
  \ifx\@titlefigure\@empty
  \else
    \begin{center}
      \@titlefigure
    \end{center}
    \vskip .5\baselineskip
  \fi
  \old@setabstract
}
\makeatother

\begin{document}

\title[To convergence actions]{From veering triangulations to convergence actions and back again}

\author[Manning]{Jason Fox Manning}
\address{\hskip-\parindent
        Department of Mathematics\\
        Cornell University\\
        Ithaca, New York, 14853, United States}
\email{jfmanning@cornell.edu}
\author[Schleimer]{Saul Schleimer}
\address{\hskip-\parindent
        Mathematics Institute\\
        University of Warwick\\
        Coventry CV4 7AL, United Kingdom}

\email{s.schleimer@warwick.ac.uk}
\author[Segerman]{Henry Segerman}
\address{\hskip-\parindent
        Department of Mathematics\\
        Oklahoma State University\\
        Stillwater, Oklahoma, 74078, United States\\}
\email{henry.segerman@okstate.edu}

\date{\today}

\titlefigure{%
\includegraphics[width = 0.8\textwidth]{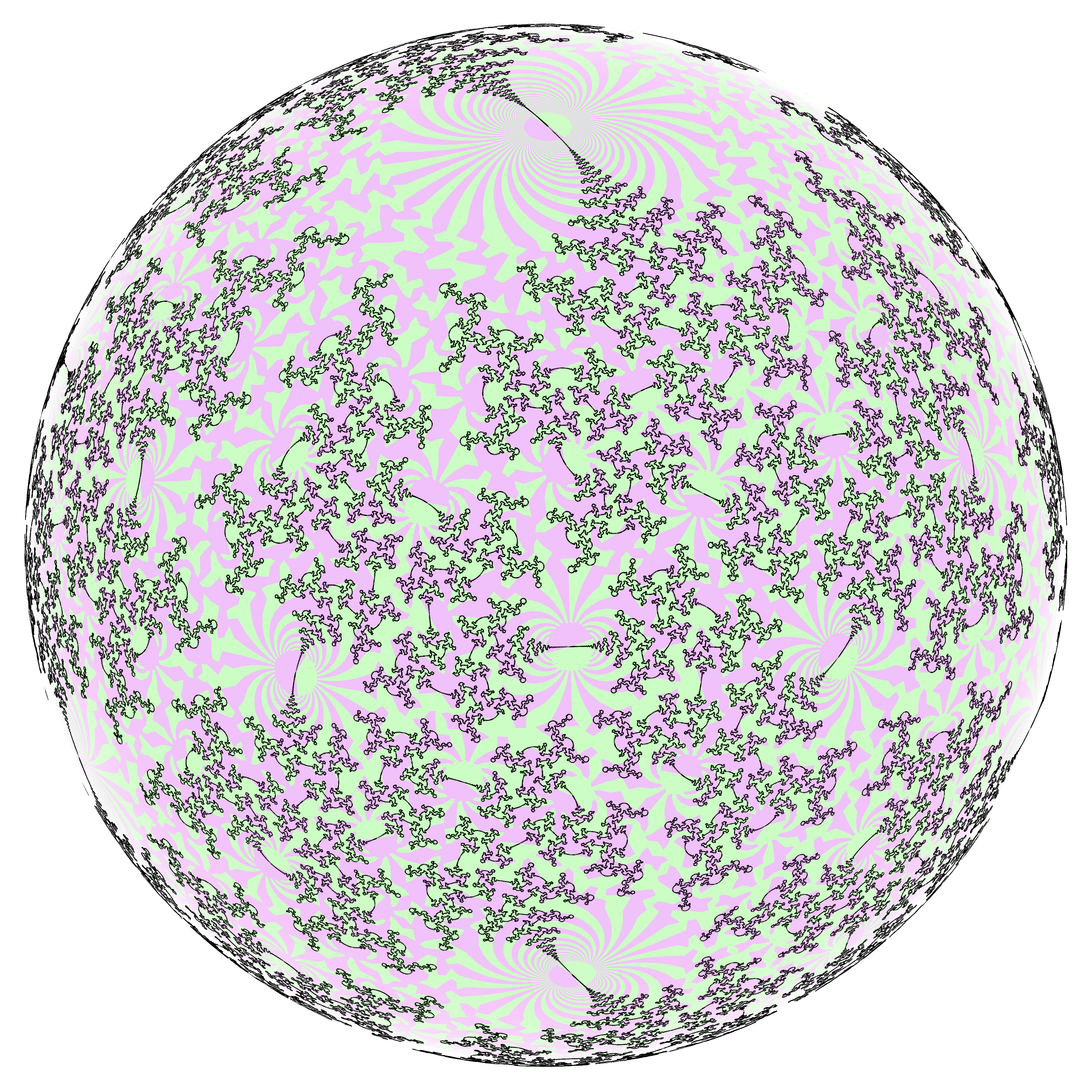}
\captionof{figure}{Two approximations to the Cannon--Thurston map associated to the figure-eight knot complement.
Thurston's technique is in black;
ours is the common boundary of the green and purple regions.}
\label{Fig:CT_fig8_intro}
}

\begin{abstract}
Suppose that $M$ is a finite-volume cusped hyperbolic three-manifold, equipped with a veering triangulation $\calV$.
We prove that the action of the fundamental group of $M$ on the veering two-sphere is a geometrically finite convergence action.
Applying a result of Yaman, we deduce that the veering two-sphere is equivariantly homeomorphic to the boundary of hyperbolic space. 

As an application, we obtain Cannon--Thurston maps associated to veering triangulations.  
If $\calV$ is layered we recover the classical Cannon--Thurston map.
If it is not we obtain Cannon--Thurston maps that do not come from surface subgroups. 
These are the first such examples in the cusped case.

Finally, we implement an algorithm to draw approximations of these Cannon--Thurston maps. 
This improves upon previous approximations obtained by Thurston and others. 
\end{abstract}






\vspace*{-2cm}
\maketitle

\section{Introduction}
\label{Sec:Intro}

\subsection{Motivation}

In a paper first circulated as a preprint in the 1980s, Cannon and Thurston~\cite{CannonThurston07} produce sphere-filling curves associated to closed hyperbolic three-manifolds that fibre over the circle.  
Let $M$ be such a fibred manifold with fibre $F$.
Set $G = \pi_1(M)$ and $H = \pi_1(F)$.  
Since $F$ is a closed surface, with genus at least two, we have $\bdyi H \homeo S^1$ and $\bdyi G \homeo S^2$.  
Cannon and Thurston show:
\begin{enumerate}
\item
\label{dir:mj} 
there is a continuous extension of $H\hookrightarrow G$ to a map $\bdyi H\to \bdyi G$, and
\item
\label{dir:dc} 
there is an action of $G$ on the circle $\bdyi H\homeo S^1$ (extending the action of $H$), with respect to which the above extension is equivariant, surjective, and sphere-filling.
\end{enumerate}
At least two overlapping but distinct mathematical directions emerge from this remarkable paper, giving rise to two meanings of the term ``Cannon--Thurston map''.
  
In the direction of~\eqref{dir:mj}, we consider a group $H$ acting properly by isometries on a Gromov hyperbolic space $X$.
We ask whether an orbit map extends continuously to a map from some boundary for $H$ to the Gromov boundary of $X$.
Such an extension is often called a Cannon--Thurston map.
For example, if $H$ is a geometrically finite kleinian group and $X$ is $\HH^3$, then many authors consider cases of this question, with a complete affirmative answer obtained by Mj~\cite{Mj17}.
When $H$ is a hyperbolic group, and $X$ is an ambient hyperbolic supergroup, the Cannon--Thurston map sometimes does not exist~\cite{BakerRiley13}.
Nonetheless, Mj and collaborators prove the existence of an extension in many circumstances;
see, for example~\cites
{Mj98top} 
{Mj98jdg} 
{MjSardar12}. 
Many of these ideas are extended to the relatively hyperbolic setting in~\cite{MjPal11}.
Much more on this direction can be found in Mj's ICM address~\cite{Mj18}. 

In the direction of~\eqref{dir:dc}, we fix a group $G$ and an appropriate boundary $Z$.
We ask whether there is an action $G\curvearrowright S^1$ leading to an equivariant surjection $S^1\to Z$.
When $G = \pi_1(M)$ for a hyperbolic three-manifold $M$, then $Z$ is the sphere at infinity of hyperbolic three-space.
 For a general hyperbolic (or relatively hyperbolic) group, $Z$ is instead the Gromov (or Bowditch) boundary of $G$.
From this point of view a \emph{Cannon--Thurston map} consists of an action of $G$ on a circle, and an equivariant map from that circle \emph{onto} $Z$.
This is the point of view we take in the current paper; 
see \refdef{CT}.

The first Cannon--Thurston maps not coming from fibrations are constructed by Fenley from pseudo-Anosov flows (without perfect fits) on three-manifolds~\cite{Fenley12}.
Flows with weaker hypotheses are shown to give rise to Cannon--Thurston maps in~\cites
{Frankel15} 
{Fenley16}. 
These kinds of Cannon--Thurston maps are the subject of much current research;
see~\cites{FLT26} 
{CalegariLoukidou26}. 
A different and quite recent construction of Cannon--Thurston maps is given by Buckminster~\cite{Buckminster26};
the domain of Buckminster's Cannon--Thurston map is the \emph{leftmost universal circle} constructed by Calegari--Dunfield from a taut foliation on a closed hyperbolic three-manifold in~\cite{CalegariDunfield03}.
In all of these papers, the Cannon--Thurston map arises from some structure on a closed hyperbolic manifold.   

These constructions do not obviously generalise to the cusped setting. 
Thurston anticipated work in the fibred cusped case by producing the first images of approximations to the Cannon--Thurston map for the figure-eight knot complement~\cite{Thurston82}.
His pictures were made rigorous by Alperin, Dicks, and Porti~\cite{AlperinDicksPorti99}.
Bowditch generalises the original Cannon--Thurston machinery to punctured surface groups in~\cite[Section~6]{Bowditch07};  
see also~\cite{Mj09}.

\subsection{Our results} 
We focus on the cusped, but not necessarily fibred, case.
We construct Cannon--Thurston maps from a structure which can only occur in the finite-volume cusped setting: namely, a \emph{veering triangulation} (\refdef{Veering}).

Suppose that $M$ is a complete and finite-volume cusped hyperbolic three-manifold.
In the case that $M$ is fibred, with monodromy $\varphi$ having no internal singularities, Agol~\cite{Agol11} constructs a (layered) veering triangulation $\calV$ of $M$.
This is an ideal triangulation whose combinatorics are canonically generated by $\varphi$.
Gu\'eritaud~\cite{Gueritaud16}
connects the combinatorics of the Cannon--Thurston map for $\varphi$ to the combinatorics of the triangulation $\cover{\calV}$ of the universal cover.

The more general (not necessarily layered) notion of a veering triangulation is introduced in~\cite{HRST11};
they give the first veering triangulations on non-fibred manifolds.
In previous work with Frankel~\cite{FSS22}, the second two authors construct 
the \emph{veering circle}, $\Circle$ (\refsec{Dynamics}),
from the data of the veering triangulation $\cover{\calV}$ of the universal cover $\cover{M}$.

In this paper, we collapse the veering circle to obtain the \emph{veering two-sphere}, $\Sphere$ (\refdef{VeeringSphere}).
Our main result is the following:

\begin{restate}{Theorem}{Thm:Convergence}
Suppose that $\calV$ is a finite locally veering triangulation of a three-manifold $M$. 
Then the action of $\pi_1(M)$ on the veering two-sphere $\Sphere$ is a geometrically finite convergence action.  
Furthermore, its parabolic points are exactly the cusps $\Delta_\calV$.
\end{restate}

Combined with the work of Yaman~\cite{Yaman04}, we obtain the following. 

\begin{restate}{Corollary}{Cor:Yaman}
Suppose that $\calV$ is a finite locally veering triangulation of a three-manifold $M$. 
Then the veering two-sphere $\Sphere$ is equivariantly homeomorphic to $\bdyi \HH^3$ (as equipped with the action of $\pi_1(M)$ coming from the discrete and faithful representation). 
\end{restate}

This, together with the map from $\Circle$ to $\Sphere$, gives a Cannon--Thurston map $\Psi_\calV$ for a veering triangulation (\refthm{VeeringCT}).
When the veering triangulation is layered, this recovers the Cannon--Thurston map for any carried fibration. 
When the veering triangulation is not layered, our Cannon--Thurston maps are new (\refcor{NotMadeThere}).
To our knowledge, these give the first examples of Cannon--Thurston maps associated to non-fibred finite-volume cusped hyperbolic manifolds.
(Although one could perhaps produce a construction using Thurston's slitherings paper~\cite{Thurston97}.)

The first illustrations of (approximations to) Cannon--Thurston maps are due to Thurston~\cites{Thurston82}{Thurston98}.
Like his, our machinery is explicit enough that we also obtain illustrations.
Our illustrations are significantly different from Thurston's, even in the fibred case, as shown in \reffig{CT_fig8_intro}.
In particular, we obtain better approximations while doing less computational work.
For details, see \refapp{Drawing}.
This answers a question of Fran\c{c}ois Gu\'eritaud~\cite{Gueritaud13}.

\subsection{Outline}

In \refsec{Triangulations} we recall the definitions of \emph{ideal}, \emph{taut}, and \emph{veering triangulations}.
We also review (from \cite{FSS22}) the \emph{veering circle} and its associated laminations.
In \refsec{Sphere} we define the \emph{veering two-sphere}.
In \refsec{Loom} we recall the definitions of \emph{loom} and \emph{link} spaces and a version of the \emph{astroid lemma} (\refsec{Astroid}).
In \refsec{Disc} we prove that the veering circle $\Circle$ compactifies the link space $\Link$, giving the \emph{veering disc} $\Disc$ (\refthm{Disc}).
In \refsec{Hausdorff} we characterise the possible Hausdorff limits of sequences of leaves (\refprop{LimitOfLeaves}) and then sequences of skeletal rectangles (\refprop{LimitOfRectangles}) in $\Disc$.

In \refsec{GeomFinite} we recall the definition of a \emph{geometrically finite convergence action}.
In \refsec{ObtainingConvergence} we prove that the action of $\pi_1(M)$ on the veering two-sphere is a convergence action (\refprop{Convergence}).
In Sections \ref{Sec:Parabolic},  \ref{Sec:LeafConical}, and  \ref{Sec:SingletonConical} we show that cusps are bounded parabolic points and all other points of $\Sphere$ are conical.
Proving that singletons are conical requires the machinery of \emph{hulls} between pairs of points of $\Disc$ (\refdef{Hull}).
We gather these results together and prove our main theorem in \refsec{MainProof}.
In \refsec{CannonThurston} we apply our results to Cannon--Thurston maps.
In particular, we prove that one can recover the veering triangulation $\calV$ from the Cannon--Thurston map $\Psi_\calV$ (\refthm{Equivalence}).

In \refapp{Action} we show that the action of $\pi_1(M)$ on $\Circle$ also recovers $\calV$ (\refthm{Isomorphism}).
In \refapp{InternalSing} we give a lemma that we will use in future work on classifying veering triangulations.
In \refapp{Drawing} the second and third authors give a new algorithm for drawing approximations to Cannon--Thurston maps.
Finally, \refapp{Topology} collects the necessary tools from point-set topology needed in Sections~\ref{Sec:Sphere} and~\ref{Sec:Disc}.

\subsection*{Acknowledgements} 

We thank Ian Agol, Nathan Dunfield, Fran\c{c}ois Gu\'eritaud, Katie Mann, and Sam Taylor for helpful conversations.

This material is based in part upon work supported by the National Science Foundation under Grant No.~DMS-1439786 and the Alfred P. Sloan Foundation award G-2019-11406 while the authors were in residence at the Institute for Computational and Experimental Research in Mathematics in Providence, RI, during the Illustrating Mathematics program.

This material is based in part upon work supported by the National Science Foundation under Grant No.~DMS-1928930, while the second and third authors were in residence at the Simons Laufer Mathematical Sciences Institute in Berkeley, California, during the Topological and Geometric Structures in Low Dimensions program.

The first author was visiting Cambridge University for part of this work and thanks Trinity College and DPMMS for their hospitality.
He also thanks the Simons Foundation for support (grants \#524176 and \#942496, JFM).

The third author visited Institut Henri Poincar\'e (UAR 839 CNRS-Sorbonne Universit\'e), supported by LabEx CARMIN (ANR-10-LABX-59-01) for part of this work. 
He was also supported in part by National Science Foundation grants DMS-1708239 and DMS-2203993.

\subsection*{Tool use declaration}

The second and third authors used Sagemath~\cite{sage}, Regina~\cite{regina}, SnapPy~\cite{snappy}, PyX~\cite{pyx}, and Rhinoceros~\cite{rhino} together with the Veering codebase~\cite{VeeringCodebase} to generate the images of approximations to Cannon--Thurston maps in this paper.

The second author used various AI models to suggest line-by-line copyediting in the final stages of writing.
Transcripts are available upon request.

\section{Triangulations}
\label{Sec:Triangulations}

We review the definition of an ideal triangulation.
See~\cite[Sections~4.1 and~4.2]{Thurston80} for more details.

\begin{definition}
\label{Def:IdealTriangulation}
Suppose that $(t_i)$ is a disjoint collection of tetrahedra.
Suppose that $(\phi_j)$ is a collection of linear homeomorphisms between the faces of the $t_i$ that pairs every face with exactly one other face.
Let $\calT$ be the result of quotienting the disjoint union of the $(t_i)$ by the face pairings $(\phi_j)$.

Suppose that $M$ is a three-manifold with boundary.
If $M - \bdy M$ is homeomorphic to $\calT$ minus its vertices, then we say that $\calT$ is an \emph{ideal triangulation} of $M$.
\end{definition}

We say that $\calT$ is \emph{finite} if the collection $(t_i)$ is finite.
We take $\cover{\calT}$ to be the induced ideal triangulation of the universal cover $\cover{M}$.
This gives a canonical bijection between $\Delta_\calT$, the ideal vertices of $\cover{\calT}$, and $\Delta_M$, the components of the boundary of the universal cover $\cover{M}$.
We refer to the elements of $\Delta_\calT$ or $\Delta_M$ as \emph{cusps}.

\subsection{Taut triangulations}
\label{Sec:Taut}

We review the definition of a taut ideal triangulation.
See~\cite[pages~370 and~371]{Lackenby00} for more details.

\begin{definition}
\label{Def:Taut}
Suppose that $\calT$ is an ideal triangulation with tetrahedra $(t_i)$.
Let $E$ be the collection of all edges of all tetrahedra (six per tetrahedron).
Let $\alpha \from E \to \{0,\pi\}$ be a map such that
\begin{itemize}
\item for each vertex $v$ of each tetrahedron $t_i$, the sum of the values assigned to the edges of $t_i$ incident to $v$ is $\pi$, and
\item for each collection of edges that are identified together in the quotient, the sum of the assigned values is $2\pi$.
\end{itemize}
Then we say that $\alpha$ is a \emph{taut angle structure} on $\calT$.
\end{definition}

We suppress the notation for $\alpha$ and refer to $\calT$ as a \emph{taut ideal triangulation}.
Smoothing the two-skeleton of a taut ideal triangulation according to the angles gives a branched surface $\calB = \calB_\calT$ (without vertices) properly embedded in $M$.
We call $\calB$ the \emph{horizontal branched surface}.

If $\calB$ admits a coorientation then we say that $\calT$ is \emph{transverse taut}.
A \emph{weighting} of $\calB$ is an assignment of (non-negative) real numbers to the faces.
A weighting satisfies the \emph{branched surface equation} at an edge $e$ if the sums of the weights, on either side (as determined by the smoothing) of $e$, are equal.

\begin{definition}
\label{Def:Layered}
Suppose that $\calT$ is a transverse taut triangulation of a manifold $M$.
We say that $\calT$ is \emph{layered} if the branched surface equations for $\calB$ have a positive solution.
\end{definition}

If $\calT$ is a layered triangulation of a manifold $M$ then $M$ can be realised as a surface bundle with fibre carried by $\calB$.
Conversely, if $M$ is fibred then it admits a layered triangulation.
See the first paragraph of~\cite[page~377]{Lackenby00}.

\begin{lemma}
\label{Lem:CommensurablePreservesLayered}
Suppose that $(M, \calS)$ and $(N, \calT)$ are commensurable.
Then $\calS$ is layered if and only if $\calT$ is. 
\end{lemma}

\begin{proof}
Let $(P, \calR)$ be the common finite-sheeted cover.
Suppose that $\calS$ is layered.
By definition, the horizontal branched surface $\calB_\calS$ admits a positive weighting. 
Lifting, we find that $\calB_\calR$ admits a positive weighting.
We find a positive weighting for $\calB_\calT$ by, for any fixed face, adding the weights of its finitely many preimages.
Thus $\calT$ is layered.
\end{proof}

\subsection{Veering triangulations}
\label{Sec:Veering}

We review the definitions.
See~\cite[Definition~4.1]{Agol11}, \cite[Definition~1.3]{HRST11}, \cite[Definition~1.2]{FuterGueritaud13}, and~\cite[Section~1.2]{Gueritaud16} for more details.

\begin{definition}
\label{Def:Veering}
Suppose that $\calT$ is a transverse taut ideal triangulation of an oriented manifold $M$.
We pull back the orientation of $M$ to obtain orientations on the tetrahedra $t_i$.
We say that $\calT$ is \emph{transverse veering} if there is a two-colouring, in red and blue, of the edges of the one-skeleton of $\calT$ with the following property.
We pull the colouring back to the edges of each tetrahedron $t_i$.
For every vertex $v$ of $t_i$, the edges of $t_i$ incident to~$v$, taken in anticlockwise order as viewed from $v$, have
\begin{itemize}
\item angle $\pi$,
\item the colour blue, and
\item the colour red. \qedhere
\end{itemize}
\end{definition}

\begin{remark}
\label{Rem:LocallyVeering}
Throughout most of the paper we will simply refer to transverse veering triangulations as veering triangulations.
However, when the need arises, we consider taut triangulations (possibly not transverse) on three-manifolds (possibly not orientable).
We call such a taut triangulation \emph{locally veering} if it becomes transverse veering in a finite cover.
See~\cite[Definition~10.2]{FSS22}.
\end{remark}

\subsection{Hyperbolic geometry of veering triangulations}

By taking covers, one can promote a locally veering triangulation to a veering triangulation.
This allows us to apply the work of~\cite{HRST11}.

\begin{lemma}
\label{Lem:VeeringHyperbolic}
Suppose that $\calV$ is a finite locally veering triangulation of a three-manifold $M$.
Then the interior of $M$ admits a complete finite-volume hyperbolic metric.
\end{lemma}

\begin{proof}
It suffices to show the interior of some finite-sheeted cover of $M$ admits a complete finite-volume hyperbolic metric;
see~\cite[Theorem~8.36]{Kapovich10}.
Since $M$ admits a locally veering triangulation, there is a finite-sheeted cover $M' \to M$ which admits a veering triangulation.  
Appealing to~\cite[Theorem~1.5]{HRST11}, and the sentence immediately after it, we obtain the desired hyperbolic structure on this cover.
\end{proof}

\begin{remark}
\label{Rem:VeeringHyperbolic}
With \reflem{VeeringHyperbolic} in hand, for the remainder of the paper, whenever $M$ is a manifold with a finite (locally) veering triangulation we fix an identification between the universal cover of $M - \bdy M$ and hyperbolic three-space $\HH^3$.
\end{remark}

\subsection{Dynamics of veering triangulations}
\label{Sec:Dynamics}

From a (transverse) veering triangulation we build a \emph{veering circle} $\Circle$; 
it is equipped with an orientation as well as an orientation-preserving, continuous action of $\pi_1(M)$~\cite[Theorem~7.1]{FSS22}. 
If $x, y \in \Circle$ are distinct points then we use $(y, x)^\acw$ to denote the component of $\Circle - \{x, y\}$ which is anticlockwise of $y$ and clockwise of $x$.  
We use $[y, x]^\acw$ to denote the closure of $(y, x)^\acw$ in $\Circle$. 
The cusps $\Delta_\calV$ are a distinguished countable dense subset of $\Circle$.

The \emph{M\"obius band past infinity}, denoted $\Mobius$, is the space of unordered pairs of points in $\Circle$.  
There is a natural compactification of $\Mobius$ by $\Circle$.
In this larger space, a sequence $(\{x_n, y_n\})$ converges to a point $z \in \Circle$ if and only if $x_n \to z$ and $y_n \to z$.

We construct the \emph{upper and lower laminations} $\Lambda^\calV$ and $\Lambda_\calV$ as subsets of $\Mobius$~\cite[Theorem~8.1]{FSS22}. 
These are both invariant under the action of $\pi_1(M)$.  
See \reffig{Laminations} for examples of leaves, each represented by an arc in the disc.  
This drawing style leads us to the following terminology;
if $\lambda = \{x, y\}$ then we call $x$ and $y$ the \emph{ideal points} of $\lambda$.
Every leaf of both laminations is either a \emph{boundary} or \emph{interior leaf}~\cite[Definitions~8.13 and~8.14]{FSS22}.  
Every boundary leaf is associated to a unique cusp; 
the collection of boundary leaves, in each lamination, is dense.

Suppose that $c \in \Delta_\calV$ is a cusp.  
The \emph{upper crown}~\cite[Definition~9.5]{FSS22} of the cusp $c$, denoted $\Lambda^c \subset \Lambda^\calV$, is the collection of upper boundary leaves associated to $c$.  
We define the \emph{lower crown} $\Lambda_c\subset \Lambda_\calV$ similarly.  
We take the \emph{ideal points} of a crown to be the ideal points of its leaves, together with its cusp.
We refer to the other ideal points of a crown as its \emph{thorns}.

We now collect together (in order) 
Lemmas~7.15(1),
8.17(5), 
8.17(6),
8.19, 
9.2, 
and~9.9 
from~\cite{FSS22}. 

\begin{figure}[htbp]
\includegraphics[width=0.45\textwidth]{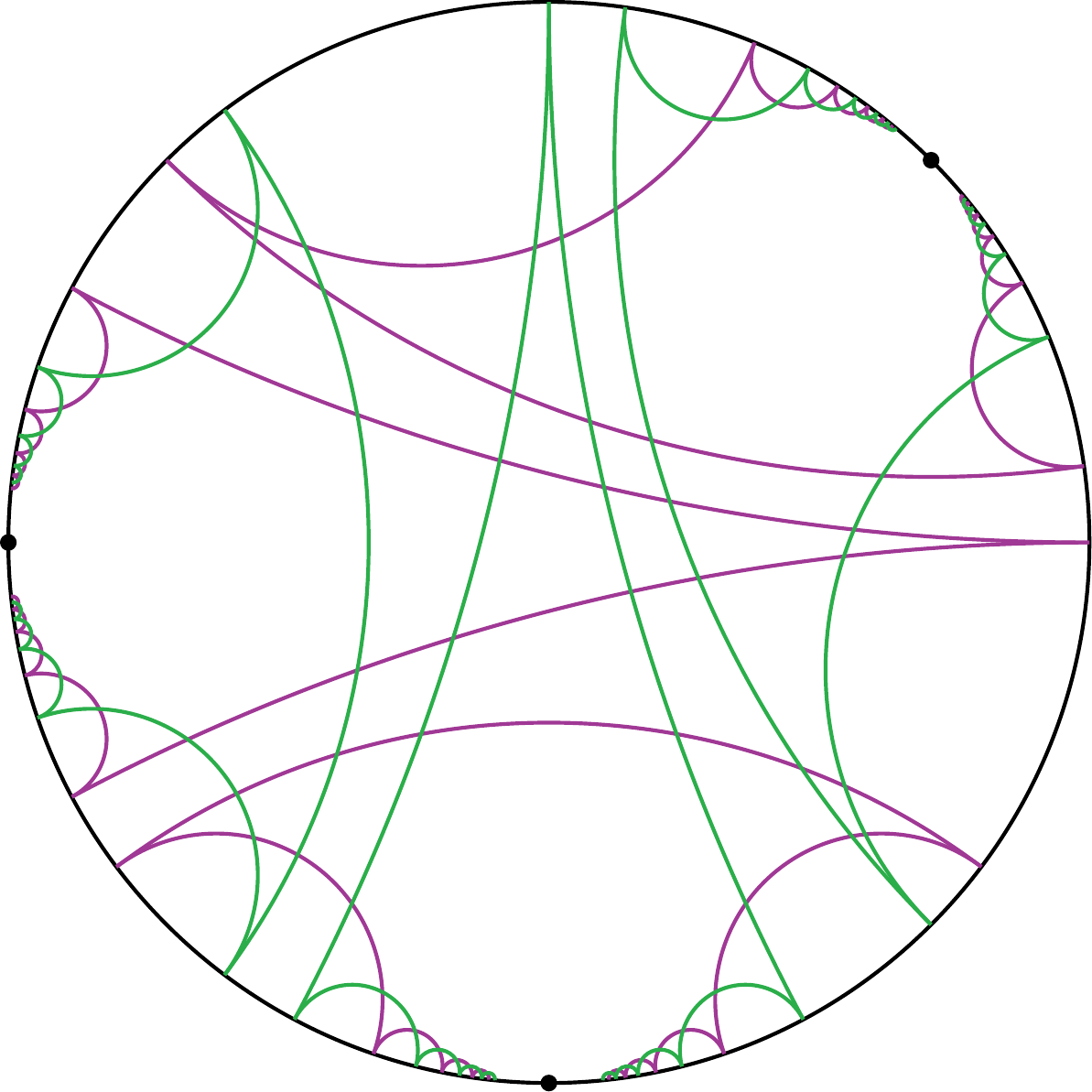}
\caption{A bit of the upper (green) and lower (purple) laminations. 
Here we show the six crowns associated to three cusps.}
\label{Fig:Laminations}
\end{figure}

\begin{lemma}
\label{Lem:Laminations}
\leavevmode
\begin{enumerate}
\item
\label{Itm:CrownZZ}
The leaves in a crown associated to a cusp $c$ accumulate on $c$ from both sides.
Also, $c$ is the unique accumulation point of any sequence of distinct leaves in the crown.
\item
\label{Itm:NoCusps}
No leaf of either lamination has an ideal point at any cusp. 
\item
\label{Itm:Adjacent}
Distinct leaves $\lambda$ and $\lambda'$ have a common ideal point $x$ if and only if 
they are adjacent boundary leaves of a common crown.
\item
\label{Itm:Approach}
Every leaf of either lamination is the limit of both boundary and interior leaves of that lamination.  
Interior leaves are limits from both sides while boundary leaves are only approached from the side opposite their associated cusp.
\item
\label{Itm:NoEndpoints}
No leaf $\lambda \in \Lambda^\calV$ shares an ideal point with any leaf $\mu \in \Lambda_\calV$. 
\item
\label{Itm:Interleave}
For any cusp $c$, the crowns $\Lambda^c$ and $\Lambda_c$ \emph{interleave}:
any pair of (non-cusp) ideal points of one crown is linked by some pair of (non-cusp) ideal points from the other.  \qed
\end{enumerate}
\end{lemma}

\section{The veering two-sphere}
\label{Sec:Sphere}

The goal of this section is to build the veering two-sphere $\Sphere$ as a quotient of the veering circle $\Circle$.  

\begin{definition}
\label{Def:DecompositionElementsCircle}
A \emph{decomposition element} in $\Circle$ is one of the following subsets of $\Circle$. 
\begin{enumerate}
\item
Suppose that $c \in \Delta_\calV$ is a cusp.  
Then $c$, together with the ideal points of $\Lambda^c$ and the ideal points of $\Lambda_c$ (its thorns) is a decomposition element.
We refer to this decomposition element as the \emph{clamshell boundary} associated to $c$.
\item
Suppose that $\lambda$ is an interior leaf of $\Lambda^\calV$ or of $\Lambda_\calV$.  
Then the pair of points $\lambda$ is a decomposition element.
\item  
Suppose that $x \in \Circle$ is not contained in one of the above.
Then the singleton set $\{x\}$ is a decomposition element.
\end{enumerate}
The decomposition element containing $x$ is denoted $[x]_{S^1}$. 
\end{definition}

We justify the name ``decomposition element'' with the following corollary of \reflem{Laminations}.

\begin{corollary}
\label{Cor:CircleDecomposition}
The decomposition elements partition $\Circle$. \qed
\end{corollary}

\begin{remark}
We abuse notation slightly and refer to $x \in \Circle$ as a \emph{singleton} if it is neither a cusp nor an ideal point of a leaf.
\end{remark}

\begin{definition}
\label{Def:VeeringSphere}
We denote the quotient of $\Circle$ by the decomposition elements as $\Sphere$.
We denote the quotient map by $\Psi_\calV \from \Circle \to \Sphere$.
We call $\Sphere$ the \emph{veering two-sphere}.   
\end{definition}

Since the decomposition elements are defined in terms of $\Lambda^\calV$ and $\Lambda_\calV$, and since these are invariant under the action of $\pi_1(M)$, the action descends to $\Sphere$.  
The main result of this section is the following.

\begin{theorem}
\label{Thm:Sphere}
The veering two-sphere $\Sphere$ is a two-sphere.  
Furthermore, the action of $\pi_1(M)$ on $\Sphere$ is continuous, faithful, and orientation preserving. 
Finally, if $\calV$ is finite then the action has dense orbits.
\end{theorem}

The proof is closely modelled on the discussion around~\cite[Theorem~5.7]{Thurston82}.  
Note that we have additional subtleties to deal with;
the complementary regions to our laminations have a countable infinity of sides. 
We proceed as follows.

\begin{definition}
\label{Def:PreSphere}
Let $S^2$ be the unit sphere in $\RR^3$;
let $Z = \{ z = 0 \}$ be the horizontal plane.
We identify $\Circle$ with the intersection $S^2 \cap Z$.
The orientation of $\Circle$ is anticlockwise as viewed from above.
Let $D^\calV$ and $D_\calV$ be the (closed) upper and lower hemispheres of $S^2$.
Using the hemisphere model of $\HH^2$ we identify their interiors with $\HH^2$.
We call this copy of $S^2$ the \emph{pre-sphere}.
\end{definition}

Suppose that $\lambda \in \Lambda^\calV$ is a leaf.
Let $[\lambda]^D \subset D^\calV$ be the intersection of $D^\calV$ with the vertical plane in $\RR^3$ containing the ideal points of $\lambda$.
Note that $[\lambda]^D$ is a closed arc embedded in the upper hemisphere of the pre-sphere.
We define 
\[
[\Lambda^\calV]^D = \bigcup_{\lambda \in \Lambda^\calV} [\lambda]^D 
\qquad
\mbox{and}
\qquad
[\Lambda^c]^D = \bigcup_{\lambda \in \Lambda^c} [\lambda]^D
\]
We use similar notation, $[\lambda]_D$, for the lower lamination $\Lambda_\calV$ extended into the lower hemisphere $D_\calV$.

Fix a cusp $c \in \Delta_\calV$.
We define $D^c$ to be the closure of the component of $D^\calV -  [\Lambda^c]^D$ that is disjoint from $[\Lambda^\calV]^D$.
By \reflem{Laminations}\refitm{CrownZZ} and~\refitm{Adjacent} we find that $D^c$ is the closure of an ideal polygon with countably many sides, meeting $\Circle$ in countably many points, whose only accumulation point in $S^2$ is at the cusp $c \in \Circle$.  
We define $D_c$ similarly.
By \reflem{Laminations}\refitm{NoEndpoints}, we have that $D^c \cap D_c = \{c\}$.  
Thus $D^c \cup D_c$ is the one-point union of two closed discs in the pre-sphere. 
See \reffig{DecompositionElementsOnPreSphere}.

\begin{figure}[htbp]
\includegraphics[width=0.65\textwidth]{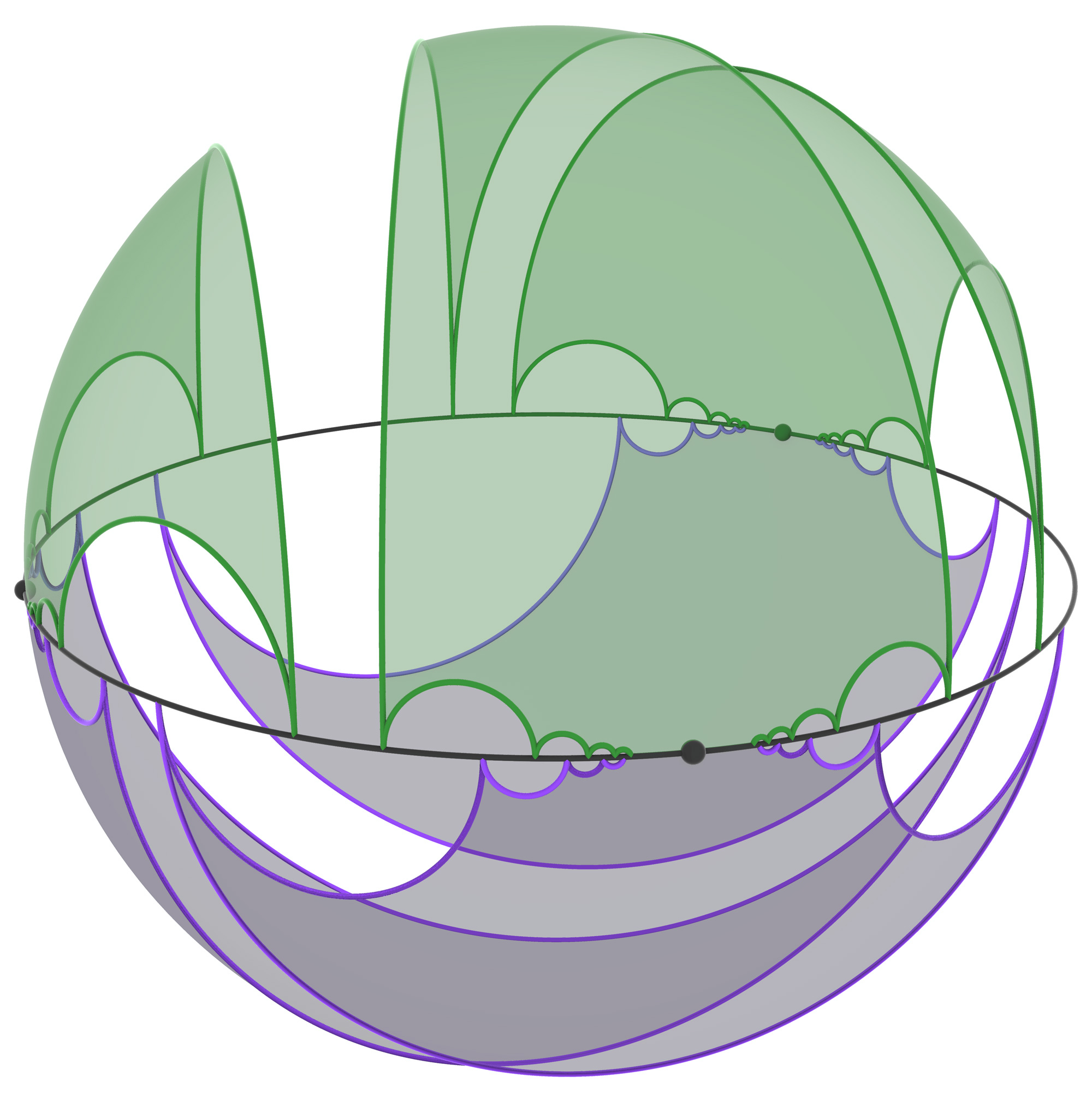}
\caption{The three clamshells in the pre-sphere associated to three cusps. 
Compare with \reffig{Laminations}.}
\label{Fig:DecompositionElementsOnPreSphere}
\end{figure}

We now define decomposition elements in the pre-sphere.
\begin{definition}
\label{Def:DecompositionElementsSphere}
Suppose that $x$ is a point of $S^2$.
We construct its decomposition element, $[x]_{S^2}$, as follows. 
\begin{enumerate}
\item 
If $x$ lies in $D^c$ or $D_c$, for some cusp $c \in \Delta_\calV$, then we take $[x]_{S^2} = D^c \cup D_c$. 
\item
If $x$ lies in $[\lambda]^D$ for some interior leaf $\lambda$ of $\Lambda^\calV$, then we take $[x]_{S^2}  = [\lambda]^D$. 
\item
If $x$ lies in $[\lambda]_D$ for some interior leaf $\lambda$ of $\Lambda_\calV$, then we take $[x]_{S^2}  = [\lambda]_D$. 
\item
If $x$ lies in $\Circle$, and is not a cusp or a leaf ideal point, then we take $[x]_{S^2}  = \{x\}$. 
\end{enumerate}
We call these decomposition elements \emph{clamshells}, \emph{upper leaves}, \emph{lower leaves}, and \emph{singletons}, respectively.
\end{definition}

From Definitions~\ref{Def:DecompositionElementsCircle} and~\ref{Def:DecompositionElementsSphere} we deduce the following.

\begin{lemma}
\label{Lem:Slice}
For any point $x$ in the pre-sphere, we have that $[x]_{S^2}$ meets the veering circle $\Circle$.
Also, for a point $x \in \Circle$ we have $[x]_{S^2} \cap \Circle = [x]_{S^1}$.
This induces a bijection between the quotient of the pre-sphere and the veering two-sphere $\Sphere$. \qed
\end{lemma}

\begin{lemma}
\label{Lem:Equiv}
The decomposition elements of \refdef{DecompositionElementsSphere} form a closed equivalence relation on the pre-sphere.
\end{lemma}

\begin{proof}
We show that the decomposition elements cover the pre-sphere.
Suppose that $x$ is a point of the pre-sphere.
By \refcor{CircleDecomposition}, if $x$ lies in $\Circle$ then $x$ is either a cusp, a point in the boundary of a leaf (of $\Lambda^\calV$ or of $\Lambda_\calV$), or a singleton.
In all cases, $[x]_{S^2}$ is a decomposition element.
Suppose instead, breaking symmetry, that $x$ lies in the interior of $D^\calV$.
If $x$ lies in $[\lambda]^D$ for some $\lambda \in \Lambda^\calV$ then we are done.
Suppose not.
Let $E$ be the closure of the component of $D^\calV - [\Lambda^\calV]^D$ containing $x$.
By \reflem{Laminations}\refitm{Approach} the disc $E$ meets only boundary leaves of $[\Lambda^\calV]^D$ and these boundary leaves are all associated to a single cusp, say $c$.
We deduce that $E = D^c$ and thus $x \in [c]_{S^2}$.
Thus the decomposition elements cover the pre-sphere.

We now show that decomposition elements partition the pre-sphere.
Fix $x$ and $y$ in the pre-sphere.  
Suppose that $[x]_{S^2} \cap [y]_{S^2}$ is non-empty.  
Since all decomposition elements meet $\Circle$, we may assume that the representatives $x$ and $y$ lie in $\Circle$.  
Applying \refcor{CircleDecomposition} and \reflem{Slice}, we conclude that the decomposition elements partition the pre-sphere. 

We now show that the resulting equivalence relation is closed.
Suppose that $(x_n, y_n)$ is a sequence in $S^2 \times S^2$ so that $[x_n]_{S^2} = [y_n]_{S^2}$ for all $n$.
Suppose that the sequence converges to $(x_\infty, y_\infty)$.
We must show that $[x_\infty]_{S^2} = [y_\infty]_{S^2}$.
Taking subsequences if necessary, we may assume that all decomposition elements $[x_n]_{S^2}$ are of the same type. 

Suppose that $[x_n]_{S^2}$ is a singleton for all $n$.
Thus $x_n = y_n$ for all $n$.
Thus $x_\infty = y_\infty$.
Therefore $[x_\infty]_{S^2} = [y_\infty]_{S^2}$.

Suppose now that $[x_n]_{S^2}$ is a leaf for all $n$.
Passing to a further subsequence and breaking symmetry if needed, we may assume that $[\lambda_n]^D =  [x_n]_{S^2}$ is a leaf of the lamination $[\Lambda^\calV]^D$ in the upper hemisphere $D^\calV$.
Passing to a further subsequence, we may assume that the $[\lambda_n]^D$ Hausdorff converge to either a single point $z_\infty$ in $\Circle$ or a leaf $[\lambda_\infty]^D$ of $[\Lambda^\calV]^D$.
In the first case $x_\infty = z_\infty = y_\infty$ and we are done.
In the second case, $[\lambda_\infty]^D$ contains both $x_\infty$ and $y_\infty$.
Since $[\lambda_\infty]^D \subset [\lambda_\infty]_{S^2}$ we are done.

Finally, suppose that $[x_n]_{S^2} = C_n$ is a clamshell for all $n$.
Pass to a subsequence so that the $C_n$ Hausdorff converge to $L$. 

We write $D^n = C_n \cap D^\calV$ and $D_n = C_n \cap D_\calV$. 
Each of $D^n$ and $D_n$ is an infinity-gon.
These Hausdorff converge to $L^\infty$ and $L_\infty$, respectively.

\begin{claim*}
Each of $L^\infty$ and $L_\infty$ is either a point of $\Circle$, a leaf of $[\Lambda^\calV]^D$ or $[\Lambda_\calV]_D$, or an infinity-gon.
\end{claim*}

\begin{proof}
It suffices to consider $L^\infty$.
Note that $L^\infty$ is convex in the upper hemisphere metric on $D^\calV$ since the approximating infinity-gons $D^n$ are.
If $L^\infty$ has interior then the $D^n$ eventually contain a small disc of the interior, hence they meet each other, hence the sequence is eventually constant.
In this case $L^\infty$ is an infinity-gon.

Set $\bdyi D^n = D^n \cap \Circle$.
Note that $\bdyi L^\infty$ is the Hausdorff limit of $\bdyi D^n$.
If $\bdyi L^\infty$ is a single point then by convexity, $L^\infty$ is a single point and we are done.
If $\bdyi L^\infty$ contains three or more points then by convexity $L^\infty$ has interior and we are in the previous case.
So suppose that $\bdyi L^\infty$ consists of exactly two points, $u$ and $v$ in $\Circle$.
Thus $L^\infty$ is a geodesic.
For large $n$, the vertices of $D^n$ are partitioned by being close to either $u$ or $v$.
Thus there are exactly two boundary leaves, of $D^n$, each of which has its ideal points close to $u$ and $v$.
Thus $L^\infty$ is a Hausdorff limit of leaves, and so is itself a leaf.
\end{proof}

Note that the $C_n$ are connected.
Thus, $L = L^\infty \cup L_\infty$ is also connected.
We now consider cases.
\begin{itemize}
\item
Suppose that both $L^\infty$ and $L_\infty$ are points. 
Then $L$ is a point and we are done.

\item
Suppose that neither $L^\infty$ nor $L_\infty$ is an infinity-gon and at least one is a leaf.
Then the other is not a leaf by \reflem{Laminations}\refitm{NoEndpoints}.
Thus $L$ is a leaf of either $[\Lambda^\calV]^D$ or $[\Lambda_\calV]_D$, so is contained in a decomposition element, and we are done.

\item
Finally, breaking symmetry, suppose that $L^\infty$ is an infinity-gon.
Since infinity-gons have non-empty interiors, we deduce that the sequence $D^n$ is eventually constant.
Thus the sequence $D_n$ is eventually constant.
Thus $L$ is a clamshell and we are done. 
\end{itemize}
This completes the proof of \reflem{Equiv}.
\end{proof}

\begin{lemma} 
\label{Lem:Homeo}
The bijection of \reflem{Slice} is a homeomorphism.
\end{lemma}

\begin{proof}
We prove that both directions of the bijection are closed maps.

Suppose we have a closed subset $A$ of the quotient of the pre-sphere.
The preimage $B$ of $A$ in the pre-sphere is closed and saturated in the pre-sphere.
Therefore $C$, the intersection of $B$ with the circle, is closed and saturated in the circle.
Therefore $D$, the image of $C$ in the quotient of the circle, is closed.

Suppose that we have a closed subset $D$ of the quotient of the circle, $\Sphere$.
Thus $C$, the preimage of $D$ in the circle, is closed and saturated in the circle.
Extend $C$ by taking its union with all leaves and clamshells it meets. 
Call the resulting extension $B$; this is saturated by construction.
We now prove that $B$ is closed.
Suppose that $x_n \in B$ is a sequence with limit $x_\infty$ in the pre-sphere.
For each $n$, applying \reflem{Slice} we choose $c_n \in [x_n]_{S^2} \cap \Circle$.
Since $\Circle$ is compact, we pass to a subsequence so that $c_n \to c_\infty$.
Note that $c_\infty$ lies in $C$.
So $[c_\infty]_{S^2}$ lies in the extension $B$.
Note that $[x_n]_{S^2} = [c_n]_{S^2}$. 
By \reflem{Equiv}, the equivalence relation on the pre-sphere is closed.
Thus it is sequentially closed. 
Thus $[x_\infty]_{S^2} = [c_\infty]_{S^2}$.
Thus $x_\infty$ lies in $B$, so $B$ is closed.
Thus $A$, the image of $B$ in the quotient of the pre-sphere, is closed.
\end{proof}

\begin{lemma}
\label{Lem:ConnNonSep}
For each $x$ in the pre-sphere, the decomposition element $[x]_{S^2}$ is connected and non-separating. 
\end{lemma}

\begin{proof}
This is because $[x]_{S^2}$ is either a single point, an embedded arc, or a one-point union of a pair of embedded discs. 
\end{proof}

\begin{proof}[Proof of \refthm{Sphere}]
Lemmas~\ref{Lem:Equiv} and~\ref{Lem:ConnNonSep} verify the hypotheses of \refthm{Moore}.  
Thus the quotient of the pre-sphere is homeomorphic to a two-sphere.  
By \reflem{Homeo}, we have that $\Sphere$ is a two-sphere.

We now analyse the action of $\pi_1(M)$ on the veering two-sphere $\Sphere$.  
By~\cite[Theorem~7.1]{FSS22}, the action of $\pi_1(M)$ on $\Circle$ has all of the properties listed in the conclusions of \refthm{Sphere}.
This implies that the action on $\Sphere$ is continuous and, when $\calV$ is finite, has dense orbits.
The action of $\pi_1(M)$ on the cusps $\Delta_\calV$ is faithful;
these embed in $\Sphere$.
Thus the action on $\Sphere$ is faithful.

We now show that the action on $\Sphere$ is orientation preserving.
Let $d^\calV$ be the interior of the upper hemisphere $D^\calV$, in the pre-sphere.
The image of $d^\calV$ in $\Sphere$ is a countably branching dendrite. 
The action of $\pi_1(M)$ preserves the dendrite and the circular order of the arcs emanating from any branch point.
Thus the action is orientation preserving.
\end{proof}

We need the following in the proof of \refthm{Equivalence}.

\begin{lemma}
\label{Lem:DecompositionElementsRecover}
The collection of infinite decomposition elements in $\Circle$ recovers the upper and lower laminations $\Lambda^\calV$ and $\Lambda_\calV$, up to swapping the two laminations.
\end{lemma}

\begin{proof}
Suppose that $[c]_{S^1} \in \Circle$ is an infinite decomposition element.
By \refcor{CircleDecomposition}, we deduce that there is exactly one element $c \in [c]_{S^1}$ which is an accumulation point; 
this is the cusp.
Up to choosing a basepoint, this gives a unique order isomorphism between $[c]_{S^1} - \{c\}$ and $\ZZ$.
Using this isomorphism, pairs of points differing by two are the ideal points of boundary leaves of the upper and lower laminations $\Lambda^\calV$ and $\Lambda_\calV$.

However, leaves that link belong to opposite laminations.
Since the link space is a plane, under the relation of linking, the boundary leaves form a connected bipartite graph.
This allows us to separate the boundary leaves of $\Lambda^\calV$ from the boundary leaves of $\Lambda_\calV$.
Since they are the closures of their boundary leaves (see \refsec{Dynamics}), the upper and lower laminations $\Lambda^\calV$ and $\Lambda_\calV$ are thus determined. 
However, we do not learn which lamination is which.
\end{proof}

\section{Loom and link spaces}
\label{Sec:Loom}

The goal of this section is to collect various definitions and notions regarding of \emph{loom} and \emph{link} spaces~\cites{SchleimerSegerman24, FSS22}.
These are closely modelled on the \emph{leaf} (or \emph{orbit}) spaces arising from pseudo-Anosov flows (without perfect fits).
For an introduction, see~\cites{Calegari07}{BarthelmeMann26}.

\subsection{Loom spaces}
\label{Sec:Loom}

Suppose that $\calL$ is a copy of the plane $\RR^2$ equipped with a pair of non-singular transverse foliations $F^\calL$ and $F_\calL$.
We reproduce the following definitions from~\cite[Section~2]{SchleimerSegerman24}.

\begin{definition}
\label{Def:Rectangle}
Suppose that $f_R \from (0, 1)^2 \to \calL$ is an embedding with image $R$.  
We call $R$ a \emph{rectangle} if $f_R$ sends arcs parallel to the $y$--axis ($x$--axis) to arcs of $F^\calL$ (of $F_\calL$).  
\end{definition}

\begin{figure}[htbp]
\vspace{-10pt}
\subfloat[Cusp rectangle]{
\label{Fig:CuspRect}
\includegraphics[width = 0.4\textwidth]{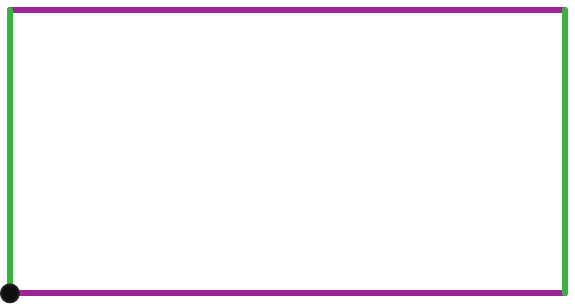}
}
\subfloat[Edge rectangle]{
\label{Fig:EdgeRect}
\includegraphics[width = 0.4\textwidth]{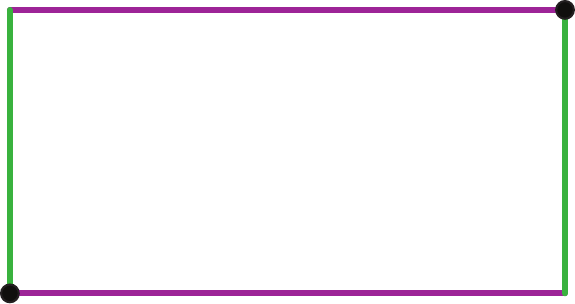}
}

\subfloat[Face rectangle]{
\label{Fig:FaceRect}
\includegraphics[width = 0.4\textwidth]{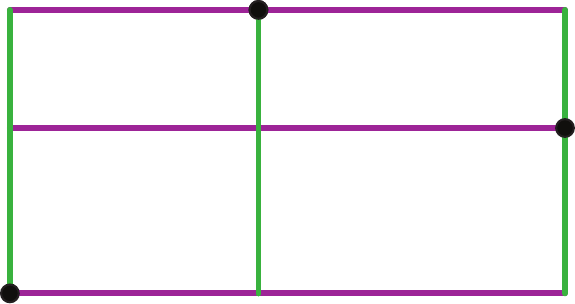}
}
\subfloat[Tetrahedron rectangle]{
\label{Fig:TetRect}
\includegraphics[width = 0.4\textwidth]{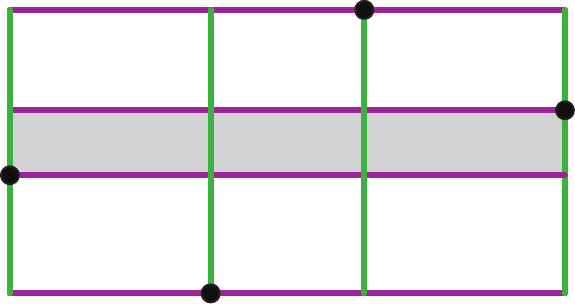}
}
\caption{Examples of cusp, edge, face, and tetrahedron rectangles.
When $\calL = \Link$ is a link space, \refcor{OnlyAsymptotics} implies that the points of $\bdy R$ contained in $\Circle$ are in fact cusps of $\Delta_\calV$.}
\label{Fig:SkeletalRects}
\end{figure}

\begin{definition}
\label{Def:CuspRect}
A rectangle $R$ is a \emph{south-west cusp rectangle} if there is a continuous extension of $f_R$ to a homeomorphism 
\[
\closure{f}_R \from [0,1]^2 - \{(0,0)\} \to \closure{R}
\]
See \reffig{CuspRect}.
We define cusp rectangles with the other intercardinal directions similarly.
The two sides (see~\cite[Definition~2.6]{SchleimerSegerman24}) of a cusp rectangle which are not closed intervals are called \emph{cusp sides}.
\end{definition}

\begin{definition}
\label{Def:TetRect}
A rectangle $R$ is a \emph{tetrahedron rectangle} if there are $a, b, c, d \in (0,1)$ and a continuous extension of $f_R$ to a homeomorphism 
\[
\closure{f}_R \from [0,1]^2 - \{(a,0), (1,b), (c,1), (0,d)\} \to \closure{R} \qedhere
\]
\end{definition}

See \reffig{TetRect}. 
Edge and face rectangles are defined similarly, see~\cite[Definitions~2.28 and~2.29]{SchleimerSegerman24} and Figures~\ref{Fig:EdgeRect} and~\ref{Fig:FaceRect}.
We collectively call the edge, face, and tetrahedron rectangles \emph{skeletal rectangles}.

\begin{definition}
\label{Def:Loom}
We say that $\calL$ is a \emph{loom space} if its rectangles satisfy the following axioms.
\begin{enumerate}
\item 
\label{Itm:Cusp}
For every cusp side $s$ of every cusp rectangle, some initial open interval of $s$ is contained in some rectangle.
(See \reffig{CuspSide}.)
\item 
\label{Itm:Tet}
Every rectangle is contained in some tetrahedron rectangle. \qedhere
\end{enumerate}
\end{definition}

\begin{wrapfigure}[10]{r}{0.29\textwidth}
\centering
\vspace{-3 pt}
\includegraphics[width = 0.22\textwidth]{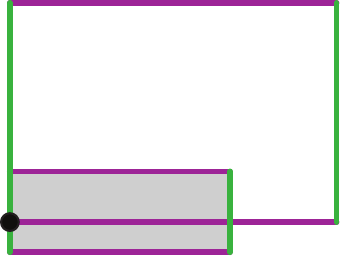}
\caption{A south-west cusp rectangle. An initial segment of its southern cusp side is contained in the shaded rectangle.}
\label{Fig:CuspSide}
\end{wrapfigure}

We now reproduce~\cite[Definitions~3.3 and~3.8]{SchleimerSegerman24}.

\begin{definition}
\label{Def:Cusp}
A \emph{cusp} $c$ of a loom space $\calL$ is an equivalence class of the transitive closure of the following relation.
Two cusp rectangles of $\calL$ are \emph{adjacent} if some cusp side of one is contained in some cusp side of the other.
\end{definition}

\begin{definition}
\label{Def:CuspNonCuspLeaves}
A leaf of either foliation $F^\calL$ or $F_\calL$ is a \emph{cusp leaf} if it contains some cusp side of some cusp rectangle.
We say that the leaf \emph{belongs to} the cusp containing the cusp rectangle.
\end{definition}

We will need the following simple lemma.

\begin{lemma}
\label{Lem:NoTriangles}
Suppose that $\ell$ and $\ell'$ are distinct cusp leaves of the same foliation, belonging to a common cusp $c$.
Then any leaf $m$ of the opposite foliation meets at most one of $\ell$ and $\ell'$.
\end{lemma}

\begin{proof}
Since $\ell$ and $\ell'$ belong to a common cusp there is a chain of cusp rectangles with $\ell$ containing a cusp side of the first and $\ell'$ containing a cusp side of the last.
Since $\ell$ and $\ell'$ are distinct and in the same foliation, this implies that there is a cusp leaf $m'$ of the opposite foliation which
\begin{itemize}
\item
contains a cusp side of one of the cusp rectangles and
\item
separates $\ell$ from $\ell'$ in the plane $\calL$.
\end{itemize}
The leaf $m$ cannot cross $m'$ since they belong to the same foliation.
\end{proof}

Each side of a tetrahedron rectangle lies in the union of a pair of cusp leaves belonging to a given cusp.
We formalise such pairs as follows.

\begin{definition}
\label{Def:BiLeaf}
Suppose that $\ell$ and $\ell'$ are cusp leaves of the same foliation.  
Suppose that $c$ is the cusp of both $\ell$ and $\ell'$.  
Suppose further that $\ell$ and $\ell'$ are separated (in $\Link$) by exactly one cusp leaf $m$ in the other foliation.  
Note that $m$ necessarily also has its cusp at $c$.  
We call the union 
\[
k = \ell \cup \{c\} \cup \ell'
\]
a \emph{bi-leaf}.
In an abuse of notation, we say that $k$ \emph{belongs to} the foliation that contains $\ell$ and $\ell'$. 
We call $c$ the \emph{centre} of $k$.
We call the leaf $m$ the \emph{divider} of $k$.
\end{definition}

See \reffig{BiLeaf}.

\begin{figure}[htbp]
\subfloat[Rectangular style.]{
\labellist
\small\hair 2pt
\pinlabel $c$ [r] at 0 137
\pinlabel $\ell$ [l] at 6 210
\pinlabel $m$ [t] at 100 134
\pinlabel $\ell'$ [l] at 6 60
\endlabellist
\includegraphics[height=2.5in]{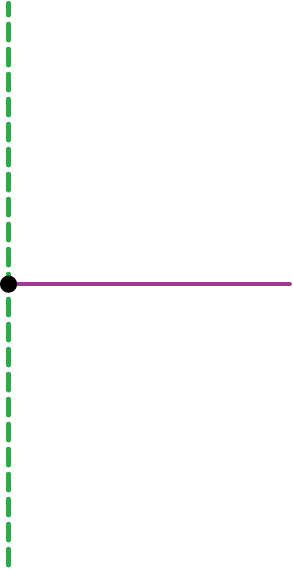}
}
\qquad
\subfloat[Hyperbolic style.]{
\labellist
\small\hair 2pt
\pinlabel $c$ [r] at 10 170
\pinlabel $\ell$ [tl] at 110 200
\pinlabel $m$ [t] at 120 148
\pinlabel $\ell'$ [bl] at 125 67
\endlabellist
\includegraphics[height=2.5in]{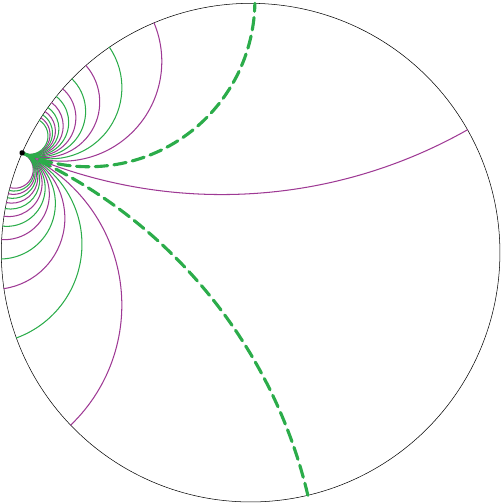}
}
\caption{A bi-leaf centred at $c$; it is drawn with a thick dashed line.} 
\label{Fig:BiLeaf}
\end{figure}
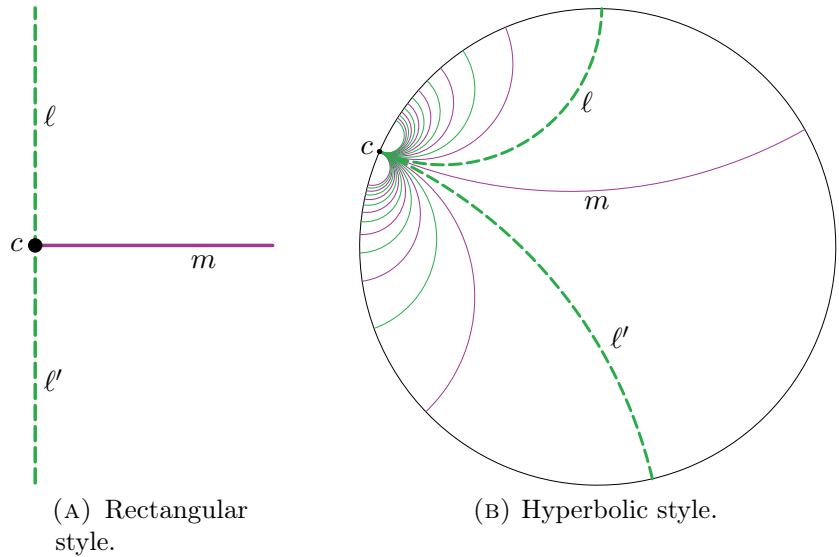

\subsection{The astroid lemma}
\label{Sec:Astroid}

One of the key properties of loom spaces is recorded in the \emph{astroid lemma}~\cite[Lemma~4.10]{SchleimerSegerman24}. 
For details, see that paper. 
Here we only give a brief description.

Suppose that $p$ is a point of the loom space $\calL$.
Let $\ell$ and $m$ be the leaves of $\calL$ through $p$. 
(If $\ell$ or $m$ is a cusp leaf, we extend it to the appropriate bi-leaf, as in \refdef{BiLeaf}.)
The union of all closures of rectangles with $p$ as their south-west corner is the \emph{north-east staircase for $p$}.
The three other intercardinal directions give the three other staircases for $p$.
We now consider all cusp leaves, emanating from cusps in the boundary of the staircases. 
The intersection of these cusp leaves with $\ell$ and $m$ gives a countable set of points. 
The astroid lemma tells us that these intersection points accumulate at $p$ and also exit all ends of $\ell$ and $m$. 
\reffig{Astroid} shows the four staircases about a point on an intersection of non-cusp leaves.

\begin{figure}[htbp]
\includegraphics[width=\textwidth]{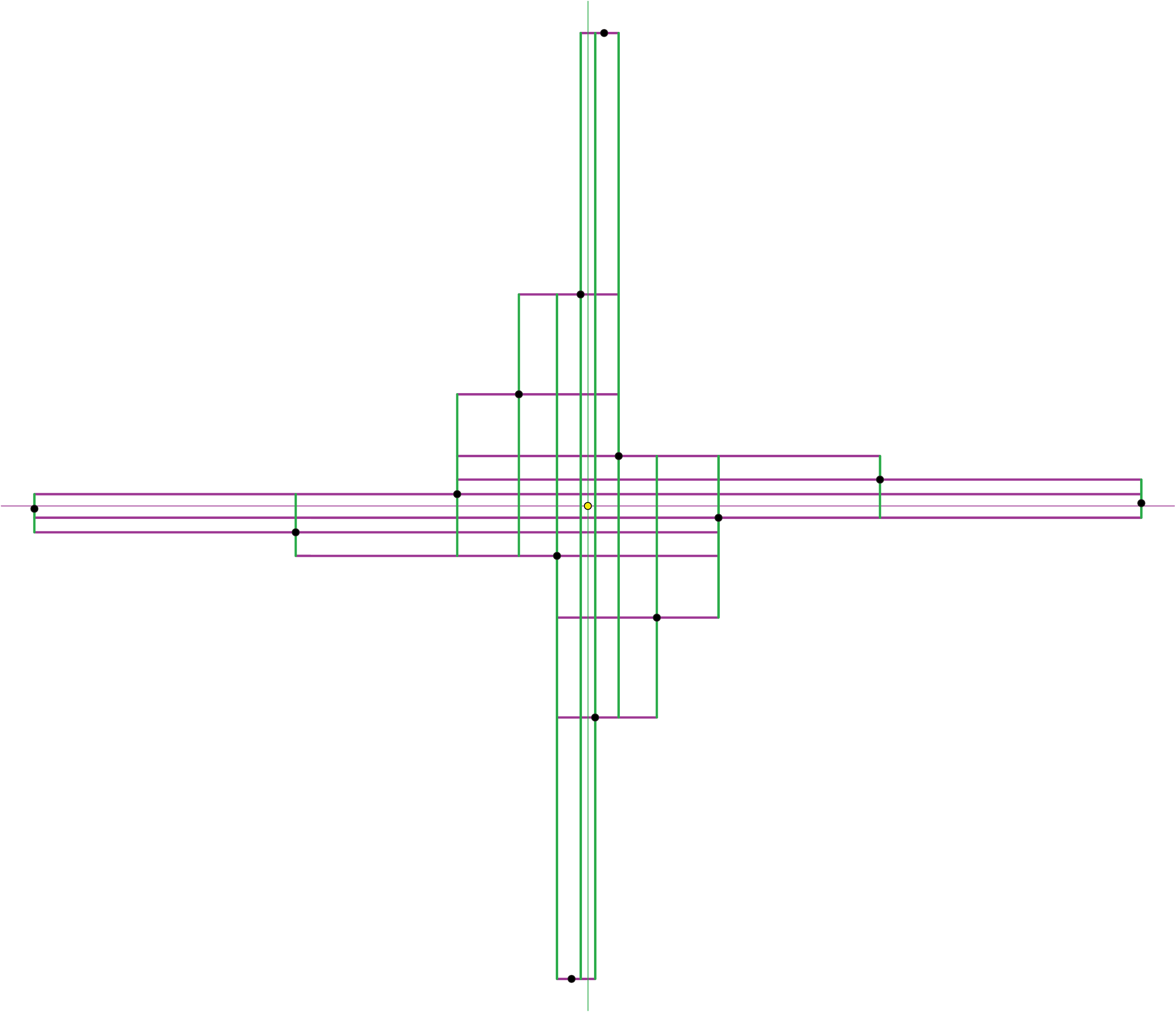}
\caption{The four staircases about a Weierstrass point (centre) for the fibre of the monodromy of the figure eight knot complement.
}
\label{Fig:Astroid}
\end{figure}

\subsection{The link space}
\label{Sec:Link}

The \emph{link space} $\Link$ is constructed in~\cite[Section~10]{FSS22} by collapsing a certain subset of the product of laminations $\Lambda^\calV \times \Lambda_\calV$.  
In the process, interior leaves of $\Lambda^\calV$ and $\Lambda_\calV$ are taken to \emph{non-cusp leaves} of $F^\calV$ and $F_\calV$, respectively.
The remaining leaves of $F^\calV$ and $F_\calV$ are \emph{cusp leaves} in the sense of~\cite[Definition~10.7]{FSS22}.

\begin{remark}
\label{Rem:NonCuspLeavesDense}
There are countably many cusps and countably many cusp leaves emanating from each.
(Cusp leaves come from boundary leaves, which come from branch lines; see~\cite[Lemma~6.10(2) and Definition~8.14]{FSS22}.)
So there are countably many cusp leaves in total.
In particular, non-cusp leaves are dense.
\end{remark}

The following is a combination of~\cite[Theorems~10.8 and~10.51]{FSS22}.

\begin{theorem}
\label{Thm:LinkIsLoom}
Suppose that $\calV$ is a veering triangulation of a three-manifold $M$.  
Then the link space $\Link$, equipped with the foliations $F^\calV$ and $F_\calV$, is a loom space.
A leaf is a cusp leaf (in the sense of the link space) if and only if it is a cusp leaf of the loom space structure.
Furthermore, $\Link$ has a canonical orientation.
The fundamental group $\pi_1(M)$ acts on $\Link$ by orientation-preserving homeomorphisms preserving the loom space structure. \qed
\end{theorem}

\begin{definition}
\label{Def:IdealPoints}
Suppose that $\ell$ is a leaf of $F^\calV$.
If $\ell$ is the non-cusp leaf associated to an interior leaf $\lambda$ in $\Lambda^\calV$ then the \emph{ideal endpoints} of $\ell$ are the (ideal) points of $\lambda$.
If $\ell$ is the cusp leaf associated to adjacent boundary leaves $\lambda$ and $\lambda'$ of the crown $\Lambda^c$ then the \emph{ideal endpoints} of $\ell$ are the cusp $c$ and the common ideal point of $\lambda$ and $\lambda'$.
In this case we say that $\ell$ \emph{emanates} from $c$.
We make similar definitions when $\ell$ is a leaf of $F_\calV$.
See \reffig{HypRectStyles}.
\end{definition}

\reflem{Laminations}\ref{Itm:Adjacent} and~\ref{Itm:NoEndpoints}, together with \refthm{LinkIsLoom}, imply the following.

\begin{corollary}
\label{Cor:OnlyAsymptotics}
Suppose that $\ell$ and $m$ are distinct leaves in $\Link$ (in either of the foliations).
Suppose that $\ell$ and $m$ share a common ideal point.
Then $\ell$ and $m$ are both cusp leaves and the shared point is a cusp.
\qed
\end{corollary}

\begin{figure}[htbp]
\subfloat[Rectangular style.]{
\label{Fig:RectStyle}
\includegraphics[width=0.4\textwidth]{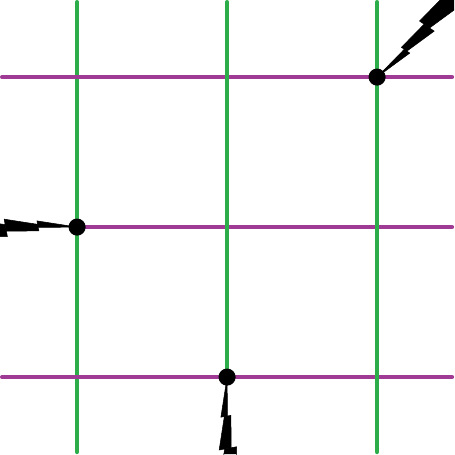}
}
\qquad
\subfloat[Hyperbolic style.]{
\label{Fig:HypStyle}
\includegraphics[width=0.4\textwidth]{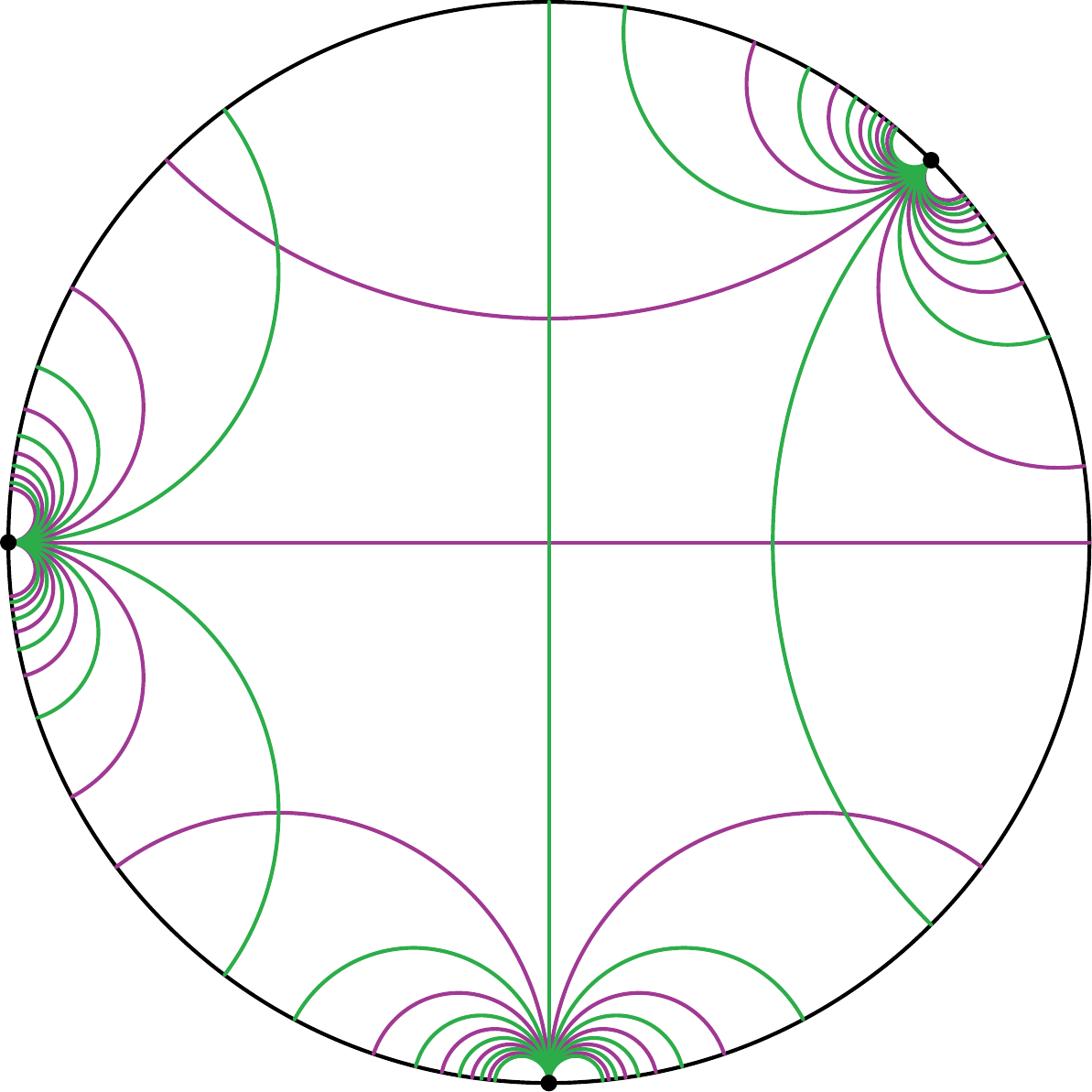}
}

\caption{Some of the cusp leaves of the foliations $F^\calV$ (in green) and $F_\calV$ (in purple), drawn in the \emph{rectangular} and \emph{hyperbolic} styles.
By \refcor{OnlyAsymptotics}, the black dots are cusps of $\Delta_\calV$ (which are dense in $\Circle$).
In \refthm{Disc} we show that $\Circle$ equivariantly compactifies $\Link$ giving a topological disc $\Disc$.
When drawing local pictures it is often useful to use the rectangular style of \reffig{RectStyle};
here the foliations $F^\calV$ and $F_\calV$ are vertical and horizontal, respectively.
The ``lightning bolts'' attached to each cusp indicate the location of the branch cut
(needed when drawing a point with infinite cone angle in the rectangular style).
}
\label{Fig:HypRectStyles}
\end{figure}





\noindent
In \refsec{Disc} we compactify the link space by the veering circle to obtain the veering disc. 
Doing this requires understanding various finite and infinite paths in $\Link$. 

\subsection{Polygonal paths}

We use \emph{arc}, \emph{loop}, \emph{ray}, and \emph{line} to denote (the image of) a proper embedding of the following spaces: 
the interval $[0, 1]$, the circle $S^1$, the half-line $\RR_{\geq 0}$, and the line $\RR$.  
Here is our version of~\cite[Definition~3.1]{Fenley12}.  

\begin{definition}
\label{Def:Polygonal}
A \emph{polygonal path} $\rho$ in $\Link$ is an oriented arc, loop, ray, or line with the following properties. 
\begin{itemize}
\item
The path $\rho$ is a finite union $(\rho_i)_{i = 0}^{n-1}$ of \emph{segments}: oriented closed connected subsets of leaves of $F^\calV$ and $F_\calV$.  
(When $\rho$ is a loop, all indices are taken modulo $n$.)
\item 
The segment $\rho_i$ lies in $F^\calV$ if and only if $\rho_{i+1}$ lies in $F_\calV$. 
\item
The terminal point of $\rho_i$ equals the initial point of $\rho_{i+1}$.
\end{itemize}
The intersection points $\rho_i \cap \rho_{i+1}$ are called the \emph{material corners} of $\rho$.  
\end{definition}

See \reffig{PolygonalPath}. There is an important special case: if $\rho$ is a polygonal line with only one segment, then $\rho$ is an oriented leaf of $F^\calV$ or of $F_\calV$. 

Note that \refdef{Polygonal} is more restrictive than the polygonal paths given in~\cite[Definition~6.3]{SchleimerSegerman24}; here we do not allow polygonal paths to pass through cusps.

\begin{figure}[htbp]
\includegraphics[width=0.35\textwidth]{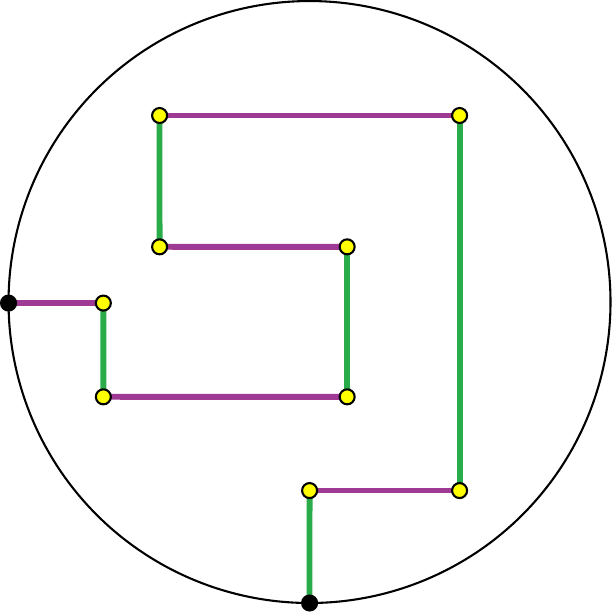}
\caption{An example of an embedded polygonal line $\rho$. 
Material corners are indicated with yellow dots, the points of $\bdyi \rho$ (not necessarily cusps) are indicated with black dots.}
\label{Fig:PolygonalPath}
\end{figure}

From \refthm{LinkIsLoom} we have the following observations. 

\begin{corollary}
\label{Cor:TransverseOmnibus}
\leavevmode
\begin{enumerate}
\item
\label{Itm:PreCompactBasis}
The precompact rectangles form a basis for the topology on $\Link$. 
\item
\label{Itm:PreCompactBoundary}
The boundary $\bdy R$ of a precompact rectangle $R$ is an embedded polygonal loop with exactly four segments. 
\item
\label{Itm:Ascending}
Every rectangle is an ascending union of precompact rectangles.
\item
\label{Itm:Proper}
Suppose that $\rho \subset \Link$ is an embedded polygonal line.  
Then $\rho$ is properly embedded in $\Link$.  
\item
\label{Itm:DiscUnion}
Suppose that $\rho$ is an embedded polygonal loop in $\Link$.  
Then the open disc bounded by $\rho$ is a finite union of (necessarily precompact) rectangles. \qed
\end{enumerate}
\end{corollary}


Next we generalise \refdef{IdealPoints} to polygonal paths.

\begin{definition}
\label{Def:Ideal}
Suppose that $\rho$ is a polygonal ray or a line in $\Link$.  
We define $\bdyi \rho$ to be the ideal points of those segments $\rho_i$ which are non-compact.  
There are at most two of these. 
See \reffig{PolygonalPath}.
\end{definition}

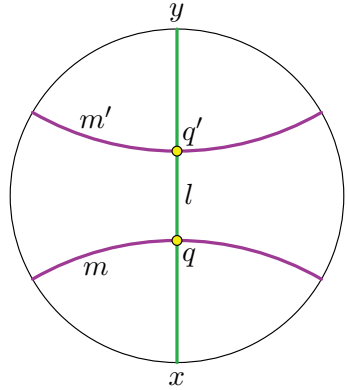
\begin{figure}[htbp]
\labellist
\small\hair 3pt
\pinlabel $m$ [t] at 75 95
\pinlabel $m'$ [b] at 75 198
\pinlabel $l$ [l] at 145 148
\pinlabel $x$ [t] at 145 0
\pinlabel $q$ [tl] at 145 105
\pinlabel $q'$ [bl] at 145 185
\pinlabel $y$ [b] at 145 290
\endlabellist
\includegraphics[width=0.35\textwidth]{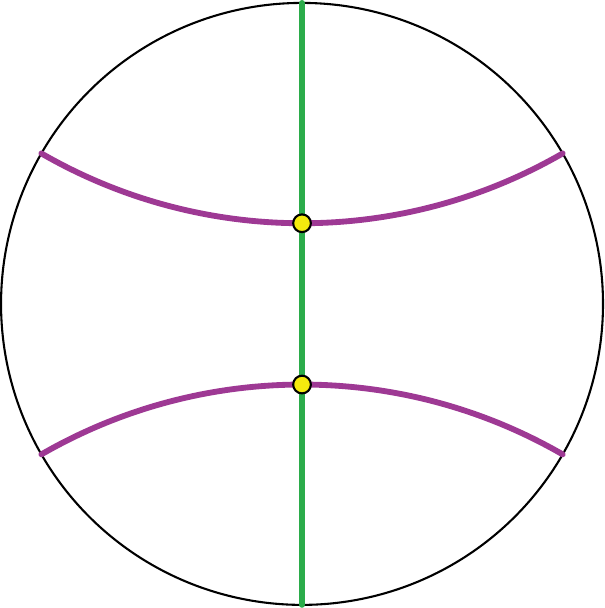}
\caption{The ideal points of the leaf $m$ separate $x$ from the ideal points of $m'$.}
\label{Fig:LeftIsLeft}
\end{figure}

We have the following relations between the topologies of $\Link$ and of $\Circle$.

\begin{lemma}
\label{Lem:LeftIsLeft}
Suppose that $\ell \subset \Link$ is a leaf of either foliation, and similarly for $m$ and $m'$. We have the following.
\begin{enumerate}
\item
\label{Itm:TwoEndpoints}
The two ideal points $\bdyi \ell$ are distinct.
\item
\label{Itm:CrossIFFLinked}
The leaves $\ell$ and $m$ cross if and only if $\bdyi \ell$ and $\bdyi m$ are linked in $\Circle$. 
\item
\label{Itm:LeftIsLeft}
Let $\{x, y\} = \bdyi \ell$.
Suppose that $m$ and $m'$ are distinct leaves that cross $\ell$ at $q$ and $q'$, respectively.
Then $q$ separates $q'$ from the end of $\ell$ associated to $x$ if and only if, in $\Circle$, the ideal points $\bdyi m$ separate $\bdyi m'$ from $x$ (see \reffig{LeftIsLeft}).
\end{enumerate}
\end{lemma}

\begin{proof}
Conclusions~\refitm{TwoEndpoints} and~\refitm{CrossIFFLinked} follow directly from the definition of $\Link$ (see~\cite[Section~10]{FSS22}).
 
Suppose that $\ell$ is a non-cusp leaf.
Let $\lambda$ be the interior leaf of $\Lambda^\calV$ giving $\ell$.
Let $\Lambda_\lambda \subset \Lambda_\calV$ be the set of leaves that link $\lambda$.
As in the proof of~\cite[Lemma~8.24]{FSS22}, the set $\Lambda_\lambda$ is linearly ordered and locally has the topology of the Cantor set.
As in the proof of~\cite[Theorem~10.51]{FSS22}, a copy of the set $\Lambda_\lambda$ collapses to give the leaf $\ell$;
the order on $\Lambda_\lambda$ orients $\ell$.

The collapsing map preserves order as follows.
Suppose that $\mu$ and $\mu'$ are leaves of $\Lambda_\lambda$ with images containing $m$ and $m'$, respectively.
Then $\mu < \mu'$ in $\Lambda_\lambda$ if and only if $q < q'$ along $\ell$.
By \reflem{NoTriangles}, the ideal points of $m$ and $m'$ are distinct.
Conclusion~\refitm{LeftIsLeft} now follows when $\ell$ is a non-cusp leaf.
The proof is similar when $\ell$ is a cusp leaf.
\end{proof}

From \reflem{Laminations}\refitm{NoCusps} we deduce the following. 

\begin{lemma}
\label{Lem:NoDoubleCusps}
A leaf $\ell$, in either foliation, has at most one ideal point at a cusp.  \qed
\end{lemma}


\section{The veering disc}
\label{Sec:Disc}

Suppose that $\calV$ is a (locally) veering triangulation of a three-manifold $M$.  
In \refsec{Veering}, we discussed the associated veering circle $\Circle$ and link space $\Link$.  
The goal of this section is to prove \refthm{Disc}; 
that is, we compactify $\Link$ by $\Circle$ to produce the \emph{veering disc} $\Disc$.

\begin{remark}
\label{Rem:Fenley}
Our result and parts of its proof are inspired by a result of Fenley~\cite[Theorem~3.31]{Fenley12}.  
We follow him in relying on Zippin's characterisation of the two-disc, stated above as \refthm{Zippin}.

Our task is easier than Fenley's in some respects but harder in others.
On the one hand, our task is easier because we begin with the veering circle; 
indeed we used it to build the link space.  
Fenley instead begins with the \emph{leaf space} and then gives a delicate construction of a circle at infinity. 
On the other hand, our task is more difficult; 
unlike Fenley, our three-manifold $M$ has torus boundary.  
Thus our proof must contend with the cusp points in $\Circle$.
\end{remark}

\begin{remark}
There are at least two other approaches to proving \refthm{Disc}.

Modelled on the construction of the veering two-sphere, we form the pre-sphere out of hemispheres $D^\calV$ and $D_\calV$, glued along $\Circle$.
As before, we build a copy of $[\Lambda^\calV]^D$ in $D^\calV$.
Now we build a (reflected) copy of $[\Lambda_\calV]_D$ also in $D^\calV$.
We denote this copy by $[\Lambda_\calV]^D$.
We define a partition of the pre-sphere as follows.
\begin{itemize}
\item Each point of $\Circle$ is its own equivalence class.
\item Each point of the interior of $D_\calV$ is its own equivalence class.
\item Two points of the interior of $D^\calV$ are equivalent if they are connected by an arc in $D^\calV$ that does not strictly cross any leaf of either $[\Lambda^\calV]^D$ or $[\Lambda_\calV]^D$.
\end{itemize}
These give a closed equivalence relation on the pre-sphere, satisfying the hypotheses of \refthm{Moore}.
Taking the quotient, we obtain a two-sphere.
The image of $\Circle$ is a Jordan curve and bounds the image of $D^\calV$.
This is the desired disc $\Disc$.

Alternatively one could follow~\cite{Mather82}. 
The exact formulation required can be found as~\cite[Theorem~1.2]{Bonatti26}. 
Note, however, that for this approach one would need to carefully check that $\Circle$ places the same circular order on the cusps as Mather's circle.
\end{remark}

\subsection{Space and topology}

We define the space $\Disc$.

\begin{definition}
\label{Def:Set}
The underlying set for the \emph{veering disc} $\Disc$ is the disjoint union of the link space $\Link$ and the veering circle $\Circle$.
\end{definition}

\begin{definition}
\label{Def:Action}
We equip $\Disc$ with an action by $\pi_1(M)$ by taking the ``disjoint union'' of its actions on $\Link$ and on $\Circle$.
\end{definition}

\begin{definition}
\label{Def:HalfDisc}
Suppose that $\rho \subset \Link$ is an oriented embedded polygonal line.  
Let $u$ and $v$ be the initial and terminal ideal points of $\rho$, respectively. 
So $\bdyi \rho = \{u, v\}$. 
Throughout this definition we assume that $u \neq v$.  
By \refcor{TransverseOmnibus}\refitm{Proper} the line $\rho$ is properly embedded.  
Thus \refthm{Schoenflies} implies that $\rho$ separates $\Link$ into two connected components.  

\begin{itemize}
\item
Let $H^\circ(\rho)$ be the component of $\Link - \rho$ to the left of $\rho$.  
We call $H^\circ(\rho)$ the \emph{half-plane} associated to $\rho$.   
\item
Let $H(\rho) = H^\circ(\rho) \cup (v, u)^\acw$.  
We call $H(\rho)$ the \emph{half-disc} associated to $\rho$.  
When $\rho$ is a cusp leaf, we call $H(\rho)$ a \emph{cusp half-disc}.
\item
We call $\bdyi \rho = \{u, v\}$ the \emph{ideal corners} of $H(\rho)$.
\item
We call the material corners of $\rho$ the \emph{material corners} of $H(\rho)$.  
\item
At every material corner we say that $\rho$ \emph{turns left} or \emph{turns right} as it turns towards or away from $H^\circ(\rho)$.  
\item
We call $\bdyi H(\rho) = (v, u)^\acw$ the \emph{ideal boundary} of $H(\rho)$.
\item
We call $\bdy_\mat H(\rho) = \rho$ the \emph{material boundary} of $H(\rho)$.
\qedhere
\end{itemize}
\end{definition}

\begin{remark}
Note that $\rho \cup \bdyi \rho$ is disjoint from $H(\rho)$.  
\end{remark}

\begin{definition}
\label{Def:Topology}
The topology on $\Disc$ is defined to be the smallest topology where the cusp half-discs $H(\ell)$ are open, as $\ell$ ranges over oriented cusp leaves of $F^\calV$ and of $F_\calV$. 
\end{definition}

\begin{remark}
\label{Rem:SubBasis}
The cusp half-discs give a \emph{sub-basis}~\cite[page~4]{SteenSeebach95}.  
\end{remark}

\begin{lemma}
\label{Lem:Action}
The $\pi_1(M)$--action on $\Disc$ is by homeomorphisms. 
\end{lemma}

\begin{proof}
The action preserves the set of cusp leaves;
thus it preserves the given sub-basis.
\end{proof}

\begin{lemma}
\label{Lem:LeafHalfDisc}
Suppose that $\ell$ is an oriented leaf in $\Link$.  
Then $H(\ell)$ is open in $\Disc$. 
\end{lemma}

\begin{proof}
Suppose that $\ell$ is a leaf of either foliation.  
If $\ell$ is a cusp leaf then $H(\ell)$ is a cusp half-disc; 
thus it is open by \refdef{Topology}.

Suppose instead that $\ell$ is a non-cusp leaf.
Fix a point $p$ on $\ell$.
We build, as in \refsec{Astroid}, the staircase $S$ about $p$ to the left of $\ell$ containing the initial end of $\ell$. 
See \reffig{ApproachLeaf}.
Let $(c_i)$ be the cusps of $S$. 
Each $c_i$ has two cusp leaves $\ell_i$ and $m_i$ meeting the interior of $S$. 
By the astroid lemma~\cite[Lemma~4.10]{SchleimerSegerman24} the intersection points $\ell \cap m_i$ exit the initial end of $\ell$. 
Again applying the astroid lemma, we deduce that the cusp leaves $\ell_i$ approach $\ell$ from the left. 
Thus $H(\ell) = \medcup_i H(\ell_i)$ is open. 
\end{proof}

\begin{figure}[htbp]
\labellist
\small\hair 2pt
\pinlabel $p$ [t] at 700 10
\pinlabel $\ell$ [l] at 1390 10
\pinlabel $\ell_i$ [l] at 1390 120
\pinlabel $m_i$ [t] at 420 10
\pinlabel $c_i$ [br] at 420 120
\endlabellist
\includegraphics[width=0.75\textwidth]{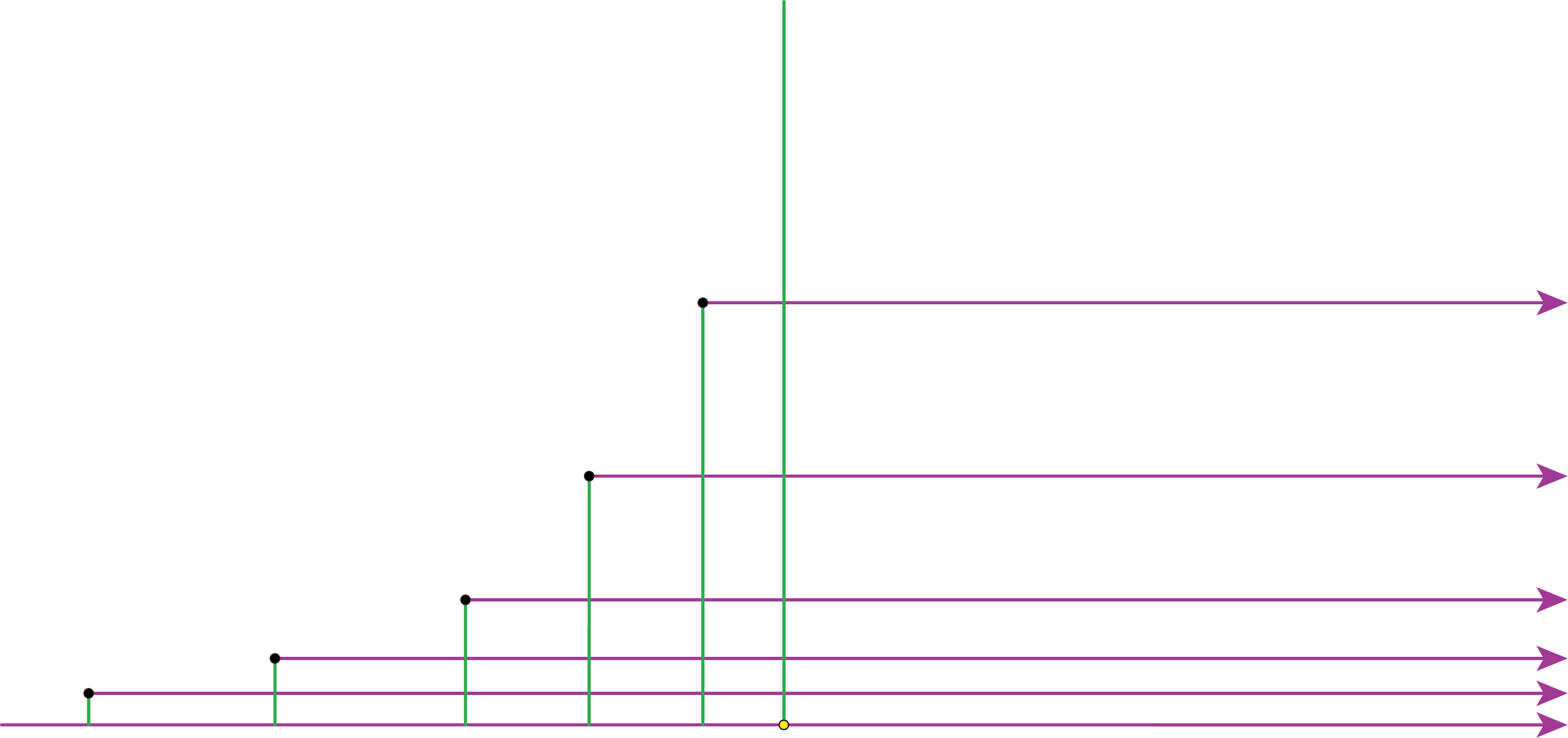}
\caption{A sequence of cusp leaves $\ell_i$ approaching a leaf $\ell$.}
\label{Fig:ApproachLeaf}
\end{figure}
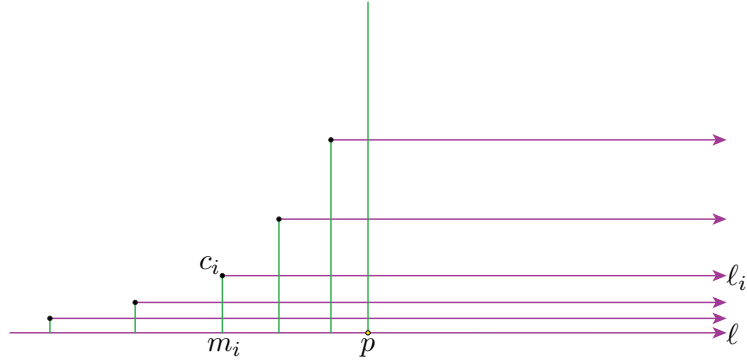

\begin{lemma}
\label{Lem:Rectangle}
Suppose that $R$ is a rectangle in $\Link$.  
Then $R$ is open in $\Disc$. 
\end{lemma}

\begin{proof}  
Applying \refcor{TransverseOmnibus}\refitm{Ascending}, it suffices to consider the case where $R$ is precompact.  
We give $\bdy R$ its induced orientation.  
By \refcor{TransverseOmnibus}\refitm{PreCompactBoundary}, there are four oriented leaves $\{\ell_i\}_{i = 0}^3$ so that 
\begin{itemize}
\item
$R$ is to the left of $\ell_i$ and 
\item
one segment of $\bdy R$ lies in $\ell_i$. 
\end{itemize}
Thus $R = \medcap_i H(\ell_i)$ and we are done by \reflem{LeafHalfDisc}.
\end{proof}

\begin{lemma}
\label{Lem:SubspaceL}
The subspace topology for $\Link \subset \Disc$ has the following properties.
\begin{enumerate}
\item
\label{Itm:Sub}
It agrees with the original topology on $\Link$; 
thus it is homeomorphic to $\RR^2$. 
\item
\label{Itm:Open}
$\Link$ is open as a subset of $\Disc$.  
\item
\label{Itm:Arc}
$\Link$ is path-connected by embedded polygonal arcs. 
\end{enumerate}
\end{lemma}

\begin{proof} 
By \reflem{Rectangle}, rectangles are open in the subspace topology on $\Link$.  
By \refcor{TransverseOmnibus}\refitm{PreCompactBasis}, precompact rectangles give a basis for the original topology. 
Thus the subspace topology has at least as many open sets as the original topology.   

Now suppose that $\ell$ is an oriented cusp leaf.  
The half-plane $H^\circ(\ell) = H(\ell) \cap \Link$ is a union of rectangles; thus the subspace topology on $\Link$ has no more open sets than the original topology on $\Link$.  
This proves \refitm{Sub}.

Since $\Link$ is a union of rectangles it is open in $\Disc$.  
This gives \refitm{Open}.

Finally, suppose that $p$ and $q$ are points of $\Link$.  
Since $\Link$ is homeomorphic to $\RR^2$ it is path-connected.  
Let $s$ be an arc from $p$ to $q$.  
Since $s$ is compact, we can cover it by finitely many rectangles.  
There is an embedded polygonal arc $\sigma$ from $p$ to $q$ contained in the union of these rectangles.  
This proves \refitm{Arc}.
\end{proof}

We turn to the task of understanding the subspace topology induced on $\Circle$.
We begin with an observation that may be proved using \refcor{TransverseOmnibus}\refitm{DiscUnion}.

\begin{lemma}
\label{Lem:HalfDiscOpen}
Suppose that $\rho$ is an embedded polygonal line in $\Link$.  
Then $H(\rho)$ is open. \qed
\end{lemma}

\begin{lemma}
\label{Lem:Span}
Suppose that $c$ and $d$ are cusps.  
Then there is an embedded polygonal line $\rho$ with ideal points at $c$ and $d$. 
\end{lemma}

\begin{proof}
Suppose that $\ell$ and $m$ are distinct cusp leaves ending at $c$ and $d$, respectively.  
\reflem{SubspaceL}\refitm{Arc} gives us an embedded polygonal arc $\beta$ running from a point of $\ell$ to a point of $m$.  
We extend $\beta$ by subrays of $\ell$ and $m$ towards $c$ and $d$, respectively. 
Finally, we delete loops and backtracking. 
\end{proof}

\begin{lemma}
\label{Lem:SubspaceS}
The subspace topology on $\Circle$ agrees with its original topology.  
\end{lemma}

\begin{proof}
By construction, the subspace topology on $\Circle$ has no more open sets than the original topology.  

By Lemmas~\ref{Lem:Span} and~\ref{Lem:HalfDiscOpen} the subspace topology also contains a basis for the original topology, namely the open intervals ending at cusps.  
Thus the subspace topology has no fewer open sets than the original topology.
\end{proof}

We will need a finer analysis of the neighbourhoods (in $\Disc$) of points in $\Circle$.  

\begin{definition}
\label{Def:Basic}
A \emph{basic} set $H = \medcap_i H(\ell_i)$ is a finite intersection of cusp half-discs.
We extend the notations of \refdef{HalfDisc} to basic sets.  
We orient the components of its material boundary, $\bdy_\mat H$, so that $H$ is (locally) to the left of each. 
We define $H^\circ = H \cap \Link$. 
\end{definition}

\begin{remark}
\label{Rem:Basis}
By \refrem{SubBasis}, the basic sets form a \emph{basis}~\cite[page~4]{SteenSeebach95} for the topology on $\Disc$. 
\end{remark}

Note that $\bdy_\mat H$ need not be connected when $H$ is a basic set.

\begin{lemma}
\label{Lem:BasicOmnibus}
Suppose that $H$ is a basic set.
\begin{enumerate}
\item
\label{Itm:TurnsLeft}
Suppose that $\rho$ is a component of $\bdy_\mat H$ and suppose that $q$ is a material corner of $\rho$.  
Then $\rho$ turns left at $q$. 
\item
\label{Itm:Loop}
Suppose that $\rho$ is a component of $\bdy_\mat H$ which is a loop.
Then $H$ is a precompact rectangle. 
\item
\label{Itm:PathConnected}
The intersection $H^\circ = H \cap \Link$ is path-connected by embedded polygonal arcs. 
\end{enumerate}
\end{lemma}

\begin{proof}
Suppose that $H$ equals $\medcap_{i = 1}^n H(\ell_i)$ where the $\ell_i$ are cusp leaves. 

Suppose that $\rho$ is a component of $\bdy_\mat H$ and suppose that $q$ is a material corner of $\rho$.  
The segments of $\rho$ lie in the leaves $\ell_i$.  
Let $\ell$ and $m$ be the leaves of the two foliations through $q$.  
Let $R$ be a small rectangle about $q$.  
Then $\ell$ and $m$ cut $R$ into four subrectangles.  
Since $H \subset H(\ell) \cap H(m)$, the basic set $H$ only contains one of the four subrectangles.  
Finally, since $\rho$ is oriented so that $H$ lies to its left, we deduce that $\rho$ turns left at $q$.  
This proves \refitm{TurnsLeft}.  

Suppose that $\rho$ is a component of $\bdy_\mat H$ which is a loop.  
By our orientation convention, $H$ locally lies to the left of $\rho$. 
By the Jordan curve theorem (\refthm{Schoenflies}), the loop $\rho$ cuts $\Link$ into a precompact piece (to its left) and a non-precompact piece (to its right).
Thus $H$ is a precompact open disc and $\bdy_\mat H = \rho$.  
By the previous paragraph, $\rho$ turns left at all material corners.  
An Euler characteristic argument (see~\cite[Section~3]{SchleimerSegerman20}) now shows that $H$ is a rectangle.  
This proves \refitm{Loop}.

Suppose that $p$ and $q$ are points of $H$.  
By \reflem{SubspaceL}\refitm{Arc} there is an embedded polygonal arc $\beta \subset \Link$ connecting them.  
Since the segments of $\beta$ are compact, we can tile them by rectangles and so move $\beta$ (inside the rectangles) so that it is transverse to $\bdy_\mat H$.  
Surger $\beta$ off of $\bdy_\mat H$, discarding loops as they arise.
This proves \refitm{PathConnected}.
\end{proof}

\begin{lemma}
\label{Lem:HalfDiscNeighBasis}
Suppose that $x$ is a point of $\Circle$.  
The half-discs (as given in \refdef{HalfDisc}) containing $x$ give a neighbourhood basis for $x$ in $\Disc$. 
\end{lemma}

\begin{proof}
Suppose that $H$ is a basic set containing $x$.  
Since $\bdyi H$ is an intersection of finitely many open intervals, 
it is a disjoint union of finitely many open intervals.  
Let $(v, u)^\acw$ be the component of $\bdyi H$ containing $x$.  
Let $\rho$ and $\sigma$ be the components of $\bdy_\mat H$ that start and end at $u$ and $v$, respectively. 
(Here we use the induced orientations coming from $H$.)
If $\rho = \sigma$ then the half-disc $H(\rho)$ equals $H$ and we are done.

Suppose instead that $\rho$ is not equal to $\sigma$.
By \reflem{BasicOmnibus}\refitm{PathConnected} there is an embedded polygonal arc $\beta' \subset H$ connecting a point near $\rho$ to a point near $\sigma$.  
We extend $\beta'$ to obtain an embedded polygonal arc $\beta$ so that $\beta$ connects a point of $\rho$ to a point of $\sigma$ and $\beta \cap \bdy_\mat H$ consists of those two points.
We orient $\beta$ from $\rho$ and towards $\sigma$.  
Let $\rho'$ be the component of $\rho - \beta$ with ideal point at $u$; 
let $\sigma'$ be the component of $\sigma - \beta$ with ideal point at $v$.  
Thus $\tau = \rho' \cup \beta \cup \sigma'$ is an embedded polygonal line.  
Note that the half-disc $H(\tau)$ is contained in $H$.
Furthermore, $\bdyi H(\tau) = (v, u)^\acw$ contains $x$.
\end{proof}

The next lemma is somewhat technical, but is needed to prove \reflem{Extension}.
That, in turn, shows how points of $\Link$ are path-connected to points of $\Circle$. 

\begin{lemma}
\label{Lem:Triple}
Suppose that $\rho$ is a leaf of either $F^\calV$ or $F_\calV$.  
Let $q$ be a point of $\Link$ not lying on $\rho$.  
Let $\ell$ and $m$ be the leaves of $F^\calV$ and $F_\calV$ meeting $q$.  
Then $q$ lies in $H^\circ(\rho)$ if and only if at least three of the four ideal points of $\ell$ and $m$ lie in $\bdyi H(\rho)$. 
\end{lemma}

\begin{proof}
It suffices to prove the forward direction.
So, assume that $q$ lies in $H^\circ(\rho)$.

\reflem{SubspaceL}\refitm{Arc} gives us an oriented embedded polygonal arc $\sigma \subset \Link$ running from $q$ to $\rho$, and meeting each only at its endpoints.  
By \refcor{TransverseOmnibus}\refitm{Proper} and \refthm{Schoenflies}, and since $q$ lies in $H^\circ(\rho)$, the interior of $\sigma$ lies in $H^\circ(\rho)$.  
Suppose that $\{\ell_i\}_{i = 0}^{n-1}$ are the leaves which $\sigma$ runs along, with orientations induced by $\sigma$.  
We order the indices of the $\ell_i$ to increase with the orientation of $\sigma$. 
We will assume that $\ell_0 = \ell$, the leaf of $F^\calV$ through $q$.
Set $\bdyi \ell = \{x, y\}$ with $\ell$ running from $x$ to $y$.  
Also, let $q_i$ and $q_{i+1}$ be the endpoints of $\sigma \cap \ell_i$.  
So $q_0 = q$ and $q_n$ lies in $\rho$.  
We will induct on $n$. 

Suppose that $n = 1$.  
By \reflem{LeftIsLeft}\refitm{CrossIFFLinked}, since $\rho$ and $\ell$ cross, their ideal points are linked in $\Circle$.
So exactly one of $x$ and $y$ lies in $\bdyi H(\rho)$.
Since $q$ lies in $H^\circ(\rho)$, we find that $m$ is not equal to $\rho$.
Since $\ell$ is oriented from $q$ to $q_1$, we deduce that $q$ separates $q_1$ from the end of $\ell$ associated to $x$.
Applying \reflem{LeftIsLeft}\refitm{LeftIsLeft}, we find that $\bdyi m$ separates $\bdyi \rho$ from $x$.  
Thus when $n = 1$ the intersection $(\bdyi \ell \cup \bdyi m) \cap \bdyi H(\rho)$ equals $\{x\} \cup \bdyi m$.
This completes the base case.  

Suppose that $n > 1$.  
Thus $q_1$ lies in $H^\circ(\rho)$. 
Recall that $\ell_0$ and $\ell_1$ are the two leaves through $q_1$.
Thus, by induction, at least three of the four ideal points of $\ell_0$ and $\ell_1$ lie in $\bdyi H(\rho)$.  

We now have various subcases. 
\begin{itemize}
\item
Suppose that both $x$ and $y$ lie in $\bdyi H(\rho)$.  
Since $\ell = \ell_0$ and $m$ cross, their ideal points are linked.  
So at least one point of $\bdyi m$ also lies in $\bdyi H(\rho)$.
\item
Suppose that $x$, say, lies in $\bdyi H(\rho)$ while $y$ does not, and $y$ is not an ideal point of $\rho$.  
Thus $\ell$ and $\rho$ cross.
In this case we replace $\sigma$ by a segment of $\ell$, and appeal to the base case.
\item
Finally, suppose that $x$, say, lies in $\bdyi H(\rho)$ while $y$ is one of the two points of $\bdyi \rho$.  
By \refcor{OnlyAsymptotics}, and because $\rho$ is not equal to $\ell$, we find that $y$ is a cusp. 
Since $\{x, y\}$ links $\bdyi m$, at least one point of $\bdyi m$ lies in $\bdyi H(\rho)$.  
Let $z$ be the other point of $\bdyi m$.
There are three subcases.
\begin{itemize}
\item
If $z$ lies in $\bdyi H(\rho)$ then we are done.
\item
If $z$ lies in $\bdyi \rho$ then by \refcor{OnlyAsymptotics} and \reflem{Laminations}\refitm{NoCusps}, the leaf $m$ is a cusp leaf emanating from $z = y$.
Thus $\bdyi \ell$ does not link $\bdyi m$, a contradiction.
\item 
If $z$ lies outside of $\bdyi H(\rho) \cup \bdyi \rho$ then $\bdyi m$ links $\bdyi \rho$.
Thus $m$ and $\rho$ cross, and we reduce to the base case.
\qedhere
\end{itemize}
\end{itemize}
\end{proof}

\begin{lemma}
\label{Lem:Extension}
Suppose that $\ell$ is a leaf of either foliation.  
Suppose that $r$ is a subray of $\ell$.  
Then the union $r \cup \bdyi r$ is an arc in $\Disc$.  
\end{lemma}

\begin{proof}
We define a function $\lambda \from [0, 1] \to \Disc$ by
\begin{itemize}
\item
taking $\lambda | [0, 1)$ to be an orientation-preserving homeomorphism from $[0, 1)$ to $r$ and
\item
setting $\lambda(1) = \bdyi r$.
\end{itemize}
Note that $\lambda$ is injective and its image meets $\Circle$ only at $\bdyi r$.  
We must show that $\lambda$ is continuous.

Let $U \subset \Disc$ be any open set.  
\reflem{SubspaceL}\refitm{Open} implies that $U \cap \Link$ is open.  
Thus $\lambda^{-1}(U \cap \Link)$ is open.  
If $\bdyi r$ does not lie in $U$ then there is nothing left to prove.  
Suppose that $\bdyi r$ lies in $U$; thus $U$ is an open neighbourhood of $\bdyi r$.   

\begin{claim*}
For some $t \in (0, 1)$ we have $(t, 1) \subset \lambda^{-1}(U \cap \Link)$.
\end{claim*}

\begin{proof}
Applying \refdef{Topology} we may replace $U$ by a smaller neighbourhood of the form $V = \medcap_{i} H(m_i)$.  
Here $\{H(m_i)\}$ is a finite collection of cusp half-discs whose ideal boundaries each contains $\bdyi r$.

Fix $i$.
We now must show that a subray of $r$ lies in $H(m_i)$.

Recall that $\bdyi r$ lies in $\bdyi H(m_i)$.  
Thus $\ell$ is not equal to $m_i$.  
If $\ell$ lies in $H(m_i)$ then we are done.   
So suppose that $\ell$ contains a point $q$ of $H^\circ(m'_i)$, where $m'_i$ is the leaf $m_i$ with the opposite orientation.
Let $m$ be the other leaf through $q$.  
Applying \reflem{Triple} (with $m'_i$ for $\rho$), at least three of the four ideal points $\bdyi m \cup \bdyi \ell$ lie in $\bdyi H(m'_i)$.  
However, $\bdyi r$ lies in $\bdyi H(m_i)$.
Therefore $\ell$ and $m_i$ cross, say at the point $q_i$.  
By \reflem{LeftIsLeft}\refitm{LeftIsLeft} we find that $q_i$ separates $q$ from the end of $\ell$ associated to $\bdyi r$.  
Thus the end of $r$ lies in $H(m_i)$, proving the claim.
\end{proof}

Since $\bdyi r$ lies in $U$, we deduce that $(t, 1]$ lies in $\lambda^{-1}(U)$, proving that $\lambda$ is continuous.
\end{proof}

As a special case of \reflem{Extension}, cusp leaves connect to their cusps. 
Since cusps are dense in $\Circle$, we deduce the following.

\begin{corollary}
\label{Cor:HalfDiscPathConnected}
Suppose that $\rho$ is an embedded polygonal line in $\Link$.  
Then $H(\rho)$ is path-connected. \qed
\end{corollary}

Here is a more subtle application of \reflem{Extension}.

\begin{lemma}
\label{Lem:ConnectedImpliesConnected}
Suppose that $H$ and $K$ are half-discs with $\bdyi H \cap \bdyi K$ non-empty.  
Then $H^\circ \cap K^\circ$ is also non-empty. 
\end{lemma}

\begin{proof}
Note that $\bdyi H$ and $\bdyi K$ are open intervals in $\Circle$.  
Thus there is some cusp $c$ in their intersection.  
Let $r$ be a subray of a cusp leaf ending at $c$.  
Let $\lambda \from [0, 1] \to \Disc$ be the arc given by \reflem{Extension}.  
So $\lambda^{-1}(H)$ and $\lambda^{-1}(K)$ are open neighbourhoods of $1$; 
thus there is a point $t \in (0, 1)$ in both. 
Thus $\lambda(t)$ lies in $H^\circ \cap K^\circ$. 
\end{proof}

\subsection{Peano continuum}

Before giving the proof of \refthm{Disc}, we must prove that $\Disc$ is a Peano continuum (\refdef{PeanoContinuum}).  

\begin{lemma}
\label{Lem:Hausdorff}
The space $\Disc$ is Hausdorff.
\end{lemma}

\begin{proof}
Suppose that $x$ and $y$ are distinct points.  
If both lie in $\Link$ then we appeal to \reflem{SubspaceL}\refitm{Sub}.

Suppose that $x$ and $y$ both lie in $\Circle$.  
Let $c, d \in \Delta_\calV$ be a pair of cusps that link the pair $(x, y)$.  
By \reflem{Span} there is an embedded polygonal line $\rho$ with ideal points $c$ and $d$.  
Thus $x$ and $y$ are contained in the left and right half-discs corresponding to $\rho$.  
By \reflem{HalfDiscOpen} these are open. 

Finally suppose that $x$ lies in $\Link$ while $y$ lies in $\Circle$.  
Let $R$ be a rectangle which is precompact in $\Link$ and contains $x$; 
by \reflem{Rectangle} the rectangle $R$ is open.  

We orient the sides of $R$ anticlockwise.  
This induces orientations of the leaves $\{\ell_i\}_{i = 0}^3$ containing the sides of $R$.  
Note that $\ell_i$ and $\ell_{i+1}$ cross.  
By \refcor{OnlyAsymptotics} the leaves $\ell_i$ have no ideal points in common.  
Thus the half-discs $H(\ell_i)$ cover $\Circle$; 
breaking symmetry suppose that $H(\ell_0)$ contains $y$; 
by \reflem{LeafHalfDisc} the half-disc $H(\ell_0)$ is open.  
Since $H(\ell_0)$ and $R$ are disjoint, and are open as above, we are done.
\end{proof}

\begin{lemma}
\label{Lem:Compact}
The space $\Disc$ is compact.
\end{lemma}

\begin{proof}
Suppose that $\calU_I = \{ U_i \}_{i \in I}$ is the given open cover.  
Since the $U_i$ are open, by appealing to \reflem{HalfDiscNeighBasis} for every $x \in \Circle$ we may pick an embedded polygonal line $\rho_x$ so that 
$x \in H(\rho_x) \subset U_i$ for some $i \in I$.
Set $V_x = H(\rho_x)$ and $\calV_{S^1} = \{ V_x \}_{x \in S^1}$.  
So $\calV_{S^1}$ covers $\Circle$.  
Appealing to the compactness of $\Circle$ there is a finite set $X \subset S^1$ so that $\calV_X = \{ V_x \}_{x \in X}$ covers $\Circle$.   


\begin{claim*}
The set $C = \Link - \medcup_X V_x$ is compact. 
\end{claim*}

\begin{proof}
Let $\Sigma$ be the closure of $C$ minus the interior of $C$.  
Each polygonal line $\rho_x$ is made up of a finite collection of segments lying in leaves from $F^\calV$ and $F_\calV$.  
Also, $\Sigma$ lies in the union $\cup_{x\in X} \rho_x$.
So $\Sigma$ is a union of finitely many segments in leaves.
For a contradiction, suppose that one of these segments contains a ray, say $\rho$.
Let $y \in \Circle$ be the ideal point associated to $\rho$.
Since the intervals $\bdyi V_x$, for $x \in X$, cover $\Circle$, the point $y$ lies in one of these.  
Thus, by Lemmas~\ref{Lem:Extension} and~\ref{Lem:HalfDiscOpen} the ray $\rho$ enters $V_x$, a contradiction.  
Thus all (finitely many) segments of $\Sigma$ are intervals.
So $\Sigma$ is a finite graph in the plane $\Link$.

For every loop $\sigma \subset \Sigma$, \refthm{Schoenflies} provides a unique disc $H(\sigma)$ in $\Link$ with $\bdy H(\sigma) = \sigma$.   
Choose $\sigma$ so that $H(\sigma)$ does not contain any smaller such discs.
Suppose for a contradiction that there is some $x \in X$ so that $H^\circ(\rho_x)$ lies in $H(\sigma)$.  
Then the points $\bdyi \rho_x$ are contained in the closure of $H(\sigma)$.  
Thus they are contained in $\Link$, the desired contradiction. 

We deduce that $C$ is contained in the union of the closures of the (finitely many) $H(\sigma)$, together with finitely many intervals in leaves. 
Thus $C$ is compact. 
\end{proof}

Applying the claim and recalling that $\calU_I$ is an open cover, there is a finite subset $K \subset I$ so that $\calU_K$ covers $C$.  
For each $x \in X$ choose some $j \in I$ so that $V_x \subset U_j$.
Collecting the finitely many indices $j$ into a set $J \subset I$, we find that $\calU_{J \cup K}$ is a finite cover of $\Disc$. 
\end{proof}

\begin{lemma}
\label{Lem:Continuum}
The space $\Disc$ is a Peano continuum.
\end{lemma}

\begin{proof}
Since $\Disc$ contains $\Link$, which is homeomorphic to $\RR^2$, it has two distinct points.  
Thus $\Disc$ is non-degenerate.

The sub-basis for the topology, given in \refdef{Topology}, is countable.  
Taking finite intersections gives a basis (\refrem{Basis}) which is still countable.  
Thus $\Disc$ is second countable.

Lemmas~\ref{Lem:Hausdorff} and~\ref{Lem:Compact} tell us that $\Disc$ is Hausdorff and compact.  Thus, by~\cite[Theorem~1.27, page~74]{Wilder49}, the space $\Disc$ is normal.

Since $\Disc$ is compact it is locally compact.

From Lemmas~\ref{Lem:SubspaceL}\refitm{Arc} and~\ref{Lem:SubspaceS} we have that both $\Link$ and $\Circle$ are path-connected.
\refcor{HalfDiscPathConnected} tells us that $H(\ell)$ is path-connected for any leaf $\ell$.
Thus $\Disc$ is path-connected and so is connected.

By \reflem{SubspaceL}\refitm{Sub} we have that for all $x \in \Link$ the space $\Disc$ is locally connected at $x$.  
Suppose instead that $x$ lies in $\Circle$.  
Fix any open neighbourhood $U$ of $x$ in $\Disc$.  
By \reflem{HalfDiscNeighBasis} there is some half-disc $H \subset U$ so that $x \in \bdyi H$.  
By \refcor{HalfDiscPathConnected} we have that $H$ is path-connected.  
Thus $\Disc$ is locally connected at $x$.  
Thus $\Disc$ is a Peano space. 

Finally, by \reflem{Compact}, we have that $\Disc$ is compact. 
Thus it is a Peano continuum.
This completes the proof. 
\end{proof}

\subsection{The veering disc is a disc}

\begin{theorem}
\label{Thm:Disc}
Suppose that $\calV$ is a transverse veering triangulation of a three-manifold $M$.
Then we have the following.
\begin{enumerate}
\item 
\label{Itm:Disc}
The veering disc $\Disc$ is a disc, having the veering circle $\Circle$ as its boundary and the link space $\Link$ as its interior. 
\item 
\label{Itm:Ends}
For any ray $r$ of any leaf of $F^\calV$ or $F_\calV$ we have that its closure in $\Disc$ is $\closure{r} = r \cup \bdyi r$.
\item 
\label{Itm:Dense}
The endpoints of cusp leaves (respectively, non-cusp leaves) are dense in $\Circle$.
\item 
\label{Itm:Unique}
Any other compactification of $\Link$ by a circle yielding a disc, with properties \refitm{Ends} and \refitm{Dense}, is homeomorphic to $\Disc$.
\item 
\label{Itm:Action}
Furthermore, the action of $\pi_1(M)$ on $\Disc$ is by orientation-preserving homeomorphisms. 
\end{enumerate}
\end{theorem}

\begin{remark}
Note that here we do not assume that $\calV$ is a finite triangulation.
\end{remark}

\begin{proof}[Proof of \refthm{Disc}]
To prove that $\Disc$ is a disc, we verify the hypotheses of Zippin's characterisation (\refthm{Zippin}).  
Note that \reflem{Continuum} proves that $\Disc$ is a Peano continuum.  
Our candidate boundary is the veering circle $\Circle$.  

Suppose that $\ell$ is a leaf of $F^\calV$ or $F_\calV$.  
By \reflem{LeftIsLeft}\refitm{TwoEndpoints} the leaf $\ell$ has two distinct ideal points.  
By \reflem{Extension} the union $\ell \cup \bdyi \ell$ is an arc in $\Disc$.  
By construction, this union meets $\Circle$ in exactly $\bdyi \ell$.  
Thus $\ell \cup \bdyi \ell$ is a spanning arc, giving hypothesis \refitm{Exists} of \refthm{Zippin}.

Suppose that 
\[
(\gamma, \bdy \gamma) \subset (\Disc, \Circle)
\]
is an oriented spanning arc.  Let $x$ and $y$ be the two points of $\bdy \gamma$, with $\gamma$ oriented away from $x$.  Let $\delta = (y, x)^\acw$ and $\epsilon = (x, y)^\acw$; these are the two components of $\Circle - \bdy \gamma$.  

We write $\gamma^\circ$ for $\gamma - \bdy \gamma$; we call this the \emph{interior} of $\gamma$.  Note that $\gamma^\circ$ is an embedding of $(0, 1)$ into $\Link$.  

\begin{claim}
The interior $\gamma^\circ$ is a Jordan line in $\Link$.
\end{claim}

\begin{proof}
We must prove that $\gamma^\circ$ is properly embedded.  Fix a compact subset $K \subset \Link$.  Since $\gamma$ is continuous and $[0, 1]$ is compact, the preimage $\gamma^{-1}(K)$ is compact.  It is disjoint from $\{0, 1\}$ by assumption.  Note that $(\gamma^\circ)^{-1}(K) = \gamma^{-1}(K)$.  Thus the latter is compact in $(0, 1)$ and we are done. 
\end{proof}

Applying \refthm{Schoenflies}, we find that $\gamma^\circ$ separates $\Link$ into two open discs.  We denote these by $H$ and $K$.  

Fix $u \in \delta$.  
Since $\Disc$ is normal, there is a basic open set $U$ containing $u$ and disjoint from $\gamma$.  
By \reflem{BasicOmnibus}\refitm{PathConnected} the set $U^\circ = U \cap \Link$ is connected; 
thus $U^\circ$ is contained in $H$ or in $K$.  
Breaking symmetry, suppose that $U^\circ$ is contained in $H$.  

\begin{claim}
\label{Clm:DinD}
Suppose that $v$ lies in $\delta$.  
Suppose that $H_v$ is a half-disc neighbourhood of $v$ that misses $\gamma$.
Then $H_v \cap \Link$ lies in $H$.
\end{claim}

\begin{proof}
Let $\delta'$ be any compact subarc of $\delta$ containing both $u$ and $v$.  
Since $\Disc$ is normal, and applying \reflem{HalfDiscNeighBasis}, 
for every $w \in \delta'$ we may pick a half-disc $H_w$ so that $w \in H_w$ and so that $H_w$ misses $\gamma$.  
The compactness of $\delta'$ gives us a finite subcollection of the $H_w$ that covers $\delta'$.  
By \reflem{ConnectedImpliesConnected}, if $H_w$ and $H_{w'}$ meet in $\Circle$ then they meet in $\Link$.  
The connectedness of $\delta'$ now implies that all of the sets $H_w \cap \Link$ lie in $H$.  
We deduce that $H_v \cap \Link$ lies in $H$, as desired. 
\end{proof}

\begin{claim}
\label{Clm:EinE}
Suppose that $v$ lies in $\epsilon$.  
Suppose that $K_v$ is a half-disc neighbourhood of $v$ that misses $\gamma$.  
Then $K_v \cap \Link$ lies in $K$.
\end{claim}

\begin{proof}
Recall that $x$ and $y$ are the ideal points of $\gamma$ and that $\Circle - \{x, y\} = \delta \sqcup \epsilon$.  
Set $K_v^\circ = K_v \cap \Link$.  
For a contradiction, suppose that $K_v^\circ$ lies in $H$.  
Fix any cusp $d$ in $\delta$ and suppose that $H_d$ is a half-disc neighbourhood of $d$, disjoint from $\gamma$.  
By \refclm{DinD}, the set $H_d^\circ$ lies in $H$.  
Let $e$ be any cusp in $\bdyi K_v$.  
As a bit of notation we take $K_e = K_v$. 
Since $v$ lies in $\epsilon$, which is connected, we deduce that $e$ also lies in $\epsilon$.  
Thus $(x, y)$ links $(d, e)$.

Pick rays $r_d$ and $r_e$ (contained in cusp leaves) ending at the cusps $d$ and $e$ and contained in $H_d$ and $K_e$, respectively.  
Since $H$ is path-connected, and since $H_d^\circ$ and $K_e^\circ$ both lie in $H$, there is a polygonal arc connecting a point of $r_d$ to a point of $r_e$.  
Deleting loops, we find that there is an embedded polygonal line $\rho$, contained in $H$, that has $d$ and $e$ as its ideal points.  

Note that $\gamma^\circ$ and $\rho$ are disjoint properly embedded lines in $\Link$.  
Thus $\rho$ divides $H$ into two components.  
We orient $\rho$ so that $H^\circ(\rho)$ is the component of $H - \rho$ that does not contain any point of $\gamma^\circ$ in its closure.  
Since $(x, y)$ links $(d, e)$, breaking symmetry, we have that $y$ lies in $\bdyi H(\rho)$.  
By \reflem{HalfDiscOpen} the half-disc $H(\rho)$ is open.  
Since $\gamma$ is an arc, thinking of it as a function we have that $\gamma^{-1}(H(\rho))$ is open and non-empty. 
So it contains an interval.  
Thus $\gamma^\circ$ meets $H^\circ(\rho)$, a contradiction.
\end{proof}

We extend our previous notation for half-discs and define $H(\gamma) = H \cup \delta$.
Let $\gamma'$ be $\gamma$ with the opposite orientation.
We take $H(\gamma') = K \cup \epsilon$.  
From the claims above, from \refcor{HalfDiscPathConnected}, 
and from the fact that $H$ and $K$ are open discs, we deduce that $H(\gamma)$ and $H(\gamma')$ are path-connected, disjoint, and open.   
Thus $\gamma$ separates, giving hypothesis \refitm{All} of \refthm{Zippin}.

Suppose that $q$ is any point of $\gamma$.  If $q$ lies in $\Link$ then \refthm{Schoenflies} implies that $H \cup \{ q \} \cup K$ is path-connected.  
If $q$ lies in $\Circle$ then $q$ is equal to one of $x$ or $y$: the ideal points of $\gamma$.  
Thus $\delta \cup \{ q \} \cup \epsilon$ is path-connected.  
In either case, $H(\gamma) \cup \{ q \} \cup H(\gamma')$ is path-connected, giving hypothesis \refitm{None} of \refthm{Zippin}.

Thus $\Disc$ is a closed disc with boundary $\Circle$. 
It follows that $\Link$ is the interior of $\Disc$. 
This proves property \refitm{Disc}.

\reflem{Extension} proves property \refitm{Ends}.

Cusps are dense in $\Circle$ by construction~\cite[Theorem~7.1]{FSS22}.
Ends of non-cusp leaves are dense by \reflem{Laminations}\refitm{Approach}.
Together, these imply property \refitm{Dense}.

Properties \refitm{Ends} and \refitm{Dense} imply (with appropriate choices) the characterising properties of~\cite[Theorem~2.1]{Bonatti26}.
Property \refitm{Unique} follows.

By \reflem{Action}, the action of $\pi_1(M)$ on $\Disc$ is by homeomorphisms.  
The elements of $\pi_1(M)$ preserve the orientation of $\Circle$~\cite[Theorem~7.1(1)]{FSS22} and of $\Link$~\cite[Theorem~10.8(6)]{FSS22}.  
Thus they preserve the induced orientation on $\Disc$. 
This proves property \refitm{Action}.
\end{proof}

\section{Hausdorff limits}
\label{Sec:Hausdorff}

We equip $\Disc$ with a metric compatible with its topology.
Recall that the set of closed sets in $\Disc$ is itself a metric space, called the \emph{Hausdorff metric topology}~\cite[page~154]{SteenSeebach95}.
A sequence of closed sets in $\Disc$ converging in the Hausdorff metric topology is said to \emph{Hausdorff converge}.
In this section we characterise Hausdorff limits of (closures of) leaves and skeletal rectangles in $\Disc$. 

As we shall see, the possible Hausdorff limits include (closures of) bi-leaves.
Recall from \refdef{BiLeaf} that a bi-leaf $k$ has the form $\ell \cup \{c\} \cup \ell'$ where $\ell$ and $\ell'$ are cusp leaves belonging to a cusp $c$.
We write $\bdyi k = (\bdyi \ell \cup \bdyi \ell') - \{c\}$.  
We call the points $\bdyi k$ the \emph{endpoints} of $k$.
Note that there are only countably many bi-leaves in $\Disc$. 




\subsection{Hausdorff limits of leaves}

The goal of this section is to prove the following.

\begin{proposition} 
\label{Prop:LimitOfLeaves}
Suppose that $(\ell_i)$ is a sequence of non-cusp leaves and/or bi-leaves in $\Disc$.
Suppose that the closures $(\closure{\ell}_i)$ Hausdorff converge to a set $L$.
Then $L$ is one of the following:
\begin{itemize}
\item
a point of $\Circle$,
\item
the closure of a bi-leaf, or
\item
the closure of a non-cusp leaf. 
\end{itemize}
Furthermore, in the last two cases, let $F$ be the foliation containing $L$.
Then all but finitely many of the $\ell_i$ belong to $F$.
\end{proposition}

We prove the proposition via a sequence of lemmas, all using the same notation as the proposition.

\begin{lemma}
\label{Lem:TwoImpliesLeaf}
Suppose that $L$ contains two distinct points $x$ and $y$.
Then for one of the two foliations, say $F$, there is a leaf or bi-leaf $\ell_\infty$ contained in $F$ so that
\begin{itemize}
\item $L$ contains $\ell_\infty$ and 
\item both $x$ and $y$ lie in the closure of~${\ell}_\infty$.
\end{itemize}
\end{lemma}

\begin{proof}
Since $x$ and $y$ are distinct, there is some cusp leaf $m$ that separates them.
Choose small basic neighbourhoods $U$ and $V$ of $x$ and $y$ disjoint from $m$.
Pass to a tail of the sequence $(\ell_i)$ so that all of the $\ell_i$ meet both $U$ and $V$.
All leaves in this tail cross $m$.
Therefore all $\ell_i$ in this tail are in one of the two foliations, say $F$.

By \refthm{Disc}\refitm{Disc}, the disc $\Disc$ is normal and first countable.
Thus there are nested neighbourhood bases $(U_j)$ and $(V_j)$, about $x$ and $y$ respectively, so that $\closure{U}_j \cap \closure{V}_j$ is empty.

We now recurse on $j$, starting with the $j$th subsequence $(\ell_i)$ of the original sequence of leaves.
Pass to a further tail of $(\ell_i)$ whose members meet both $U_j$ and $V_j$.
We deduce that there are arcs, rays, or lines $m_x$ and $m_y$ in the boundaries of $U_j$ and $V_j$ respectively so that the leaves of an infinite subsequence
of $(\ell_i)$ cross $m_x$ and $m_y$. 
(This gives the $j+1$st subsequence of the original.)
We deduce that the leaves containing $m_x$ and $m_y$ are not in $F$.

Set $p_i = \ell_i \cap m_x$ and set $q_i = \ell_i \cap m_y$.
Note that $p_i \neq q_i$ because the closures of $U_j$ and $V_j$ are disjoint.
Set $\ell'_i = [p_i, q_i] \subset \ell_i$.

The arcs $\ell'_i$ between $m_x$ and $m_y$ all cobound rectangles (with appropriate subarcs of $m_x$ and $m_y$).
The union of these rectangles gives a rectangle $R$.
Since $\Link$ is a loom space (\refthm{LinkIsLoom}), 
by \refdef{Loom}\refitm{Tet}, we have that $R$ is contained in a tetrahedron rectangle $R'$.  
Let $\lambda_i = \ell_i \cap R'$.
Let $\lambda^j$ be the Hausdorff limit of the (closures of the) $\lambda_i$.
Thus $\lambda^j$ is either one leaf, or two leaves and a cusp, of $F \cap \closure{R}'$.
Note that $\lambda^j$ lies in the Hausdorff limit $L$.
This completes the recursive step.

Now let $\lambda_\infty \subset L$ be the ascending union of the $\lambda^j$.
Note that $\lambda_\infty$ accumulates (in $\Disc$) on both $x$ and $y$.
The ascending union $\lambda_\infty$ either has no cusps, and so is contained in a leaf of $F$, or has a cusp, and so is contained in a bi-leaf of $F$.
Call this leaf or bi-leaf $\ell_\infty$.
Thus $\ell_\infty$ accumulates (in $\Disc$) on both $x$ and $y$.

We now show that $\ell_\infty \subset L$. 
Let $p$ be any point of $\ell_\infty$.
Let $q$ be a point of $\lambda_\infty$, chosen so that $q$ and $p$ are in the same component of $\ell_\infty$.
Let $[p, q]$ be the interval of $\ell_\infty$ with endpoints $p$ and $q$.  
Let $Q$ be a rectangle containing $[p, q]$ in its interior. 
Since the leaves $\ell_i$ approach $q$, they eventually enter $Q$ and thus approach $p$.
This proves that $\ell_\infty \subset L$.
\end{proof}

\begin{lemma}
\label{Lem:OneInteriorImpliesLeaf}
Suppose that $L$ contains a point $p$ of $\Link$.
Then $L$ contains the leaf $\ell_p$ of $F$ through $p$.
\end{lemma}

\begin{proof}
Fix a rectangle $P$ containing $p$.
The $\ell_i$ eventually pass through $P$.
Thus the Hausdorff limit $L$ (of their closures) contains the leaf $\lambda_p$ of $F \cap P$ containing $p$.
We are now done by \reflem{TwoImpliesLeaf}.
\end{proof}

\begin{figure}[htbp]
\labellist
\small\hair 2pt
\pinlabel $c$ [tr] at 178 32
\pinlabel $c'_i$ [r] at 70 48
\pinlabel $c_i$ [l] at 365 53
\pinlabel $\ell_i$ [l] at 460 64
\pinlabel $\ell_c$ [l] at 460 37
\pinlabel $C$ at 230 100
\pinlabel $D_i$ at 130 50
\endlabellist
\includegraphics[height=1.5in]{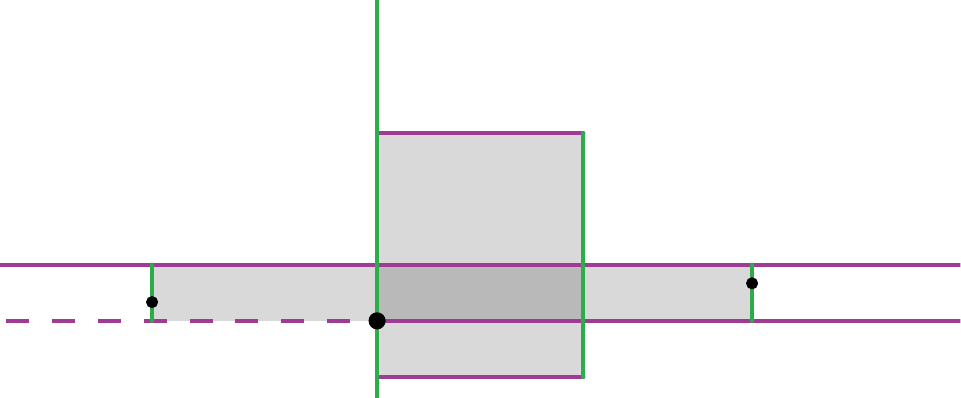}
\caption{One possibility when a sequence of leaves $\ell_i$ approaches a cusp leaf $\ell_c$.}
\label{Fig:CuspLeafImpliesBiLeaf}
\end{figure}
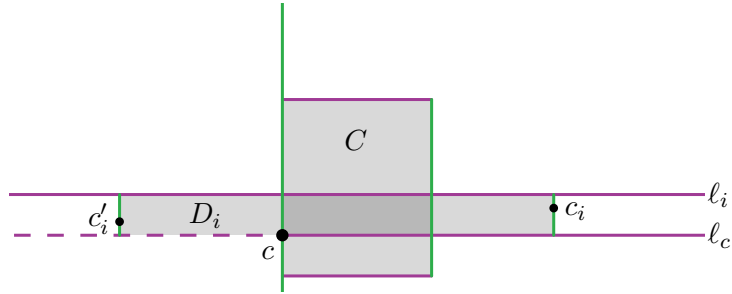

\begin{lemma}
\label{Lem:CuspLeafImpliesBiLeaf}
Suppose that $L$ contains a cusp leaf $\ell_c$ with cusp $c$.
Then $L$ contains a bi-leaf $k_c$ of $F$ containing $\ell_c$. 
\end{lemma}

\begin{proof}
By \refthm{LinkIsLoom} and \refdef{Loom}\refitm{Cusp}, we have a rectangle $C$ so that 
\begin{itemize}
\item
$c$ lies in the interior of one of the sides of $C$ and 
\item
the interior of $C$ meets the interior of $\ell_c$.  
\end{itemize}
All but finitely many of the $\ell_i$ meet $C$.  
See \reffig{CuspLeafImpliesBiLeaf}.
Passing to a subsequence, either 
\begin{itemize}
\item
the $\ell_i$ are all the same bi-leaf, call it $k_c$, containing $\ell_c$, or
\item
the $\ell_i$ approach $\ell_c$ from one side. 
\end{itemize}
Again, see \reffig{CuspLeafImpliesBiLeaf}.
In the first case, $L$ contains $k_c$ and we are done.
In the second case, \refdef{Loom}\refitm{Cusp} gives a sequence of rectangles $D_i$ so that 
\begin{itemize}
\item
a side of $D_i$ is contained in $\ell_i$,
\item
the opposite side, $s_i$ say, of $D_i$, contains $c$ in its interior, and 
\item
$D_i$ is maximal with respect to the previous two properties.
\end{itemize}
Choose a point $p$ in $\ell_c \cap C$.
Consider the two staircases for $p$ that meet $C$.
Let $m_p$ be the leaf containing $p$ and crossing $\ell_c$. 
By the astroid lemma~\cite[Lemma~4.10]{SchleimerSegerman24}, the projections of the cusps of these staircases to $m_p$ accumulate at $p$.
Thus the two sides of $D_i$ in the same foliation as $m_p$ contain cusps, say $c_i$ and $c'_i$.
See \reffig{CuspLeafImpliesBiLeaf}.

Applying the astroid lemma again, the projections of $c_i$ and $c'_i$ to the other axes of the staircases exit the ends of a bi-leaf with centre $c$.
We deduce that the ascending union of the sides $s_i$ is a bi-leaf, say $k_c$.
Thus $k_c$ contains $\ell_c$ and is contained in $L$.
\end{proof} 

We now combine the lemmas to prove the proposition.

\begin{proof}[Proof of \refprop{LimitOfLeaves}]
By \reflem{OneInteriorImpliesLeaf}, the Hausdorff limit $L$ cannot be a single point of $\Link$.
If $L$ is a single point of $\Circle$, we are done.
If $L$ contains at least two points of $\Disc$ then $L$ contains a leaf, by \reflem{TwoImpliesLeaf}.

Suppose that $L$ contains exactly one leaf $\ell$.
By \reflem{CuspLeafImpliesBiLeaf}, we deduce that $\ell$ is a non-cusp leaf.
Thus $\ell = L \cap \Link$ by \reflem{OneInteriorImpliesLeaf}.
Suppose that $x$ is a point of $L - \ell$.
Let $y$ be any point in $\ell$.
By \reflem{TwoImpliesLeaf}, we find that $x$ lies in the closure of ${\ell}$.
Thus in this case, $L$ is the closure of ${\ell}$ as desired.

Suppose that $L$ contains at least two distinct leaves $\ell$ and $\ell'$.
Pick $x$ and $y$ to be points in the interiors of $\ell$ and $\ell'$, respectively.
By \reflem{TwoImpliesLeaf}, we deduce that there is a cusp $c$ so that $k_c = \ell \cup \{c\} \cup \ell'$ is a bi-leaf. 
Suppose now that $z$ is a point of $L - k_c$.
By \reflem{TwoImpliesLeaf}, we find that $z$ lies in the closure of ${k}_c$.
Thus in this case, $L$ is the closure of ${k}_c$ as desired.
\end{proof}

\subsection{Neighbourhood bases}

In \reflem{HalfDiscNeighBasis} we prove that the half-discs in $\Disc$ give neighbourhood bases for the points of $\Circle$. 
For a cusp $c$ the boundaries of these half-discs have more and more corners as they nest.
However, for non-cusps, a single leaf suffices.

\begin{proposition}
\label{Prop:SingleLeaf}
Suppose that $x$ is a point of $\Circle - \Delta_\calV$.
Then there is a sequence of non-cusp leaves $(\ell_i)$ so that the half-discs $H(\ell_i)$ give a neighbourhood basis for $x$ in $\Disc$.
Furthermore, if $x$ is the end of a leaf $m$ then the $\ell_i$ all cross $m$.
\end{proposition} 

\begin{proof}
Suppose first that $x$ is an ideal point of a leaf $m$.
By \reflem{Extension} we may pick a sequence of points $p_k \in m$ converging to $x$. 
Let $\ell_k$ be the leaf of the opposite foliation to $m$ so that $m \cap \ell_k = p_k$.
Since non-cusp leaves are dense, we may perturb the $p_k$ to ensure that the leaves $\ell_k$ are not cusp leaves.
Pass to a subsequence so that the closures of the $\ell_k$ Hausdorff converge to a set $L$.
Note that $x$ lies in $L$.
By \refprop{LimitOfLeaves} and the fact that $x$ is not a cusp, we have that $L = \{ x \}$. 
Fix any oriented embedded polygonal line $\rho$ so that $H(\rho)$ is a neighbourhood of $x$.
Since the $\ell_i$ Hausdorff converge to $x$, all but finitely many are contained in $H(\rho)$.
Thus, for sufficiently large $i$, the neighbourhood $H(\ell_i)$ (with $\ell_i$ oriented appropriately) is contained in $H(\rho)$.

Suppose now that $x$ is a singleton.
Fix a reference point $w \in \Circle - \{x\}$.
Choose a sequence $\rho_i$ so that $H(\rho_i)$ is a neighbourhood basis of $x$.
After passing to a subsequence, none of the $H(\rho_i)$ contain $w$.
Thus there is some leaf $\ell_i$ containing a segment of $\rho_i$ so that $H(\ell_i)$ contains $x$ and not $w$.
By \refrem{NonCuspLeavesDense} non-cusp leaves are dense.
So we may replace $\ell_i$ by a nearby non-cusp leaf that again separates $x$ from $w$.

As before, we pass to a subsequence so that the closures of the $\ell_i$ Hausdorff converge to a set $L$.
Note that $L$ contains $x$.
Because $x$ is a singleton, it is not contained in the closure of any bi-leaf or non-cusp leaf.
The remainder of the argument is as in the first case.
\end{proof}

\begin{definition}
\label{Def:Sector}
Suppose that $p \in \Link \cup \Delta_\calV$.
Suppose that $S^\circ$ is a component of $\Disc$ minus the (closures of the) leaves at $p$.
We call the closure of $S^\circ$ a \emph{sector} at $p$.
Two sectors at $p$ are \emph{adjacent} if they share a point other than $p$.
\end{definition}

\begin{lemma}
\label{Lem:NoneShallPass}
Suppose that $S$ is a sector based at a cusp $d$.
Suppose that $C$ is a clamshell boundary associated to a cusp $c$.
Suppose that $c$ is not in $\bdyi S$.
Then we have
\begin{enumerate}
\item
\label{Itm:MeetsOne}
$\bdyi C$ meets precisely one of the two components of $\Circle - \bdyi S$.
\item
\label{Itm:NoBoundary}
$\bdyi C$ does not contain $d$, nor the thorns at the ends of the bounding leaves of $S$.
\end{enumerate}
\end{lemma}

\begin{proof}
The cusp $c$ lies in one of the two components.
Suppose that $\ell$ and $m$ are the boundary leaves of $S$.
Breaking symmetry, suppose that $\ell$ separates $c$ from $S$.
Any cusp leaf of $c$ crossing $\ell$ cannot cross $m$.
This proves \refitm{MeetsOne}.
We have that $c \neq d$, which implies \refitm{NoBoundary}.
\end{proof}

\begin{lemma}
\label{Lem:IVT}
Suppose that $\ell$ is a leaf and $x$ is a singleton.
Suppose that either
\begin{itemize}
\item $\ell$ is a non-cusp leaf or
\item $\ell$ is a cusp leaf on the boundary of a sector $S$ containing $x$.
\end{itemize}
Then there is a cusp $c$ with cusp leaves $m$, $m'$, and $m''$ so that:
\begin{itemize}
\item $m$ crosses $\ell$,
\item $m$ and $m'$ form a bi-leaf, as do $m$ and $m''$, and
\item exactly one of $m'$ or $m''$ separates $\ell$ from $x$. 
\end{itemize}
\end{lemma}

\begin{proof}
Breaking symmetry, suppose that $\ell$ belongs to $F^\calV$.
Breaking symmetry again, suppose that $x$ is to the right of $\ell$.
Let $R_\ell$ be the set of rays in non-cusp leaves of $F_\calV$ emanating from $\ell$ to the right.
Let $C_\ell$ be the closure of the endpoints at infinity $\bdyi R_\ell$. 
Note that $C_\ell$ is a Cantor set; the argument is similar to that in~\cite[Lemma~8.24]{FSS22}. 
Let $U$ be the complementary component of $C_\ell$ in $\Circle$ which contains $x$.
Then there is a unique cusp $c$ with the following properties.
\begin{itemize}
\item $c$ lies in $U$ and
\item $c$ has a cusp leaf $m$ that crosses $\ell$.
\end{itemize}
(When $\ell$ is a cusp leaf, this is where we use the sector $S$.)
Let $m'$ and $m''$ be the cusp leaves at $c$ whose endpoints are the endpoints of $U$.
\end{proof}

For a picture in the case where $\ell$ is a non-cusp leaf, see \reffig{IVT}.

\begin{corollary}
\label{Cor:SingletonSingleLeaf}
With hypotheses as in \refprop{SingleLeaf}, if $x$ is a singleton then we may choose the sequence $(\ell_i)$ to lie in either foliation. 
\end{corollary}

\begin{proof}
Let $(\ell_i)$ be the sequence provided by \refprop{SingleLeaf}.
If the $\ell_i$ lie in both foliations infinitely often then we are done.
Breaking symmetry, suppose that all but finitely many of the $\ell_i$ belong to $F^\calV$.
For each $\ell_i$, \reflem{IVT} gives a cusp $c_i$.
Let $S_i$ be the sector at $c_i$ which contains $x$.
Let $S'_i$ be the sector adjacent to $S_i$ across a leaf of $F^\calV$.
Let $m'_i$ be the bi-leaf boundary of $S_i \cup S'_i$.
Let $m_i$ be a non-cusp leaf sufficiently close to $m'_i$.
Note that $H(m_i)$ is contained in $H(\ell_i)$.
Then $(m_i)$ is the desired sequence of non-cusp leaves.
\end{proof}

\begin{lemma}
\label{Lem:Perturb}
Every non-cusp leaf $\ell$ is a limit (on either side) of non-cusp leaves $(\ell_k)$ in the same foliation.
Furthermore, the ideal points of $\ell_k$ tend to those of $\ell$ in $\Circle$.
\end{lemma}

\begin{proof}
By \refprop{SingleLeaf} we can pick sequences of non-cusp leaves $(m_k)$ and $(m'_k)$ exiting the two ends of $\ell$.
Fix $k$.
Let $L_k$ be the segment of $\ell$ between $m_k$ and $m'_k$.
Since $\Link$ is a loom space, $L_k$ is contained in a rectangle $R_k$.
By \refrem{NonCuspLeavesDense} non-cusp leaves are dense.
So we can find a non-cusp leaf $\ell_k$ crossing $R_k$ (on either side of $\ell$) which forms a rectangle with $\ell$, $m_k$, and $m'_k$.
\end{proof}

\begin{remark}
\reflem{Perturb} does not hold for cusp leaves.
It holds on the ``non-cusp side'' of a bi-leaf.  
However we will not use this fact.
\end{remark}

\subsection{Hausdorff limits of skeletal rectangles}

The goal of this section is to prove the following.

\begin{proposition}
\label{Prop:LimitOfRectangles}
Suppose that $(R_i)$ is a sequence of skeletal rectangles in $\Disc$.
Suppose that the closures $\closure{R}_i$ Hausdorff converge to a set $L$.  
Then $L$ is one of the following:
\begin{itemize}
\item
a point of $\Circle$,
\item
the closure of a leaf (cusp or non-cusp),
\item
the closure of a bi-leaf, or
\item
the closure of a skeletal rectangle.
\end{itemize}
The last possibility can only occur if the sequence is eventually constant.
\end{proposition}

We will require the following version of~\cite[Definition~4.20]{SchleimerSegerman24}.

\begin{definition}  
\label{Def:Spans}
Suppose that $Q$ and $R$ are rectangles in $\Link$.  
We say that $Q$ \emph{south-north spans} $R$ if $Q \cap R$ contains a leaf of $R \cap F^\calV$.
We define \emph{west-east spans} similarly, using $F_\calV$. 
\end{definition}

\begin{proof}[Proof of \refprop{LimitOfRectangles}]
Throughout we use the fact that $\Link$ is a loom space (\refthm{LinkIsLoom}).
If $L$ is a single point of $\Circle$ then we are done.

\begin{claim*}
Suppose that $L$ contains two distinct points $x$ and $y$ of $\Circle$.
Then $L$ contains a point of $\Link$.
\end{claim*}

\begin{proof}
We use an (easier) version of the proof of \reflem{TwoImpliesLeaf}.
Pick $U$ and $V$ basic open sets about $x$ and $y$, respectively, having disjoint closures.
Passing to a subsequence of $(R_i)$ there are arcs, rays, or lines $m_x$ and $m_y$ in the boundaries of $U$ and $V$, respectively, crossed by all of the rectangles $R_i$.
Let $Q_i \subset R_i$ be the maximal subrectangle having one side in $m_x$ and one in $m_y$.
Let $Q \subset \Link$ be the smallest rectangle containing all of the $Q_i$.
Breaking symmetry, we may assume that all of the $Q_i$ south-north span $Q$.
\refdef{Loom}\refitm{Tet} gives us a tetrahedron rectangle $R$ containing $Q$.
Thus the Hausdorff limit of the sequence $(\closure{Q}_i)$ contains a non-trivial segment of a leaf of $F^\calV \cap \closure{R}$.
Thus $L$ contains a point of $\Link$. 
\end{proof}

Suppose that the limit $L$ contains $P$, a rectangle in $\Link$.
Shrinking $P$ slightly, and passing to a subsequence, we may assume that all of the $R_i$ contain $P$.
Now~\cite[Lemma~4.16]{SchleimerSegerman24}
tells us that there are only finitely many skeletal rectangles containing $P$.
Thus $L$ is the closure of a skeletal rectangle $R_\infty$ and $R_i=R_\infty$ for sufficiently large $i$.

We are left with the case that $L$ contains a point $p$ of $\Link$, but does not contain a rectangle.
Let $P$ be any rectangle containing $p$.
We fix a homeomorphism $f_P \from (0, 1)^2 \to P$ sending the vertical and horizontal foliations to $F^\calV$ and $F_\calV$, respectively. 
Passing to a subsequence, we may assume that all of the rectangles $R_i$ meet $P$.
Since $L$ contains no rectangle, we deduce that either the widths or the heights of the preimages $f_P^{-1}(R_i)$ tend to zero.
Breaking symmetry, we may assume that the widths tend to zero.
It follows that the $R_i$ south-north span $P$.

Let $k^W_i$ be the bi-leaf containing the west side of $R_i$ whose divider (see \refdef{BiLeaf}) meets the closure of $R_i$.
We define $k^E_i$ similarly.
Let $L^W$ and $L^E$ be the Hausdorff limits of the sequences $(\closure{k^W_i})$ and $(\closure{k^E_i})$, respectively.
Since the widths of the $f_P^{-1}(R_i)$ tend to zero, we find that $L^W$ and $L^E$ contain $p$.
It follows from \refprop{LimitOfLeaves} that $L^W$ and $L^E$ contain the closure of $\ell_p$, the leaf of $F^\calV$ through $p$.

There are now two subcases, as $\ell_p$ is not, or is, a cusp leaf.
Suppose that $\ell_p$ is not a cusp leaf.
By \refprop{LimitOfLeaves} the limits $L^W$ and $L^E$ both equal the closure of $\ell_p$.
Thus $L$ is a subset of the closure of $\ell_p$.
To prove that $L$ is equal to the closure of $\ell_p$, we argue as follows.
Consider the four staircases for $p$.
The rectangles $R_i$ are contained in the union of these staircases.
Since the widths of the $R_i$ go to zero, the astroid lemma~\cite[Lemma~4.10]{SchleimerSegerman24} implies that (the projections of) the north and south sides of the $R_i$ exit the ends of $\ell_p$.
(See \reffig{Astroid} for an example.)
Thus $L$ is the closure of $\ell_p$.

In the other subcase $\ell_p$ is a cusp leaf.
Suppose that $c$ is the cusp of $\ell_p$.
Breaking symmetry, suppose that $c$ is the southern ideal point of $\ell_p$. 
After passing to a subsequence of the $R_i$ there are two possibilities, as follows.
\begin{itemize}
\item
The cusp $c$ lies on the south side of all of the rectangles $R_i$.
By an argument similar to the one above (applied only to the north sides of the $R_i$), the astroid lemma~\cite[Lemma~4.10]{SchleimerSegerman24} implies that $L$ is the closure of $\ell_p$ and we are done.
\item
The cusp $c$ does not lie in the south side of any of the $R_i$.
Breaking symmetry again, and passing to a further subsequence, we may assume that the rectangles $R_i$ accumulate on $\ell_p$ from the west.
Consider the staircase $\Gamma$ for $p$ so that $p$ is its north-east point.
So $c$ lies in the eastern boundary of $\Gamma$.
Let $\ell_c$ be the cusp leaf lying in the boundary of $\Gamma$ that emanates southward from $c$.
We deduce that $k_c = \ell_p \cup \{c\} \cup \ell_c$ is a bi-leaf.
We next prove that $L$ equals the closure of $k_c$.

\begin{figure}[htbp]
\labellist
\small\hair 2pt
\pinlabel $p$      [l] at 585 578
\pinlabel $c$      [l] at 585 267
\pinlabel $\ell_p$ [l] at 583 410
\pinlabel $\ell_c$ [l] at 583 130
\pinlabel $R_i$        at 341 412
\pinlabel $c_i$    [b] at 402 227
\endlabellist
\includegraphics[height=0.7\textwidth]{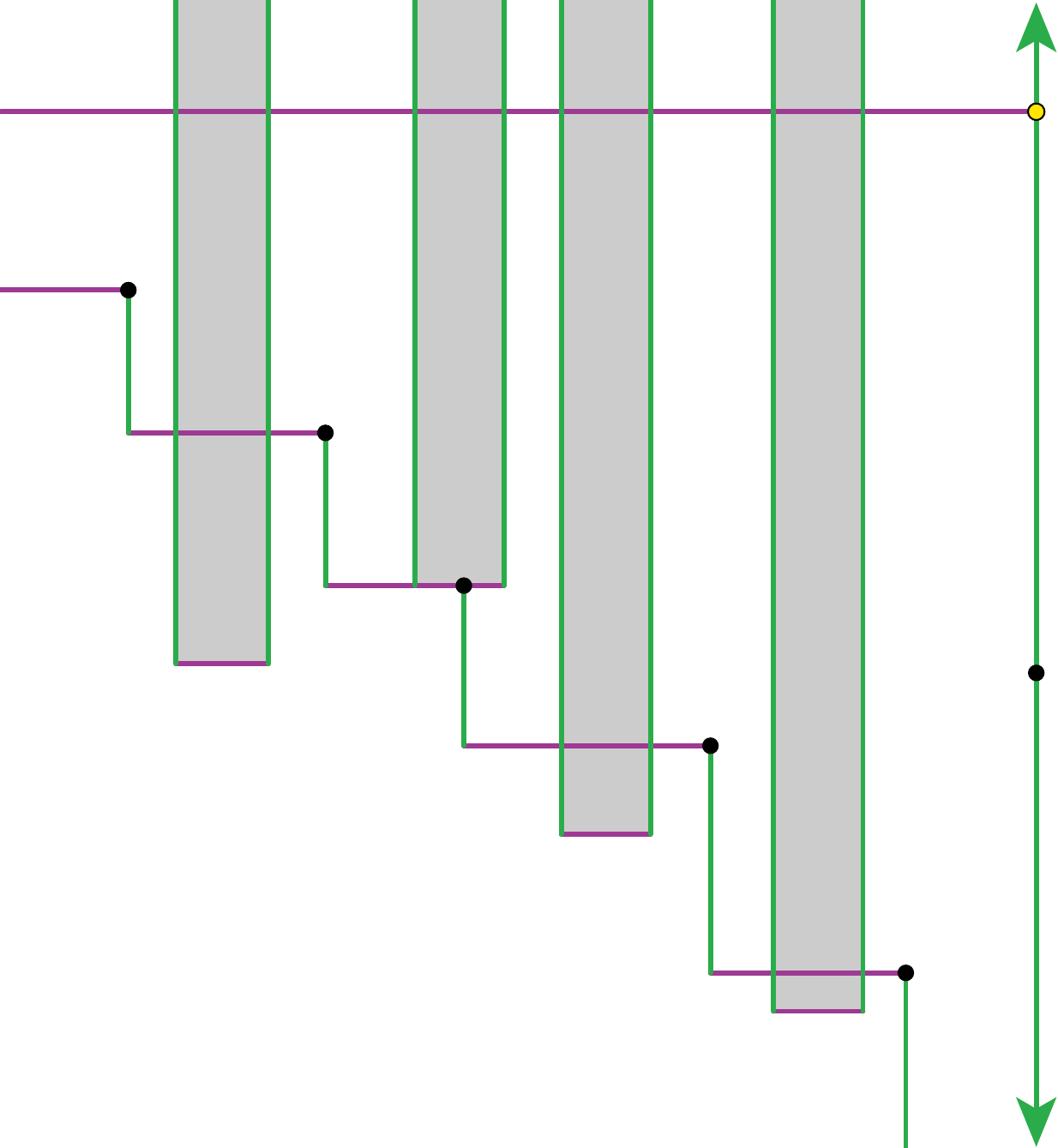}
\caption{A sequence of skeletal rectangles $R_i$, accumulating on $p$ from the west and thus meeting the staircase $\Gamma$. 
We do not draw most of the cusps incident to the $R_i$.}
\label{Fig:SneakyRectangles}
\end{figure}
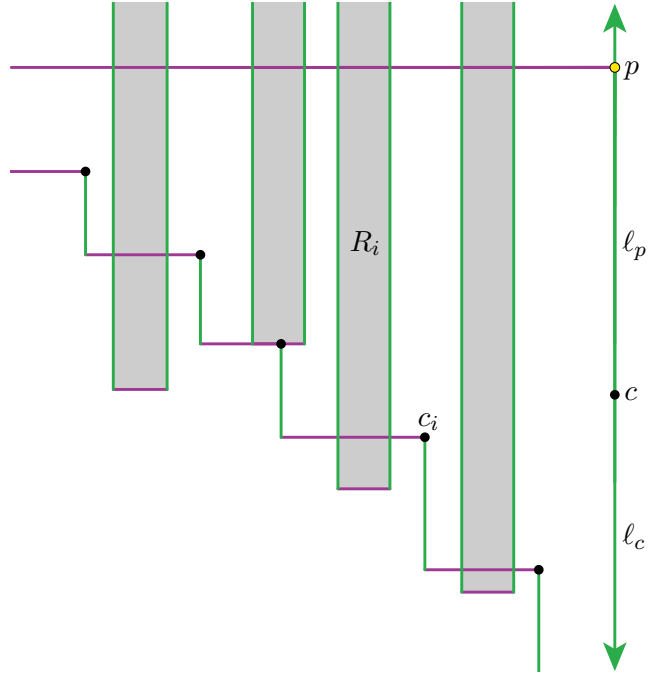

Recall that the skeletal rectangles $R_i$ accumulate on $p$ from the west.
Thus, after passing to a subsequence, they all meet $\Gamma$.
See \reffig{SneakyRectangles}.
Let $s_i$ be the southern side of the rectangle $R_i \cap \Gamma$.
Thus $s_i$ meets a (westward) cusp leaf emanating from a cusp $c_i$ in the south-west boundary of $\Gamma$.
The astroid lemma~\cite[Lemma~4.10]{SchleimerSegerman24} tells us that the (eastward projections of the) cusps $c_i$ eventually land in $\ell_c$, and then exit to the south.
We deduce that the $R_i$ accumulate on the cusp $c$ from the west.
Arguing as above we deduce that the Hausdorff limits $L$, $L^W$, and $L^E$ all equal the closure of $k_c$, as desired.
\end{itemize}

This completes the proof of \refprop{LimitOfRectangles}.
\end{proof}

\section{Geometrically finite convergence action}
\label{Sec:GeomFinite}

Over the next several sections we prove the various parts of our main result, \refthm{Convergence}.
Here we give the remaining relevant definitions.

\subsection{Definitions}

The notion of a \emph{convergence group} acting on the two-sphere goes back at least to Gehring and Martin~\cite{GehringMartin}.  
The ideas were further developed by many others in the context of groups acting on more general compacta, especially by Tukia~\cite{Tukia94}, Freden~\cites{Freden95}{Freden97}, and Bowditch~\cites{Bowditch99}{Bowditch98}.

Recall that $Z$ is a \emph{compactum} if it is compact and Hausdorff. 
Recall that $Z$ is \emph{perfect} if it has no isolated points. 

\begin{definition}
\label{Def:Collapsing}
Suppose that $Z$ is a compactum with at least three points.
Suppose that $(g_i)_{i\in\NN}$ is a sequence of homeomorphisms of $Z$ to itself.
Suppose that $x, y \in Z$ are points (not necessarily distinct).
Suppose that the restrictions $g_i|(Z - \{x\})$ converge uniformly on compact sets to the constant map $(Z - \{x\}) \to \{y\}$.
Then we call $(g_i)$ a \emph{collapsing sequence}.
Also, we call $x$ a \emph{repelling} point, and we call $y$ an \emph{attracting} point, for the sequence. 
\end{definition}

\begin{remark}
Note that collapsing sequences have unique repellers and attractors.
\end{remark}



\begin{definition}
\label{Def:Convergence}
Suppose that $Z$ is a compactum with at least three points.
Suppose that $G$ is a countable group acting by homeomorphisms of $Z$.  
Then $G$ is a \emph{convergence group} if every infinite sequence of distinct elements of $G$ has a subsequence which is a collapsing sequence.  
We may also say $G\acts Z$ is a \emph{convergence action}.
\end{definition}

\begin{remark}
\label{Rem:FiniteIndex}
Suppose that $G$ acts on $Z$.
Suppose that $H < G$ is a subgroup of finite index. 
Suppose that the induced action of $H$ on $Z$ is a convergence action. 
Then the action of $G$ on $Z$ is a convergence action.
\end{remark}

\begin{definition}
\label{Def:CLP}
Suppose that $G \acts Z$ is a convergence action.
Suppose that $x, y, z \in Z$ are points, with $y \neq z$.
Suppose that $(g_i)_{i \in \NN}$ is a collapsing sequence with $x$ and $y$ as its repeller and attractor, respectively.
Suppose that $(g_i(x))$ converges to $z$.
Then we call $x$ a \emph{conical limit point} of the action.
\end{definition}

\begin{definition}
Let $f \from Z \to Z$ be a homeomorphism.
We say that $f$ is \emph{loxodromic} if it is of infinite order and it has precisely two fixed points.    
\end{definition}

\begin{definition}
\label{Def:Parabolic}
Suppose that $G \acts Z$ is a convergence action.  
Suppose that $x \in Z$ has infinite stabiliser, but no loxodromic fixes $x$. 
Then we call $x$ a \emph{parabolic point} of the action.

Furthermore, suppose that the stabiliser of $x$ acts cocompactly on $Z - \{x\}$.
Then we call $x$ a \emph{bounded} parabolic point.
Finally, we call the stabiliser of any parabolic point a \emph{maximal parabolic} subgroup for the action.
\end{definition}

\begin{definition}
\label{Def:GeoFinConv}
Suppose that $G \acts Z$ is a convergence action. 
Suppose that every point of $Z$ is either a conical limit point or is a bounded parabolic point. 
Then we call the action \emph{geometrically finite}. 
\end{definition}

The terminology ``geometrically finite'' is taken from kleinian groups.
The definition we are using was abstracted by Bowditch from results on geometrically finite kleinian groups due to Beardon and Maskit~\cite{BeardonMaskit74}.  
In particular, following~\cite[Theorem~2]{BeardonMaskit74} (see also~\cite[Section~4]{Bowditch93}) we have the following.

\begin{theorem}
\label{Thm:HyperbolicGeomFinite}
Suppose that $M$ is a finite-volume cusped hyperbolic three-manifold.
Then the induced action of $\pi_1(M)$ on $\bdyi \HH^3$ is a geometrically finite convergence action. 
\end{theorem}
  
A \emph{group pair} $(G, \calP)$ is a group $G$ and a collection $\calP$ of subgroups of $G$ that is closed under conjugation.

For completeness, and in order to state Yaman's theorem (\refthm{Yaman}), we give the following definition. 

\begin{definition}
Suppose that $(G, \calP)$ is a group pair. 
Suppose that $X$ is a connected simplicial graph. 
Suppose that $G \acts X$ acts via graph automorphisms.
Suppose, moreover, that: 
\begin{enumerate}
\item 
$X$ is Gromov hyperbolic,
\item 
$X$ is \emph{fine}: for any edge $e$ and any $n>0$, there are finitely many embedded edge-circuits of length $n$ containing $e$,
\item 
the action $G \acts X$ is cocompact, 
\item 
$\calP$ is the set of infinite stabilisers of vertices,
\item
every $P \in \calP$ is finitely generated, and
\item 
edge stabilisers are finite.
\end{enumerate}
Then we call $(G, \calP)$ \emph{hyperbolic in the sense of Bowditch}.
\end{definition}


Here then is our statement of Theorem~0.1 of~\cite{Yaman04}.

\begin{theorem}
\label{Thm:Yaman}
Suppose that $Z$ is a perfect metrisable compactum, with at least three points.
Suppose that $G$ acts on $Z$ via a geometrically finite convergence action.
Let $\calP$ be the set of maximal parabolics.
Suppose that every $P$ in $\calP$ is finitely generated.
Suppose that $\calP / G$ is finite.
Then $(G, \calP)$ is hyperbolic in the sense of Bowditch.

Suppose that $Z'$ is another such compactum, equipped with another such action by $G$, having the same set of maximal parabolics.
Then there is a $G$--equivariant homeomorphism from $Z$ to $Z'$. \qed
\end{theorem}

\section{Convergence action}
\label{Sec:ObtainingConvergence}

In this section we prove the following. 

\begin{proposition}
\label{Prop:Convergence}
Suppose that $M$ is a three-manifold with boundary.
Suppose that $\calV$ is a locally veering triangulation on $M$.
The action of $\pi_1(M)$ on the veering two-sphere $\Sphere$ is a convergence action.  
\end{proposition}

The remainder of this section is devoted to the proof.

\begin{remark}
\label{Rem:Cover}
We may replace the hypothesis that $\calV$ is a locally veering triangulation by the hypothesis that $\calV$ is a transverse veering triangulation.
We do this by taking the appropriate number of degree-two covers and appealing to \refrem{FiniteIndex}. 
\end{remark}

Fix an infinite sequence of group elements $(g_i)$.  
We must find a collapsing subsequence. 

In the following lemma, closures are taken in $\Disc$.

\begin{lemma}
\label{Lem:Diagonal}
There is a subsequence of $(g_i)$ that, after reindexing, gives the following.
\begin{enumerate}
\item Suppose that $c$ is a cusp. 
Then the sequence $(g_i(c))$ converges in~$\Circle$. 
\item Suppose that $k$ is a bi-leaf.
Then the sequence $\left(g_i\left(\closure{k}\right)\right)$ Hausdorff converges in~$\Disc$.
\item Suppose that $R$ is a skeletal rectangle.
Then the sequence $\left(g_i\left(\closure{R}\right)\right)$ Hausdorff converges in~$\Disc$.
\end{enumerate}
\end{lemma}

\begin{proof}
There are only countably many cusps, bi-leaves, and skeletal rectangles~\cite[Corollary~4.25]{SchleimerSegerman24}.
For each there is a subsequence that (Hausdorff) converges by the compactness of $\Disc$ (by \refthm{Disc}).
Thus, passing to subsequences countably many times and taking a diagonal gives the result.
\end{proof}

Fix any cusp $c_0 \in \Delta_\calV$.
Let $A = [\lim_i g_i(c_0)] \subset \Circle$ be the associated decomposition element, as given in \refdef{DecompositionElementsCircle}.  
We now show that $A$ does not depend on the choice of $c_0$.

\begin{lemma}
\label{Lem:WellDef}
For any cusp $c$, the point $\lim_i g_i(c)$ lies in $A$.
\end{lemma}


\begin{proof}
Recall that $\cover{\calV}$ is the induced ideal triangulation of the universal cover.
Since $\cover{\calV}$ is connected there is a chain of tetrahedra glued along faces which joins $c_0$ to $c$.
Recall that the veering tetrahedra of $\cover{\calV}$ are in bijection with the tetrahedron rectangles of $\link(\calV)$. 
Thus there is a finite sequence $(R_j)_{j=0}^n$ of skeletal rectangles satisfying the following:
\begin{itemize}
\item
$\closure{R}_0$ meets $c_0$;
\item
$R_j$ meets $R_{j+1}$ (along a face rectangle) for each $j<n$; and
\item
$\closure{R}_n$ meets $c$.  
\end{itemize}
For each $j$, \reflem{Diagonal} implies that the sequence $g_i(\closure{R}_j)$ Hausdorff converges to some set $L_j$. \refprop{LimitOfRectangles} enumerates the possibilities for $L_j$.
Since the $g_i$ are all distinct, the sequence $g_i(R_j)$ is not eventually constant 
and the limit $L_j$ is either a point of $\Circle$, the closure of a leaf, or the closure of a bi-leaf.  In any of these cases $L_j \cap \Circle$ lies in a decomposition element $A_j$.
Let $Q_j = R_j \cap R_{j+1}$.
Applying \refprop{LimitOfRectangles} to $g_i(\closure{Q_j})$, we deduce that $L_j$ meets $L_{j+1}$.  
Hence $A_j$ meets $A_{j+1}$.  
But decomposition elements meet if and only if they are equal.  
Thus $A_n = A_0 = A$. 
\end{proof}

Fix, once and for all, a \emph{reference point} $r \in \Circle$ chosen so that
$r$ is not a point of $A$, not a cusp, and not the endpoint of a leaf of either lamination.
Recall that $\calO_\calV$ is the given circular order on $\Circle$.  
Applying another diagonalization argument yields the following.

\begin{lemma}
\label{Lem:Direction}
By passing to a subsequence of $(g_i)$ we obtain the following property.  
For any pair of cusps $c, d \in \Delta_\calV$ the sequence 
\[
\calO_\calV( g_i(c), g_i(d), r )
\]
is eventually constant. \qed
\end{lemma}

This finishes the process of passing to subsequences.  
We show in several steps that $(g_i)$ is a collapsing sequence, and so obtain \refprop{Convergence}.  
The decomposition element $A$ will serve as the attracting point for $(g_i)$.  
The repelling decomposition element $B$ is described in the next subsection.

\subsection{The repelling point}
\label{Sec:Repell}

\begin{definition}
Let $a\in A$.  
The set $I(a)\subset \Delta_\calV$ consists of those cusps $c$ so that $\lim_ig_i(c) = a$.  
\end{definition}

\begin{remark}
\label{Rem:AttractorPartition}
By \reflem{WellDef}, the sets $I(a)$ partition the cusps $\Delta_\calV$.
\end{remark}

\begin{definition}
A \emph{collapsing interval} $J$ in $\Circle$ is a closed interval so that, for some $a$ in $A$, the restrictions $g_i|J$ converge uniformly to the constant map $J\to\{a\}$.
In this case we say that $J$ \emph{collapses to} $a$ and we say that $J$ is an $a$--\emph{interval}.
\end{definition}

\begin{lemma}
\label{Lem:Union}
Suppose that $I$ and $J$ are collapsing intervals.
Suppose that $I$ and $J$ intersect.
Then $I \cup J$ is a collapsing interval.
\end{lemma}

\begin{proof}
Suppose that $I$ collapses to $a$ and $J$ collapses to $b$.
Since $I$ and $J$ intersect, we deduce that $a = b$.
Choose a neighbourhood $U$ of $a$.
Thus there exists $i_0$ so that for all $i \geq i_0$ we have that $g_i(I) \subset U$.
We define $j_0$ similarly for $J$.
So for $i \geq \max\{i_0, j_0\}$ we have $g_i(I \cup J) \subset U$.
Thus $I \cup J$ converges uniformly to $a$.
However, there is no sequence of homeomorphisms of the circle to itself which converges uniformly to a constant map.
Thus $I \cup J$ is an interval and thus is an $a$--interval.
\end{proof}

\begin{lemma}
\label{Lem:a_interval}
Fix $a \in A$, and suppose $c, d \in I(a)$.
Then exactly one of $[c, d]^\acw$, $[d, c]^\acw$ is an $a$--interval.
\end{lemma}

\begin{proof}
If both intervals were $a$--intervals then, since their union is $\Circle$, we would contradict \reflem{Union}.

Let $(U_j)_{j\in\NN}$ be a nested neighbourhood basis of $a$, chosen so that $U_0$ misses the reference point $r$.  
By \reflem{Direction}, there is an $n_0$ so that for all $i\geq n_0$, we have $\calO_\calV(g_i(c), g_i(d), r) = \calO_\calV(g_{n_0}(c), g_{n_0}(d), r)$.  
Moreover for any $j$ there is an $n_j\geq n_0$ so that $\{g_i(c),g_i(d)\}\subset U_j$ whenever $i\geq n_j$.  
If $\calO_\calV(g_{n_0}(c), g_{n_0}(d), r) = 1$, then we have $g_i([c, d]^\acw)\subset U_j$ for all $i\ge n_j$ and thus $[c, d]^\acw$ is an $a$--interval.  
Otherwise $g_i([d, c]^\acw)\subset U_j$ for all $i\ge n_j$ and so $[d, c]^\acw$ is an $a$--interval.
\end{proof}

\begin{corollary}\label{Cor:ConnectedI(a)}
For any $a \in A$, the set $I(a)$ is the intersection of $\Delta_\calV$ with some connected (possibly empty) subset of $\Circle$. \qed
\end{corollary}

\begin{definition}
\label{Def:Parting}
We say $p \in \Circle$ is a \emph{parting} for $(g_i)$ if it does not lie in the interior of a collapsing interval.
The set $P$ of partings for the sequence $(g_i)$ is called the \emph{parting set}. 
\end{definition}

\begin{lemma}
\label{Lem:Parting}
The parting set $P$ is closed and non-empty.
\end{lemma}

\begin{proof}
From the definition, the complement of $P$ is open in $\Circle$.
For a contradiction, suppose that $P$ is empty.
Then $\Circle$ is covered by the interiors of collapsing intervals, hence by finitely many.
However, this contradicts \reflem{Union}.
\end{proof}

\begin{lemma}
\label{Lem:AtMostTwo}
For all $a \in A$, the intersection $I(a) \cap P$ has at most two points.
\end{lemma}

\begin{proof}
Suppose for contradiction that $p$, $q$, and $r$ are distinct cusps of $I(a) \cap P$, ordered anticlockwise.
By \reflem{a_interval} and \refdef{Parting} we have that $[p, q]^\acw$ is an $a$--interval.
Similarly $[q, r]^\acw$ is an $a$--interval.
Therefore $[p, r]^\acw$ is an $a$--interval that contains the parting point $q$, a contradiction.
\end{proof}

\begin{lemma}
\label{Lem:ClosedSubIntervals}
Suppose that $C$ is a component of $\Circle - P$.
Then every closed sub-interval of $C$ is a collapsing interval.
\end{lemma}

\begin{proof}
Suppose that $c$ and $d$ are points of $C$.
Let $J$ be the interval in $C$ with endpoints at $c$ and $d$.
Since there are no partings in $C$, the interval $J$ is covered by the interiors of collapsing intervals.
\reflem{Union} implies that $J$ is a collapsing interval.
\end{proof}

\begin{lemma}
\label{Lem:I(a)Component}
Suppose that, for some $a \in A$, the set $I(a)$ contains at least two points of $\Circle$.
Then there is a component $C$ of $\Circle - P$ whose closure contains $I(a)$.
Furthermore, $C \cap \Delta_\calV$ is contained in the set $I(a)$. 
\end{lemma}

\begin{proof}
Since $I(a)$ has at least two points, \reflem{a_interval} implies that $I(a)$ is countably infinite.
By \reflem{AtMostTwo}, the difference $I(a) - P$ is non-empty.
So, again using \reflem{a_interval}, there is a unique component $C$ of $\Circle - P$ containing $I(a) - P$.
By \reflem{Parting}, the component $C$ is an open interval.
Any points of $I(a) \cap P$ must be boundary points of $C$ by \reflem{a_interval} and \refdef{Parting}.
This gives the first half of the lemma.
The second half of the lemma follows from \reflem{ClosedSubIntervals}.
\end{proof}

\subsection{From the circle to the sphere}

Recall from \refdef{VeeringSphere} that $\Psi_\calV \from \Circle \to \Sphere$ is the quotient map collapsing decomposition elements.

\begin{lemma}
\label{Lem:AlmostCollapsing}
On $\Sphere - \Psi_\calV(P)$, the maps $(g_i)$ uniformly converge on compact sets to the constant map with image $A$.
\end{lemma}

\begin{proof}
Suppose that $K\subset \Sphere - \Psi_\calV(P)$ is compact.  
Then the preimage $\Psi_\calV^{-1}(K)$ is compact and misses the closed set $P$.  
In particular $\Psi_\calV^{-1}(K)$ is contained in finitely many closed intervals disjoint from $P$.
Each of the finitely many intervals is a collapsing interval by \reflem{ClosedSubIntervals}. 
So for any open neighbourhood $U$ of $A$ in $\Sphere$, and for all but finitely many $i$, the set $g_i(\Psi_\calV^{-1}(K))$ lies in $\Psi_\calV^{-1}(U)$.
Thus for all but finitely many $i$, the image $g_i(K)$ lies in $U$.
\end{proof}

\subsection{A decomposition element for the parting set}

We show in the rest of the section that there is a decomposition element $B$ which contains $P$.
(However, $P$ may be a proper subset of $B$.)  
The fact that $(g_i)$ is a collapsing sequence (for the action on $\Sphere$) then follows from \reflem{AlmostCollapsing}.
Our proof proceeds by cases, depending on how many partings there are.
We first prove some preliminary lemmas about the partition of $\Delta_\calV$ (\refrem{AttractorPartition}) into the sets $I(a)$.

\begin{lemma}
\label{Lem:OrdersAgree}
Suppose that $a_1$, $a_2$, and $a_3$ are distinct elements of $A$, and that $c_j \in I(a_j)$ for each $j$.  
Then $\calO_\calV(c_1, c_2, c_3) = \calO_\calV(a_1, a_2, a_3)$.  
\end{lemma}

\begin{proof}
By \refrem{Cover}, we have that $\calV$ is transverse veering. 
Thus the action $\pi_1(M)\acts\Circle$ preserves the circular order.  
\end{proof}

\begin{lemma}
\label{Lem:InaccessibleCusps}
Let $a \in A$.  Then $I(a)$ is either
\begin{itemize}
\item empty,
\item a singleton,
\item $J \cap \Delta_\calV$ where $J$ is a closed interval in $\Circle$, or
\item all of $\Delta_\calV$.
\end{itemize}
\end{lemma}
\begin{proof}
  
\begin{figure}[htbp]
\centering
\labellist
\small\hair 3pt
\pinlabel $C$ [r] at 3 140
\pinlabel $\ell$ [bl] at 56 120
\pinlabel $m$ [bl] at 28 110
\pinlabel $c$ [tr] at 34 34
\endlabellist
\includegraphics[width=0.6\textwidth]{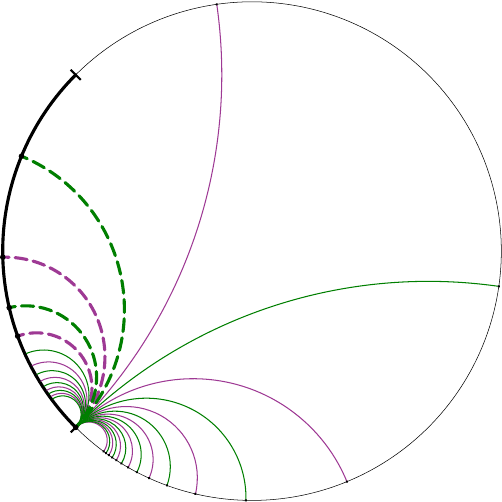}
\caption{The bi-leaves $\ell$ and $m$ are drawn with thick dashed lines.}
\label{Fig:CuspAtEndOfInterval}
\end{figure}
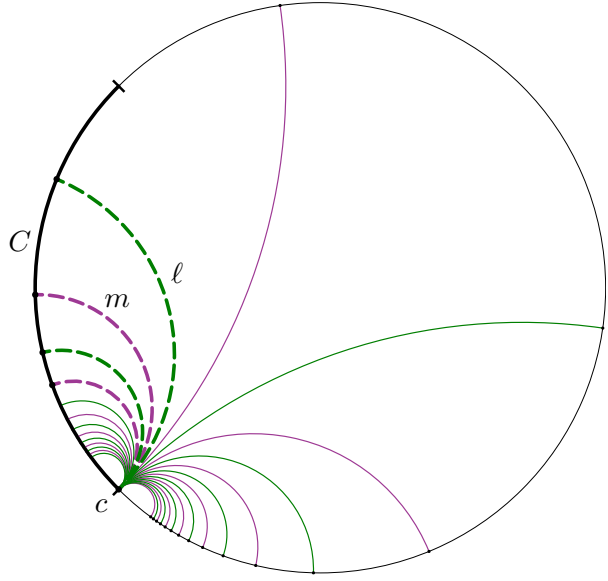
  
\refcor{ConnectedI(a)} tells us that $I(a) = C \cap \Delta_\calV$ where $C$ is a connected subset of $\Circle$.  
So $C$ is empty, a point, an interval (closed, open, or half open), or all of $\Circle$.
We claim that $\closure{C} - C$ does not meet $\Delta_\calV$.  
Given this claim, we may replace $C$ by its closure and the lemma follows.

So, to obtain a contradiction, suppose that $c\in \closure{C} - C$ is a cusp.
Since $c \notin C$, we have that $c \notin I(a)$.
Thus $c\in I(b)$ with $b \neq a$. 
Both $F^\calV$ and $F_\calV$ contain infinitely many bi-leaves with centre $c$ whose endpoints are contained in $C$.  
See \reffig{CuspAtEndOfInterval}.  
Fix bi-leaves $\ell$ of $F^\calV$ and $m$ of $F_\calV$, each containing $c$ and two points of $C$.  
By \reflem{Diagonal}, the Hausdorff limits $L = \lim_i g_i(\closure{\ell})$ and $M = \lim_i g_i(\closure{m})$ both exist.
The limits $L$ and $M$ must both contain $a$ and $b$.   
By \refprop{LimitOfLeaves}, and since $a \neq b$, the limits $L$ and $M$ are (closures of) non-cusp leaves or bi-leaves.  
By \refcor{OnlyAsymptotics}, since $L$ and $M$ have a common ideal point, they must be closures of bi-leaves $\ell_\infty$ and $m_\infty$, respectively, emanating from a common cusp and sharing a cusp leaf.
In particular, $\ell_\infty$ and $m_\infty$ are in the same foliation.

Since $\ell$ and $m$ are in opposite foliations, so are $g_i(\ell)$ and $g_i(m)$.
From the final statement of \refprop{LimitOfLeaves}, we deduce that $\ell_\infty$ and $m_\infty$ are in different foliations, a contradiction.
\end{proof}

\reflem{OrdersAgree} and \reflem{InaccessibleCusps} have some useful consequences.  

\begin{lemma}
\label{Lem:PartingsDecomposeDeltaM}
Suppose that $I(a)$ contains at least two points.
Then there is a component $C$ of $\Circle - P$ so that $I(a) = \closure{C} \cap \Delta_\calV$. 
\end{lemma}

\begin{proof}
Since $I(a)$ contains at least two points we are in the third or fourth case of \reflem{InaccessibleCusps}.
Suppose that $I(a) = \Delta_\calV$.
In this case, \reflem{a_interval} implies that $P$ is a single point.
Thus $\closure{C} = \Circle$ and we are done.

Suppose instead that $I(a) = J \cap \Delta_\calV$ where $J$ is a closed interval.
We must show that $J \cap P = \bdy J$.
We first claim that $P$ misses the interior of $J$.
For a contradiction, suppose that $p \in P$ lies in the interior of $J$.
Since cusps are dense, there are cusps of $I(a)$ on either side of $p$, contradicting \reflem{a_interval}.

Now let $C$ be the component of $\Circle - P$ containing the interior of $J$.
Let $q$ be an endpoint of $J$.
Suppose for a contradiction that $q$ lies in the interior of $C$.
Fix a closed sub-interval $I$ of $C$ containing $q$ in its interior.
The interval $I$ contains cusps of $I(a)$ since it meets the interior of $J$.
Also, $I$ contains cusps not in $I(a)$ (and so in some $I(b)$ for $b \neq a$, by \refrem{AttractorPartition}) by the definition of $J$ and because $I$ meets the interior of $\Circle - J$.
But by \reflem{ClosedSubIntervals} the interval $I$ is a collapsing interval.
This is a contradiction.
\end{proof}

\begin{corollary}
\label{Cor:NotCuspBefore}
Suppose that $P$ has at least two points.
Suppose that $p \in P$ is not an accumulation point of $P$.
Then $p$ is not a cusp.
\end{corollary}

\begin{proof}
Let $C$ and $C'$ be the components of $\Circle - P$ adjacent to $p$.
Since $P$ has at least two points, $C$ and $C'$ are disjoint.
By \reflem{ClosedSubIntervals}, both contain collapsing intervals.
By \reflem{PartingsDecomposeDeltaM}, there are points $a$ and $a'$ in $A$ so that $I(a) = \closure{C} \cap \Delta_\calV$ and $I(a') = \closure{C}' \cap \Delta_\calV$.
Therefore $a \neq a'$.
By \refrem{AttractorPartition}, we have that $I(a)$ and $I(a')$ are disjoint.
Note that $p$ is a point of $\closure{C} \cap \closure{C}'$.
If $p$ were a cusp then $p$ would lie in both $I(a)$ and in $I(a')$, a contradiction.
\end{proof}

\begin{lemma}
\label{Lem:SingletonGoesToCusp}
Suppose for some $a \in A$, the set $I(a)=\{c\}$ has a single element.  
Then $a$ is a cusp.
\end{lemma}

\begin{proof}
Suppose that $x \in \Circle - \{c\}$ is any point.  
Consider the interval $(c, x)^\acw$.  
Suppose for a contradiction that $I(b)$ meets $(c, x)^\acw$ for only finitely many of the $b$ in $A$.  
Thus $c$ is in the closure of one of the $I(b)$.  
\reflem{InaccessibleCusps} (applied to $I(b)$) implies that $a = b$.
Thus $I(a)$ meets $(c, x)^\acw$.  
Thus $I(a)$ is not a singleton, a contradiction.  
We deduce that $A$ is infinite.  
By \refdef{DecompositionElementsCircle} we find that $A$ is the decomposition element for some cusp.
Note that, similarly to the above, the interval $(x, c)^\acw$ also meets infinitely many of the sets $I(b)$ for $b$ in $A$.  

Fix $a_0 \in A - \{a\}$, a thorn so that $I(a_0)$ is non-empty.
Fix $c_0 \in I(a_0)$.
Note that $c_0 \neq c$.
With $c_0$ replacing $x$ above, we find distinct $(a_k)_{k\in \ZZ}$ in $A$ so that for all $k$
\begin{itemize}
\item the set $I(a_k)$ is non-empty, and
\item there is some $c_k \in I(a_k)$ with $\calO_\calV(c, c_k, c_0) = \sign(k)$.
\end{itemize}

\reflem{OrdersAgree} now implies that $\calO_\calV(a, a_k, a_0) = \sign(k)$.
Two thorns of $A$ separate the remaining elements of $A$ into a finite set and an infinite set. 
See \reffig{CuspAtEndOfInterval}.
We deduce that $a$ is a cusp, as desired. 
\end{proof}

The following is immediate from \reflem{SingletonGoesToCusp} and \refdef{DecompositionElementsCircle}.
\begin{corollary}
\label{Cor:OneSingleton}
For at most one $a\in A$, the set $I(a)$ consists of a single element. \qed
\end{corollary}

Recall that if $\ell$ is an oriented leaf of either foliation of $\Link$ then $H(\ell)$ is the half-disc in $\Disc$ associated to $\ell$. 
See \refdef{HalfDisc}.

\begin{lemma}
\label{Lem:Case2}
Suppose that the parting set $P$ consists of exactly two points $p$ and $q$.  
Then $P$ is contained in a single decomposition element.
\end{lemma}

\begin{proof}
By \refcor{NotCuspBefore} neither $p$ nor $q$ is a cusp.  
By \refprop{SingleLeaf} there is a nested neighbourhood basis $\{ H(\ell_j) \}_{j \in \NN}$ of $p$. 
After discarding a finite number of basis elements, the closure of each $\ell_j$ connects a point of some $a$--interval to a point of some $b$--interval.
The Hausdorff limit $L_j = \lim_i g_i(\closure{\ell}_j)$ contains the points $a$ and $b$.  
By \refprop{LimitOfLeaves}, and since $a \neq b$, we find that $L_j$ is the closure of either a non-cusp leaf or a bi-leaf with endpoints $a$ and $b$.  
There is only one such, so the limit $L = L_j$ is independent of $j$. 
There is likewise a nested neighbourhood basis of $q$, say $\{H(\ell_j')\st j\in\NN\}$, where each $\ell_j'$ connects some $a$--interval to some $b$--interval.  
The limits $\lim_i g_i(\closure{\ell}_j')$ are also equal to $L$.

Breaking symmetry, suppose that the interior of $L$ lies in $F^\calV$.
Choose any non-cusp leaf $m$ in $F_\calV$ which crosses $L$.  
For each $j$, there is an integer $N(j)$ so that the leaves $g_i(\ell_j)$ and $g_i(\ell_j')$ cross $m$ whenever $i \geq N(j)$.  
It follows that the leaves $g_i^{-1}(m)$ cross $\ell_j$ and $\ell_j'$ for $i \geq N(j)$.  
In particular, after passing to a subsequence, the Hausdorff limit $M$ of the leaves $g_i^{-1}(\closure{m})$ must contain the partings $p$ and $q$.  
By \refprop{LimitOfLeaves} the interior of $M$ is therefore a non-cusp leaf or bi-leaf connecting $p$ to $q$.  
In either case $P$ is contained in a single decomposition element.
\end{proof}

\begin{lemma}
\label{Lem:NoCuspsAfter}
Suppose both $\closure{I(a)}$ and $\Circle - \closure{I(a)}$ have non-empty interior.  Then $a$ is not a cusp.
\end{lemma}

\begin{proof}
Suppose $\closure{I(a)}=[x, y]^\acw$.  We divide into two cases, depending on whether or not $x$ is a cusp. 

\begin{figure}[htbp]

\subfloat[Case: $x$ is not a cusp.]{
\labellist
\small\hair 2pt
\pinlabel {$\closure{I(a)}$} [r] at 1 125
\pinlabel {$x$} [br] at 38 215
\pinlabel {$y$} [tr] at 38 35
\pinlabel {$\ell$} [tl] at 75 187
\pinlabel {$u$} [tl] at 16 171
\pinlabel {$c'$} [br] at 60 227
\pinlabel {$v$} [b] at 84 240
\pinlabel {$c$} [b] at 114 248
\endlabellist
\includegraphics[width=0.43\textwidth]{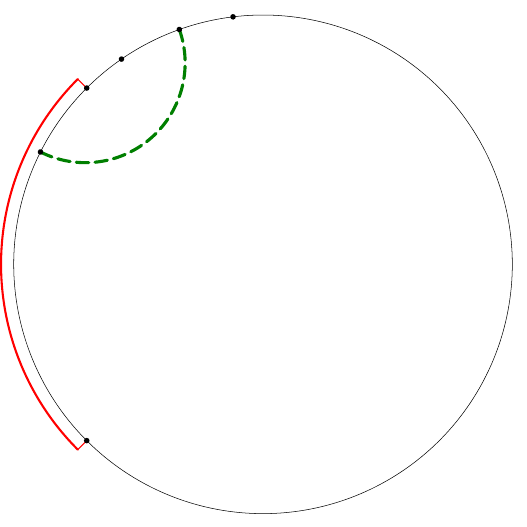}
\label{Fig:NoCuspsxNotCusp}
}
\qquad
\subfloat[The limit when $L$ is a leaf (as opposed to a bi-leaf).]{
\labellist
\small\hair 2pt
\pinlabel {$a$} [r] at 4 115
\pinlabel {$L$} [tl] at 110 147
\pinlabel {$b'$} [bl] at 158 244
\pinlabel {$z$} [bl] at 190 230
\pinlabel {$b$} [bl] at 228 192
\endlabellist
\includegraphics[width=0.43\textwidth]{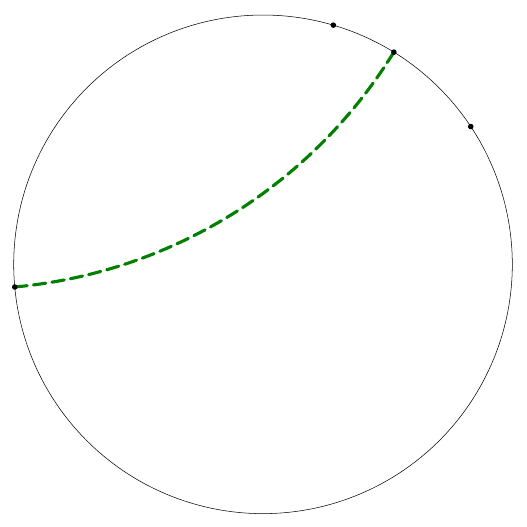}
\label{Fig:NoCuspsLimit}
}

\subfloat[Case: $x$ is a cusp.]{
\labellist
\small\hair 2pt
\pinlabel {$\closure{I(a)}$} [r] at 1 125
\pinlabel {$x$} [br] at 38 215
\pinlabel {$y$} [tr] at 38 35
\pinlabel {$\ell_{j}$} [tl] at 60 129
\pinlabel {$\ell_{j+1}$} [t] at 122 155
\pinlabel {$z_{j}$} [r] at 11 70
\pinlabel {$z_{j+1}$} [l] at 245 141
\endlabellist
\includegraphics[width=0.43\textwidth]{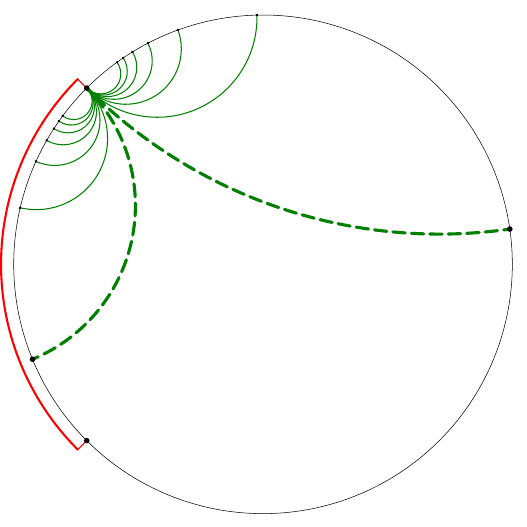}
\label{Fig:NoCuspsxCusp}
}
\caption{The leaf or bi-leaf $\ell$ is drawn with a dashed line.}
\label{Fig:NoCusps}
\end{figure}
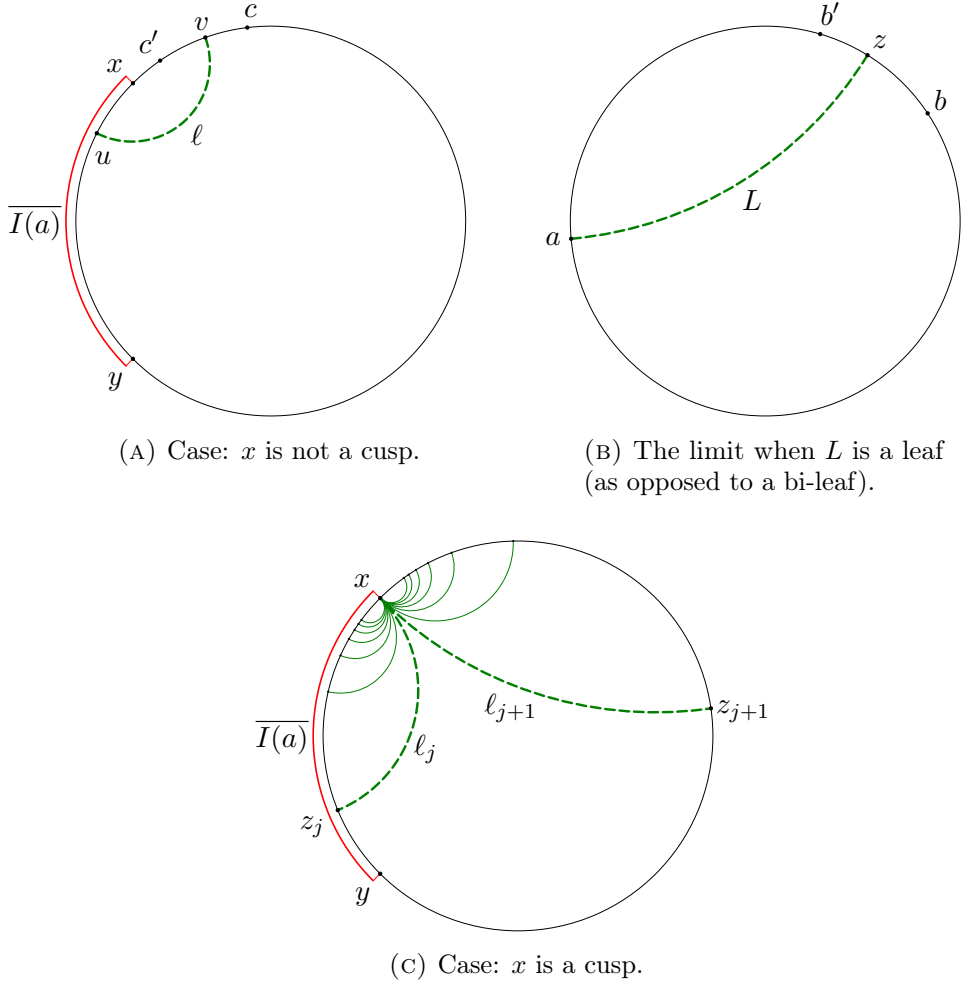

Suppose that one of $x$ and $y$ is \emph{not} a cusp. 
Breaking symmetry, suppose that $x$ is not a cusp.
See \reffig{NoCuspsxNotCusp}.
In this case \refprop{SingleLeaf} gives us a leaf $\ell$ joining a point $u$ in the interior of $\closure{I(a)}$ to a point $v$ in the interior of $\Circle - \closure{I(a)}$.  
There are cusps $c, c'$ in the interior $\Circle - \closure{I(a)}$ which separate $v$ from $\closure{I(a)}$; choose them so $\calO_\calV(c, v, c') = 1$.  
Thus we have $b$ and $b'$, distinct from $a$, so that $c \in I(b)$ and $c' \in I(b')$.  
Passing to a subsequence, the limit $z = \lim_i g_i(v)$ must lie in the (possibly degenerate) interval $[b, b']^\acw$, which does not contain $a$. See \reffig{NoCuspsLimit}.
Passing to a further subsequence, the limit $L = \lim_i g_i(\closure{\ell})$ cannot be a point of $\Circle$. 
By \refprop{LimitOfLeaves}, the limit $L$ is the closure of either a non-cusp leaf or a bi-leaf.
In either case the endpoints of $L$ are equal to $a$ and $z$.  
No cusp can be the endpoint of a non-cusp leaf or bi-leaf, so $a$ is not a cusp.

Suppose that both $x$ and $y$ are cusps.
See \reffig{NoCuspsxCusp}.
In this case we claim there is a bi-leaf $\ell$ with
\begin{itemize}
\item
centre at $x$,
\item
one endpoint in the interior of $\closure{I(a)}$, and 
\item
one endpoint in the interior of $\Circle - \closure{I(a)}$.  
\end{itemize}
Given the claim, we argue as above that $a$ is an endpoint of $\lim_i g_i(\closure{\ell})$, and so cannot be a cusp.  

It remains to find the bi-leaf.  
Let $\{ \ell_j \}_{j \in \ZZ}$ be the set of cusp leaves of $F^\calV$ emanating from $x$.  
Let $z_j$ be the endpoint of $\ell_j$ different from $x$.  
We assume the leaves $\ell_j$ are  indexed so that $\calO_\calV( z_{j}, z_{j+1}, x) = 1$ for all $j$. 
By \reflem{NoDoubleCusps} no $z_j$ is a cusp.  
Thus no $z_j$ is equal to $y$.  
There is thus a unique $j$ so that $\calO_\calV(z_j, y, z_{j+1}) = 1$.  
Thus $\ell_j \cup \ell_{j+1}$ is the desired bi-leaf.
\end{proof}

\begin{lemma}
\label{Lem:FiniteImpliesTwo}
Suppose that the parting set $P= (p_1, \ldots, p_n)$ is finite.
Then $n \leq 2$.
\end{lemma}

\begin{figure}[htbp]
\subfloat[The parting set and complementary intervals.]{
\labellist
\small\hair 2pt
\pinlabel {$p_1$} [b] at 127 9
\pinlabel {$p_2$} [r] at 245 125
\pinlabel {$p_3$} [t] at 127 243
\pinlabel {$p_n$} [l] at 9 125
\pinlabel {$\ell_1$} [b] at 127 49
\pinlabel {$\ell_2$} [r] at 205 125
\pinlabel {$\ell_3$} [t] at 127 203
\pinlabel {$\ell_n$} [l] at 49 125
\pinlabel {$C_1$} [tl] at 213 38
\pinlabel {$C_2$} [bl] at 213 213
\pinlabel {$C_3$} [br] at 38 213
\pinlabel {$C_n$} [tr] at 39 39
\endlabellist
\includegraphics[width=0.43\textwidth]{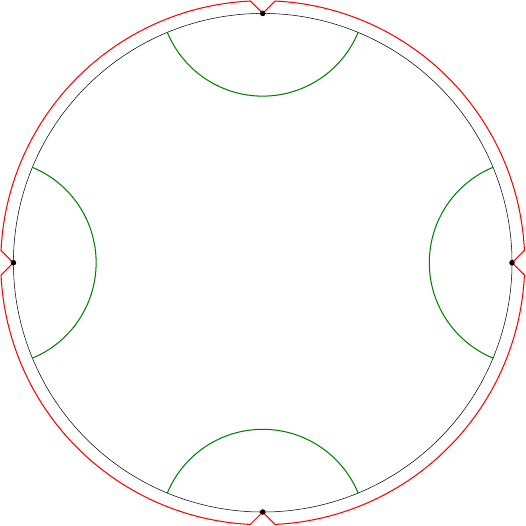}
\label{Fig:FinitePBefore}
}
\qquad
\subfloat[The limit.]{
\labellist
\small\hair 2pt
\pinlabel {$c$} [t] at 140 5
\pinlabel {$a_1$} [tl] at 221 50
\pinlabel {$a_2$} [bl] at 200 218
\pinlabel {$a_3$} [br] at 33 201
\pinlabel {$a_n$} [tr] at 53 34
\endlabellist
\includegraphics[width=0.43\textwidth]{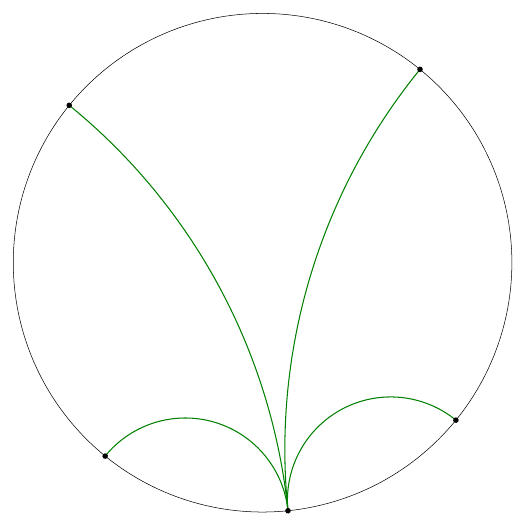}
\label{Fig:FinitePAfter}
}

\caption{Ruling out more than two parting points (here $n=4$).}
\label{Fig:FiniteP}
\end{figure}
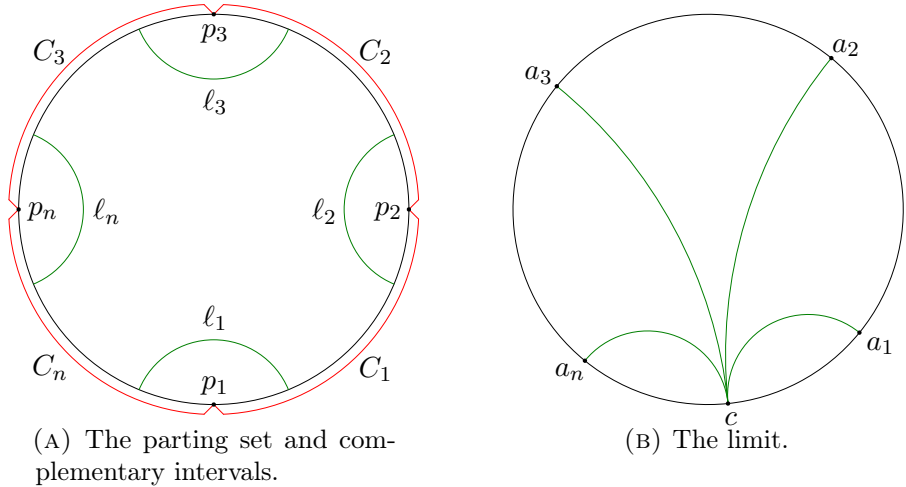

\begin{proof}
For a contradiction, suppose that $n\geq 3$.
We assume that the $p_j$ are arranged in anticlockwise order according to their indices, which we take modulo $n$.
Set $C_j = (p_j, p_{j+1})^\acw$.
By Lemmas \ref{Lem:ClosedSubIntervals} and \ref{Lem:PartingsDecomposeDeltaM}, there are distinct $a_j \in A$ so that $I(a_j) = \closure{C}_j \cap \Delta_\calV$.
No $p_j$ is a cusp by \refcor{NotCuspBefore}.
By \reflem{OrdersAgree}, the $a_j$ are also circularly ordered consistently with their indices.
Fix $j$.
\refprop{SingleLeaf} provides a non-cusp leaf $\ell_j$ separating $p_j$ from $P - \{p_j\}$.
Passing to a subsequence, the Hausdorff limit $L_j = \lim_i g_i(\closure{\ell}_j)$ exists and contains the points $ a_{j-1}$ and $a_j$.
In particular the intersection in $\Circle$ of $L_j$ and $L_{j+1}$ is non-empty for each $j$.  
Consulting the possible limits in \refprop{LimitOfLeaves} we see that this can only happen if each $L_j$ is the closure of a bi-leaf running from $a_{j-1}$ to $a_j$, and all these bi-leaves have centre at a fixed cusp $a$.
By \refcor{OnlyAsymptotics}, all of these bi-leaves lie in the same foliation, say $F^\calV$.

Let $m_j$ be the cusp leaf running between $a_j$ and $a$.
Thus for all $j$, the union $m_{j-1} \cup m_j$ is the interior of $L_j$.  
Re-indexing if necessary, we may assume that $\calO_\calV(a_n, a ,a_1)=1$.
But then the cusp leaves $m_n$ and $m_1$ are separated by $m_2$; 
thus $L_1$ is not the closure of a bi-leaf, a contradiction.
\end{proof}

We next address the case that there are infinitely many $a\in A$ so that $I(a)$ is non-empty.
By \refdef{DecompositionElementsCircle}, this implies that $A$ is the decomposition element associated to some cusp $a_A$. 

\begin{definition}
A convergent sequence $\{p_j\}\to p$ in $\Circle$ is an \emph{anticlockwise converging sequence} if $\calO_\calV(p_j,p_{j+1},p) = 1$ for all $j$.
The limit of an anticlockwise converging sequence in $P$ will be called an \emph{anticlockwise accumulation point of $P$}.
We similarly define \emph{clockwise converging sequences} and \emph{clockwise accumulation points}.
\end{definition}

Here is a lemma which can be proved from the circular order.

\begin{lemma}
\label{Lem:AccumulatingPartings}
The parting set $P$ has at most one accumulation point.
\end{lemma}

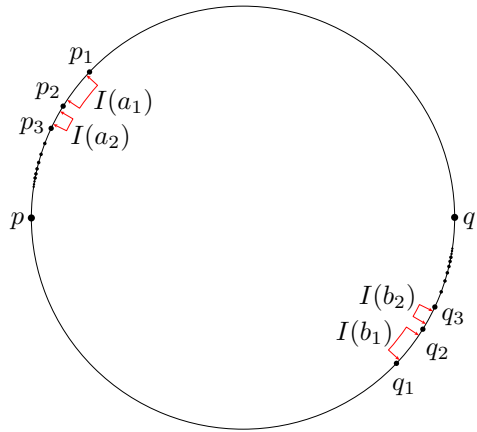
\begin{figure}[htbp]
\labellist
\footnotesize\hair 2pt
\pinlabel $p$ [r] at 1 118
\pinlabel $p_1$ [br] at 40 205
\pinlabel $p_2$ [br] at 21 185
\pinlabel $p_3$ [r] at 13 172
\pinlabel $I(a_1)$ [tl] at 35 195
\pinlabel $I(a_2)$ [tl] at 22 175

\pinlabel $q$ [l] at 243 118
\pinlabel $q_1$ [tl] at 204 31
\pinlabel $q_2$ [tl] at 223 51
\pinlabel $q_3$ [l] at 231 64
\pinlabel $I(b_1)$ [br] at 209 45
\pinlabel $I(b_2)$ [br] at 222 65
\endlabellist
\includegraphics[width=0.45\textwidth]{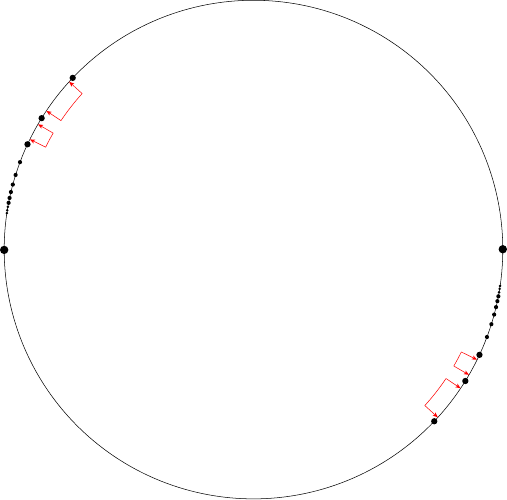}
\caption{Two anticlockwise converging sequences. }
\label{Fig:AntiClockwiseAccumulating}
\end{figure}

\begin{proof}
We first show that there is at most one anticlockwise accumulation point.
Suppose by way of contradiction that there are two anticlockwise converging sequences $p_j\to p$ and $q_j\to q$ in $P$. 
See \reffig{AntiClockwiseAccumulating}. 
Truncating the sequences if necessary we may assume $\calO_\calV(p, q_j, q) = \calO_\calV(q, p_j, p) = 1$ for each $j$. 
For each $j\in\NN$ choose cusps $c_j$ and $d_j$ so that $c_j \in [p_j, p_{j+1}]^\acw$ and $d_j \in [q_j, q_{j+1}]^\acw$.
Suppose that $c_j \in I(a_j)$ and $d_j \in I(b_j)$.
By \reflem{I(a)Component}, we deduce that $I(a_j)\subset [p_j, p_{j+1}]^\acw$ and $I(b_j)\subset [q_j, q_{j+1}]^\acw$.  
Applying \reflem{OrdersAgree}, for any $j$ we have that both $[a_j, b_j]^\acw$ and $[b_j, a_j]^\acw$ have infinite intersection with $A$.  
As in the proof of \reflem{SingletonGoesToCusp}, this can only happen if one of $a_j, b_j$ is the cusp, so it can only happen for one $j$, a contradiction.
The argument that there is at most one clockwise accumulation point is similar.

Suppose now there is a clockwise converging sequence $p_j\to p$ and an anticlockwise converging sequence $q_j\to q$.  
Arguing as in the first paragraph, if $I(a) \subset [q,p]^\acw$, then $a$ is the cusp $a_A$.
Thus $p = q$.
\end{proof}

Thus, when $P$ is infinite, it has a unique accumulation point.
We denote this by $c_P$.

\begin{lemma}
\label{Lem:AccumulationCusp}
Suppose that $P$ is infinite.
Then $c_P$ is a cusp.
\end{lemma}
 
\begin{proof}
For a contradiction, suppose that $c_P$ is not a cusp.
Then \refprop{SingleLeaf} provides non-cusp leaves $\ell_k$ determining a nested neighbourhood basis of $c_P$.  
Applying \reflem{Perturb}, we may assume that for all $k$, the ideal points $u_k$ and $v_k$ of $\ell_k$ are disjoint from $P$.
We orient $\ell_k$ away from $u_k$, arranging matters so that $\calO_\calV(v_k, c_P, u_k) = 1$.  
Thus $c_P$ is to the left of $\ell_k$.  

Now suppose that $c_P$ has both a clockwise and an anticlockwise converging sequence.
Passing to a subsequence, we may assume that for $k \neq k'$ the four points $u_k, v_k, u_{k'}, v_{k'}$ lie in distinct components of $\Circle - P$. 
We deduce that $u_k \to c_P$ and $v_k \to c_P$ are clockwise and anticlockwise converging sequences, respectively. 

Let $L_k$ be the Hausdorff limit of $g_i(\closure{\ell}_k)$.
Since $u_k$ and $v_k$ lie in distinct components of $\Circle - P$, we deduce that $L_k$ meets at least two points of $A$.
Thus, by \refprop{LimitOfLeaves} it is the closure of a bi-leaf or a non-cusp leaf.
However, $A$ is an infinite decomposition element and hence contains exactly one cusp, $a_A$, and a collection of thorns.
Non-cusp leaves do not have these as ideal points, thus $L_k$ is the closure of a bi-leaf.

Let $a_k$ and $b_k$ be the thorns of $L_k$, named so that $\calO_\calV(b_k, a_A, a_k) = 1$.
Since $u_k \to c_P$ is a clockwise converging sequence, by \reflem{OrdersAgree} we have that $a_k \to a_A$ is also a clockwise converging sequence.  
Similarly, $b_k \to a_A$ is an anticlockwise converging sequence.  
Thus the endpoints of $L_k$ are separated from $a_A$ by the endpoints of $L_{k+1}$, contradicting the fact that $L_{k+1}$ is the closure of a bi-leaf with centre $a_A$. 
Thus we have proved that $c_P$ is a cusp. 

Now, breaking symmetry, suppose that $c_P$ has an anticlockwise converging sequence in $P$, but no clockwise converging sequence.
We run the above argument again, except that in this case we may assume that $u_k$ and $u_{k'}$ lie in the same component of $\Circle - P$.
Again we have that $u_k \to c_P$ and $v_k \to c_P$ are clockwise and anticlockwise converging sequences, respectively. 
The rest of the argument is the same except that the $b_k$'s all agree, and we reach the same contradiction.
\end{proof}

\begin{lemma}
\label{Lem:InfiniteCase}
Suppose that $P$ is infinite.
Suppose that $c_P$ is the cusp given by \reflem{AccumulationCusp}.
Then $P$ is contained in the decomposition element~$[c_P]_{S^1}$.
\end{lemma}

\begin{proof}
The point $c_P$ is either an anticlockwise or a clockwise accumulation point (or possibly both).
Breaking symmetry, we suppose that $c_P\in P$ is an anticlockwise accumulation point.  
Index the elements of $P - \{c_P\}$ either by the integers or by the natural numbers so that $\calO_\calV(p_j, p_{j+1}, c_P) = 1$ for all $j$.  
If $c_P$ is an accumulation point from both sides, we use all the integers; 
otherwise we use only the non-negative integers.

By Lemmas \ref{Lem:ClosedSubIntervals} and \ref{Lem:PartingsDecomposeDeltaM}, for each $p_j\in P - \{c_P\}$, there is some $a_j\in A$ so that $I(a_j) = \Delta_\calV\cap [p_j, p_{j+1}]^\acw$.  
By \reflem{NoCuspsAfter}, no $a_j$ is a cusp.  
(Note that not every element of $A$ is necessarily given an index here.)
By \refcor{NotCuspBefore}, no $p_j$ is a cusp.
By \refprop{SingleLeaf}, we can find disjoint non-cusp leaves $\ell_j$ so that $\ell_j$ separates $p_j$ from $P - \{p_j\}$.  
As in the proof of \reflem{FiniteImpliesTwo}, the limit $L_j = \lim_i g_i(\closure{\ell}_j)$ is the closure of a bi-leaf joining $a_{j-1}$ to $a_j$, and passing through the cusp $a_A\in A$.  
Applying \refprop{LimitOfLeaves}, it follows that the $\ell_j$ were all in the same foliation.
Breaking symmetry, we suppose that the $\ell_j$ are in the foliation~$F^\calV$.

\begin{claim}
\label{Clm:GreenThorns}
For each $j$, the point $a_j$ is a non-cusp ideal point of a bi-leaf of $F^\calV$.
\end{claim}

\begin{proof}
Since the $\ell_j$ lie in $F^\calV$, by \refprop{LimitOfLeaves},  the interior of each $L_j$ is a bi-leaf of $F^\calV$.
\end{proof}

\begin{claim}
\label{Clm:AccumulatedBothSides}
The cusp $c_P$ is both an anticlockwise and a clockwise accumulation point of $P$.
\end{claim}

\begin{proof}
Suppose that the elements of $P - \{c_P\}$ are indexed by the natural numbers, starting with $p_1$.
Then $(c_P, p_1)^\acw$ is a component of $\Circle - P$.
By Lemmas \ref{Lem:ClosedSubIntervals} and \ref{Lem:PartingsDecomposeDeltaM} there is some $a_0 \in A$ so that $I(a_0) = \Delta_\calV\cap [c_P, p_1]^\acw$.  
Thus the cusp $c_P$ lies in $I(a_0)$.
Note also that, by \reflem{NoCuspsAfter}, the point $a_0$ is not a cusp, so it is not $a_A$.

Fix a bi-leaf $\ell$ of $F^\calV$ with centre $c_P$ whose other ideal points $p$ and $q$ do not lie in $(c_P, p_1)^\acw$.
Note that $p$ and $q$ do not lie in $P$ because $\ell$ cannot cross the leaves $\ell_j$.
Thus $p$ and $q$ lie in an $a_j$--interval and an $a_k$--interval, respectively, for some $j, k > 0$.
Let $L = \lim_i g_i(\closure{\ell})$.
Since $L$ contains $a_j$, $a_k$, and $a_0$, it is neither a point, nor the closure of a non-cusp leaf or bi-leaf, contradicting \refprop{LimitOfLeaves}.
\end{proof}

Thus the elements of $P - \{c_P\}$ are indexed by the integers.

\begin{claim}
\label{Clm:ForeverAlone}
$\lim_i g_i(c_P) = a_A$.
\end{claim}

\begin{proof}
Since $c_P$ is a cusp, it belongs to $I(a)$ for some $a \in A$ by \refrem{AttractorPartition}.
By \refclm{AccumulatedBothSides} and Lemmas \ref{Lem:ClosedSubIntervals} and \ref{Lem:PartingsDecomposeDeltaM}, the set $I(a)$ is a singleton.
By \reflem{SingletonGoesToCusp}, the limit $a$ is the cusp $a_A$.
\end{proof}

\begin{claim}
\label{Clm:PurpleThorns}
Suppose that $m$ is a bi-leaf of $F_\calV$ with centre $c_P$.
Then the ideal points of $m$ lie in $P$.
\end{claim}

\begin{proof}
Suppose that some ideal point of $m$, say $q$, does not lie in $P$. 
Then $q$ lies in some component of $\Circle - P$.
Thus by \reflem{ClosedSubIntervals}, the point $q$ lies in an $a_j$--interval for some $j \in \ZZ$.
Thus $M = \lim_i g_i(\closure{m})$ contains $a_A$ and $a_j$.
By \refprop{LimitOfLeaves} and \refrem{Cover}, the interior of $M$ is a bi-leaf of $F_\calV$.
However, this with \refclm{GreenThorns} and \refcor{OnlyAsymptotics} gives a contradiction.
\end{proof}

\begin{claim}
\label{Clm:CuspLeaves}
For each $j \in \ZZ$ there is exactly one bi-leaf of $F^\calV$ with centre $c_P$ and other ideal points in an $a_j$--interval and an $a_{j+1}$--interval.
\end{claim}

\begin{proof}
The proof is similar to that of \refclm{AccumulatedBothSides} but with some key differences.
Suppose that $\ell$ is a bi-leaf of $F^\calV$ with centre $c_P$.
Let $m$ be either of the two bi-leaves of $F_\calV$ that contain the divider (see \refdef{BiLeaf}) of $\ell$.
Let $q$ and $q'$ be the other ideal points of $\ell$.
Let $q''$ be the ideal point of $m$ which is also the ideal point of the divider of $\ell$.
As in the proof of \refclm{AccumulatedBothSides}, the points $q$ and $q'$ do not lie in $P$.
Suppose that $q$ and $q'$ lie in an $a_j$--interval and an $a_{j'}$--interval, respectively. 
Let $L = \lim_i g_i(\closure{\ell})$.
Note that $L$ contains $a_j$, $a_{j'}$, and by \refclm{ForeverAlone}, the cusp $a_A$.
By \reflem{NoCuspsAfter}, neither $a_j$ nor $a_{j'}$ equals $a_A$.

We deduce from \refprop{LimitOfLeaves} and \refrem{Cover} that $L$ is the closure of a bi-leaf of $F^\calV$ with centre $a_A$.
Thus $a_j$ and $a_{j'}$ are either identical or consecutive elements of $A - \{a_A\}$.
Thus either $j = j'$ or, by \reflem{OrdersAgree}, we have that $j$ and $j'$ are consecutive integers.
Suppose, for a contradiction, that $j=j'$.
Then $q''$ lies in an $a_j$--interval, contradicting \refclm{PurpleThorns}.

Induction (in both directions) now proves the claim.
\end{proof}

\begin{claim}
For each $p_j\in P - \{c_P\}$ there is a cusp leaf of $F_\calV$ joining $c_P$ to $p_j$.
\end{claim}

\begin{proof}
By \refclm{CuspLeaves}, there is a bi-leaf $\ell$ of $F^\calV$ with centre $c_P$ and other ideal points ending in an $a_{j-1}$--interval and an $a_j$--interval.
The divider of $\ell$ has the desired property by \refclm{PurpleThorns}.
\end{proof}

Thus $P$ lies in the decomposition element $[c_P]_{S^1}$.
\end{proof}

Putting together Lemmas \ref{Lem:FiniteImpliesTwo}, \ref{Lem:Case2}, and \ref{Lem:InfiniteCase}, 
we have the following.

\begin{proposition}
\label{Prop:FindingB}
There is a unique decomposition element $B \subset \Circle$ containing all the partings of $\{g_n\}$.
Furthermore, the sequence of maps $g_n|(\Sphere -\{B\})$ converges uniformly to the constant map $(\Sphere -\{B\})\to \{A\}$. \qed
\end{proposition}

\begin{corollary}
\label{Cor:Collapsing}
The sequence $(g_i)$ is a collapsing sequence. \qed
\end{corollary}

\noindent
This completes the proof of \refprop{Convergence}.

\section{Cusps are parabolic}
\label{Sec:Parabolic}

We prove the following:

\begin{proposition}
\label{Prop:Parabolic}
Suppose that $c \in \Delta_\calV$ is a cusp. 
Then $[c]_{S^1} \in \Sphere$ is a bounded parabolic point for the action of $\pi_1(M)$.
\end{proposition}

Fix $c \in \Delta_\calV$ and take $C = [c]_{S^1}$ to be the associated clamshell boundary.
Applying \refrem{Cover} if needed we assume that cusps are tori, not Klein bottles.

\begin{proposition}
\label{Prop:ParabolicFacts}
\leavevmode
\begin{enumerate}
\item
\label{Itm:ZZ^2}
$\Stab(c) \isom \ZZ^2$.
\item
\label{Itm:beta} 
There is a primitive element $\beta \in \Stab(c)$ whose fixed-point set in $\Circle$ is $C$. 
\item
\label{Itm:gamma}
If $\gamma\in \Stab(c) - \subgroup{\beta}$, then $\gamma$ preserves $C$ setwise.  
Also 
\[
\lim_{n \to \infty}\gamma^n(x) = c
\]
for any $x \in \Circle$. 
\end{enumerate}
\end{proposition}

\begin{proof}
The first item follows from the fact that $c$ is a cusp of $\cover{M}$ and that the interior of $M$ admits a finite-volume hyperbolic metric. 
See~\cite[Section~7.13]{FSS22}.
The second item is a restatement of~\cite[Lemma~A.8]{FSS22}.
The third item follows from~\cite[Lemma~7.15]{FSS22}.
\end{proof}

\begin{lemma}
\label{Lem:StabC}
The stabilisers $\Stab(c)$ and $\Stab(C)$ are equal. 
\end{lemma}

\begin{proof}
Note that $\pi_1(M)$ preserves $\Delta_\calV$.
By \refdef{DecompositionElementsCircle}, the clamshell boundary $C$ consists of $c$, together with its thorns (none of which are cusps).
By \reflem{Laminations}\refitm{CrownZZ}, the cusp $c$ is the unique accumulation point of $C$.
From this we deduce that $\Stab(C) < \Stab(c)$.
\refprop{ParabolicFacts} gives the reverse inclusion.
\end{proof}

\begin{proof}[Proof of \refprop{Parabolic}]
Suppose that $C = [c]_{S^1}$ is the decomposition element containing the cusp $c \in \Delta_\calV$. 
We first prove that $C$ is a parabolic point, as in \refdef{Parabolic}.

By \refprop{ParabolicFacts}\refitm{ZZ^2}, the stabiliser of $c$ is isomorphic to $\ZZ^2$.
Applying \reflem{StabC}, we have that $\Stab(C)$ is infinite.

We now show that no element of $\Stab(C)$ is loxodromic. 
By \refprop{ParabolicFacts}, the stabiliser $\Stab(C)$ contains an element $\beta$ whose fixed set in $\Circle$ is equal to $C$.  
Fix $\delta \in \Stab(C) - \{1\}$.  
If $\delta$ is a power of $\beta$ then its fixed-point set is equal to $C$ in $\Circle$.  
Thus, it acts as a translation on every component of $\Circle - C$.  
If $\delta$ lies in $\Stab(C) - \subgroup{\beta}$ then its powers send every point of $\Circle - c$ eventually to $c$.  
In either case, $\delta$ has no finite orbit in $\Circle - C$.
Thus $\delta$ does not stabilise a singleton or the endpoints of a leaf.

By \refcor{CircleDecomposition} the last case we need to consider is that $\delta$ fixes, setwise, some infinite decomposition element $D \neq C$.  
We deduce that $D$ contains a cusp $d$.
By \reflem{Laminations}\refitm{CrownZZ}, the element $\delta$ fixes $d$.  
But then $\delta$ has an orbit of size one, namely $d$ itself.  
This is the desired contradiction. 
We deduce that $C$ is a parabolic point.

We now prove that $C$ is a bounded parabolic point.
Fix some $\gamma \in \Stab(C) - \subgroup{\beta}$.  
By \refprop{ParabolicFacts}\refitm{gamma}, the element $\gamma$ acts as a translation on $\Circle - \{c\}$.
Let $I \subset \Circle - \{c\}$ be a fundamental domain for the action, with endpoints in $C$.  
Note that $C \cap I$ is finite.  
We take the closures of the components of $I - C$ to obtain closed intervals $I_1, I_2, \ldots, I_k$.

For each $j$ we choose a closed interval $J_j \subset I_j$ which is a fundamental domain for the action of $\subgroup{\beta}$ on the interior of $I_j$. 
The union $K = \medcup_j J_j$ is compact.  
Also  its $\Stab(C)$--translates cover all of $\Circle - C$.  
It follows that the image of $K$ in $\Sphere$ under the quotient map $\Psi_\calV$ (\refdef{VeeringSphere}) is again compact.
Additionally, the $\Stab(C)$--translates of $\Psi_\calV(K)$ cover all of $\Sphere - \{C\}$.  
Thus $\Stab(C)$ acts cocompactly on $\Sphere - \{C\}$.
This proves that $C$ is a bounded parabolic point.
\end{proof}

\section{Non-cusp leaves are conical}
\label{Sec:LeafConical}

We prove the following: 

\begin{proposition}
\label{Prop:LeafConical}
Suppose that $\ell$ is a non-cusp leaf of either foliation.
Suppose that $x \in \bdyi \ell$.
Then the decomposition element $[x]_{S^1} \in \Sphere$ is a conical limit point for the action of $\pi_1(M)$.
\end{proposition}

Suppose that $x$ is an ideal point of a non-cusp leaf $\lambda$ of $\Lambda^\calV$.
(The case of a non-cusp leaf of $\Lambda_\calV$ is similar.)
The equivalence class $[x]_{S^1} = \bdyi \lambda$ is an unordered pair of points of $\Circle$.  
Let $\Mobius$ be the space of unordered pairs of points of $\Circle$.  
This is an open M\"obius band.
In~\cite[Definition~A.13]{FSS22} the authors associate to each edge $e$ of a veering triangulation the \emph{upper and lower product neighbourhoods} $P^\calV(e)$ and $P_\calV(e)$.  
These are the images of products of intervals in $\Circle \times \Circle$ minus the diagonal.
We reproduce this definition here.

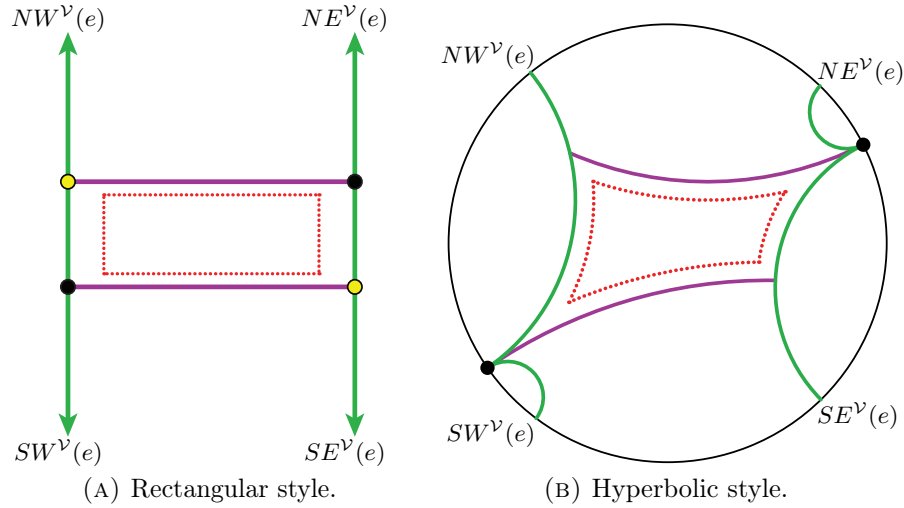
\begin{figure}[htbp]
\subfloat[Rectangular style.]{
\labellist
\scriptsize\hair 2pt
\pinlabel {$SE^\calV(e)$} [t] at 170   18 
\pinlabel {$NE^\calV(e)$} [b] at 170  253
\pinlabel {$NW^\calV(e)$} [b] at   0  253
\pinlabel {$SW^\calV(e)$} [t] at   0   18
\endlabellist
\includegraphics[width = 0.32 \textwidth]{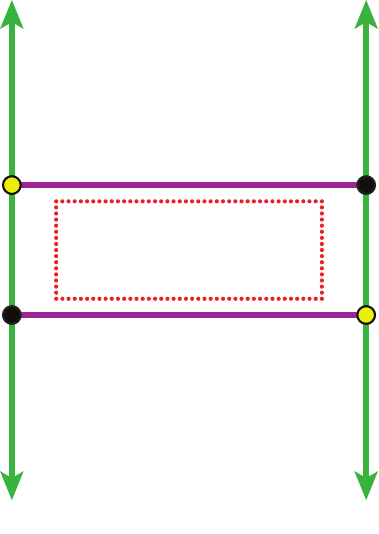}
}
\qquad
\subfloat[Hyperbolic style.]{
\labellist
\scriptsize\hair 2pt
\pinlabel {$SE^\calV(e)$} [tl] at 200   36 
\pinlabel {$NE^\calV(e)$} [bl] at 200  205
\pinlabel {$NW^\calV(e)$} [br] at   50  214
\pinlabel {$SW^\calV(e)$} [tr] at   50   26
\endlabellist
\includegraphics[width = 0.46 \textwidth]{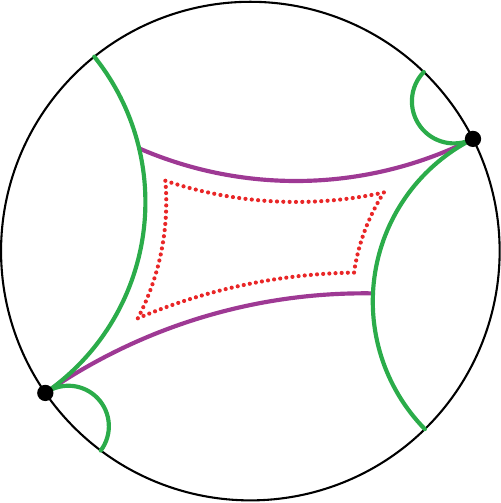}
}
\caption{The eastern and western bi-leaves of the edge rectangle $R(e)$ determine the upper product neighbourhood, $P^\calV(e)$.}
\label{Fig:UpperProductNeighbourhood}
\end{figure}

\begin{definition}
\label{Def:BoxAnatomy}
Let $e$ be an edge of $\calT$.
Let $R(e)$ be the associated edge rectangle.
Let $k^W(e)$ be the unique bi-leaf of $F^\calV$ which 
\begin{itemize}
\item
contains the west side of $R(e)$ and
\item
is a Hausdorff limit of non-cusp leaves of $F^\calV$ meeting $R(e)$. 
\end{itemize}
We define $k^E(e)$ similarly with respect to the east side of $R(e)$. 
The centres of $k^W(e)$ and $k^E(e)$ are the endpoints of the edge $e$.  
The bi-leaf $k^W(e)$ has a northern endpoint $NW^\calV(e)$ and a southern endpoint $SW^\calV(e)$.  
Similarly, $k^E(e)$ has a northern endpoint $NE^\calV(e)$ and a southern one $SE^\calV(e)$.  
The \emph{upper product neighbourhood} $P^\calV(e)$ is the image of the product 
\[
[SW^\calV(e),SE^\calV(e)]^\acw\times[NE^\calV(e),NW^\calV(e)]^\acw
\]
in $\Mobius$. 

We define the bi-leaves $k_S(e)$ and $k_N(e)$, in $F_\calV$, with respect to the south and north sides. 
The endpoints $SW_\calV(e)$ and $SE_\calV(e)$ of $k_S(e)$ and $NE_\calV(e)$ and $NW_\calV(e)$ of $k_N(e)$ determine the \emph{lower product neighbourhood} $P_\calV(e)$ which is the image of the product 
\[
[NW_\calV(e),SW_\calV(e)]^\acw \times [SE_\calV(e),NE_\calV(e)]^\acw 
\]
in $\Mobius$.
\end{definition}

See \reffig{UpperProductNeighbourhood} for depictions of an upper product neighbourhood in the rectangular and hyperbolic styles.

Recall that for each non-cusp leaf $\lambda$ of $\Lambda^\calV$, there is a non-cusp leaf of $F^\calV$ with endpoints $\bdyi \lambda$.
We observe the following.

\begin{lemma}
\label{Lem:RectangleBox}
Suppose that $\ell$ is a non-cusp leaf of $F^\calV$.
The set of endpoints $\bdyi \ell$ lies in $P^\calV(e)$ if and only if $\ell$ meets $R(e)$.  \qed
\end{lemma}


\begin{definition}
Suppose that $U$ and $V$ are subsets of a topological space. 
Then we say that $U$ is \emph{strictly nested} in $V$ if $\closure{U} \subset \operatorname{Int}(V)$.
\end{definition}

The following statement is a consequence of~\cite[Lemma~A.14]{FSS22}.
See \reffig{BoxAtInfinity} for depictions of the conclusions of this lemma.

\begin{figure}[htbp]
\subfloat[Rectangular style.]{
\includegraphics[width = 0.35\textwidth]{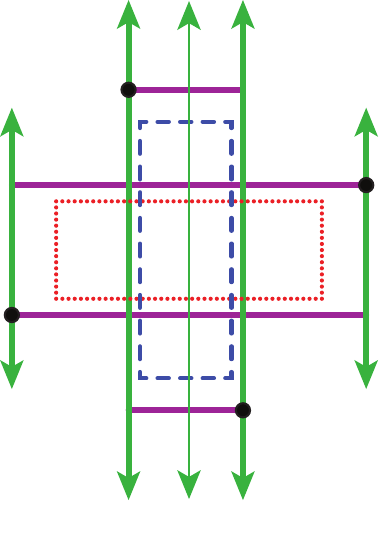}
}
\qquad
\subfloat[Hyperbolic style.]{
\labellist
\scriptsize\hair 2pt
\pinlabel {$p$} [b] at 125 242
\pinlabel {$q$} [t] at 125 0
\endlabellist 
\raisebox{10pt}{\includegraphics[width = 0.45 \textwidth]{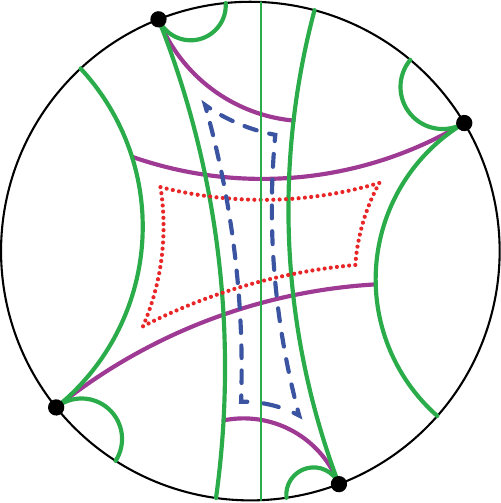}}
}
\caption{The nesting of intervals in $\Circle$ induced by the products $P^\calV(e)$ and $P^\calV(e')$.}
\label{Fig:BoxAtInfinity}
\end{figure}

\begin{lemma}
\label{Lem:BoxPairsForLeaf}
Suppose that $\ell$ is a non-cusp leaf of $F^\calV$.
There is a pair of edges $e$ and $e'$ and a sequence of group elements $(g_i)$ in $\pi_1(M)$ with the following properties.
\begin{enumerate}
\item
\label{Itm:NoCommonCusp} 
The edges $e'$ and $e$ do not share a cusp.
\item
\label{Itm:SNSpan} 
The edge rectangle $R(e')$ south-north spans the edge rectangle $R(e)$.
\item
\label{Itm:StrictlyNested} 
$P^\calV(e')$ is strictly nested in $P^\calV(e)$.
\item 
\label{Itm:NeighbourhoodBasis}
The products $g_i(P^\calV(e'))$ form a nested neighbourhood basis for 
$\bdyi \ell \in \Mobius$.
So do the products $g_i(P^\calV(e))$. \qed
\end{enumerate}
\end{lemma}

\begin{proof}[Proof of \refprop{LeafConical}]
Suppose that $\ell$ is a leaf of $F^\calV$; the other case is similar. 
Let $e$, $e'$, and $(g_i)$ be as in the conclusion to \reflem{BoxPairsForLeaf}.
Recall that $\ell$ crosses $R(e)$ and $R(e')$.
  
Using the notation of \refdef{BoxAnatomy}, set $I(e) = [NE^\calV(e), NW^\calV(e)]^\acw$ and $J(e) = [SW^\calV(e),SE^\calV(e)]^\acw$.  
We make the similar definitions for $I(e')$ and $J(e')$.
Let $p$ and $q$ be the points of $\bdyi\ell$.
Breaking symmetry, suppose that $p\in I(e')\subset I(e)$ and $q\in J(e')\subset J(e)$.  
By \reflem{BoxPairsForLeaf}\refitm{StrictlyNested}, the products $P^\calV(e')\subset P^\calV(e)$ are strictly nested.
Therefore, the intervals $I(e')\subset I(e)$ and $J(e') \subset J(e)$ are also strictly nested.  
By \reflem{BoxPairsForLeaf}\refitm{NeighbourhoodBasis}, the products $g_i(P^\calV(e))$ and $g_i(P^\calV(e'))$ form neighbourhood bases for $\bdyi\ell$.

Either infinitely many of the $g_i$ send $I(e)$ into itself, or all but finitely many send $I(e)$ into $J(e)$.
In the first case there is a subsequence so that
\begin{itemize}
\item the intervals $g_i(I(e))$ and $g_i(I(e'))$ form neighbourhood bases for $p$ and 
\item the intervals $g_i(J(e))$ and $g_i(J(e'))$ form neighbourhood bases for $q$.
\end{itemize}
The second case swaps $p$ with $q$ in the above and is handled similarly.
Applying \refprop{Convergence}, we pass to a further subsequence and reindex so that $(g_i^{-1})$ is a collapsing sequence for the action of $\pi_1(M)$ on $\Sphere$.
 
Let $m$ be a non-cusp leaf of $F_\calV$ which crosses $R(e)$ and thus crosses $\ell$.
Let $r = \ell \cap m$. 

\begin{claim*}
The point $r$ lies in $R(g_i(e')) \cap R(g_i(e))$.
\end{claim*}

\begin{proof}
By \reflem{BoxPairsForLeaf}\refitm{NeighbourhoodBasis}, for every $i$, the leaf $\ell$ crosses $R(g_i(e'))$ and $R(g_i(e))$.
We prove the corresponding statement for $m$ by induction.
Since $m$ crosses $R(e)$, it also crosses $R(e')$.
Also $g_0$ is the identity.
This gives the base case.

Suppose that $m$ crosses $R(g_i(e'))$ and $R(g_i(e))$.
Since the neighbourhood bases are nested (\reflem{BoxPairsForLeaf}), we find that $R(g_i(e'))$ and $R(g_i(e))$ west-east span $R(g_{i+1}(e))$ and thus west-east span $R(g_{i+1}(e'))$.
Thus $m$ crosses $R(g_{i+1}(e'))$ and $R(g_{i+1}(e))$.
\end{proof}

Thus the point $(g_i^{-1}(r))$ lies in $R(e') \cap R(e)$.
Note that the conclusions~\refitm{NoCommonCusp} and~\refitm{SNSpan} of \reflem{BoxPairsForLeaf} imply that the intersection $R(e') \cap R(e)$ is precompact in $\Link$.  
Passing to a subsequence, we have that $(g_i^{-1}(r))$ converges to some point $r_\infty$ in the closure of $R(e') \cap R(e)$.
In particular, $r_\infty$ is in $\Link$.

Passing to further subsequences, the following Hausdorff limits exist.
\[
L = \lim_{i\to\infty} g_i^{-1}\left(\closure{\ell}\right) \qquad M = \lim_{i\to\infty} g_i^{-1}\left(\closure{m}\right)
\]
By \refprop{LimitOfLeaves}, the limits $L$ and $M$ are the closures of non-cusp or bi-leaves of $F^\calV$ and $F_\calV$, respectively.
Since $L$ and $M$ cross at $r_\infty$, they are not bi-leaves for a common cusp.
From Definitions~\ref{Def:DecompositionElementsCircle} and~\ref{Def:DecompositionElementsSphere} we deduce that $\bdyi L$ and $\bdyi M$ lie in distinct decomposition elements of $\Sphere$. 
So define $A = [\bdyi M]_{S^1}$ and $B = [\bdyi L]_{S^1}$ and note that $A \neq B$.

Suppose that $m' \neq m$ is another non-cusp leaf of $F_\calV$ crossing $R(e)$. 
As above, we obtain $A' = [\bdyi M']_{S^1}$ with $A' \neq B$.
Since $(g^{-1}_i)$ is a collapsing sequence, we deduce that $A = A'$.
It follows that $\bdyi \ell$ is the repelling point.
From this we deduce that $A$ is the attracting point.
Since $A \neq B$ we have verified the properties of \refdef{CLP}.
\end{proof}

\section{Singletons are conical}
\label{Sec:SingletonConical}

We prove the following: 

\begin{proposition}
\label{Prop:SingletonConical}
Suppose that $x \in \Circle$ is a singleton.
Then the decomposition element $[x]_{S^1} \in \Sphere$ is a conical limit point for the action of $\pi_1(M)$.
\end{proposition}

\subsection{Hulls}

Our proof of \refprop{SingletonConical} uses \emph{hulls} in $\Disc$.
We follow~\cite[Section~6]{SchleimerSegerman24}.
Recall the notion of a sector (\refdef{Sector}) of a point of $\Link \cup \Delta_\calV$. 

\begin{definition}
Suppose that $u$, $v$, and $w$ are points of $\Disc$, 
with $v$ in $\Link \cup \Delta_\calV$.
Suppose that there are non-adjacent sectors $S$ and $T$ at $v$ so that $u \in S$ and $w \in T$.
Then we say that $v$ is \emph{sector-between} $u$ and $w$.  
\end{definition}

The next two results follow directly from the definition.

\begin{lemma}
\label{Lem:BetweenTransitive}
Suppose that $u$, $v$, $w$ and $x$ are points of $\Disc$, 
with $v$ and $w$ in $\Link \cup \Delta_\calV$.
Suppose that $v$ is sector-between $u$ and $x$. 
Suppose that $w$ is sector-between $v$ and $x$.
Then $w$ is sector-between $u$ and $x$. \qed
\end{lemma}

\begin{lemma}
\label{Lem:CuspBetween}
Suppose that $u$, $v_i$, and $w$ are points of $\Disc$, with $v_i \in \Link$.
Suppose that $c$ is a cusp.
Suppose that $u$ and $w$ do not lie on cusp leaves of $c$.
Suppose that the $v_i$ tend to $c$ and that the $v_i$ are sector-between $u$ and $w$.
Then $c$ is sector-between $u$ and $w$. \qed
\end{lemma}

\begin{notation}
Suppose that $u$ is a point of $\Link \cup \Delta_\calV$ and $x$ is a singleton.
We denote the sector at $u$ containing $x$ by $S(u, x)$.
\end{notation}

\begin{lemma}
\label{Lem:NestedSectors}
Suppose that $c$ and $d$ are distinct cusps and $x$ is a singleton.
Suppose that $d$ is sector-between $c$ and $x$.
Then $S(d, x)$ is a proper subset of $S(c, x)$.
\end{lemma}

\begin{proof}
We first prove that $\bdy S(c, x)$ is disjoint from $\bdy S(d, x)$. 
Let $\ell_c$, $m_c$, $\ell_d$, and $m_d$ be the cusp leaves on the boundaries of $S(c, x)$ and $S(d, x)$, respectively.
Leaves from the same foliation do not intersect. 
So, for a contradiction, suppose that $p = \ell_d \cap m_c$.
Consider the initial arc of $m_c$ ending at $p$.
This is entirely contained in some sector $T$ based at $d$, with $\ell_d$ in its boundary.
Thus $T$ contains $c$, and $T$ is adjacent to (or equal to) $S(d, x)$.
Thus $d$ is not sector-between $c$ and $x$, a contradiction.

Since $d$ is sector-between $c$ and $x$, we deduce that $c$ does not lie in $S(d, x)$.
Since the cusp leaves of $S(c, x)$ and $S(d, x)$ are disjoint, we deduce that the sectors either are disjoint or nest as desired.
However, the sectors are not disjoint because both contain $x$.
\end{proof}

\begin{definition}
\label{Def:Hull}
Suppose that $u, v \in \Disc$.
The \emph{hull} of $u$ and $v$, written $\Hull(u, v)$, is the closure of the set of points of $\Disc$ that are sector-between $u$ and $v$.
\end{definition}

From this and \reflem{CuspBetween} we have the following.

\begin{corollary}
\label{Cor:CuspBetween}
Suppose that $c$ is a cusp in $\Hull(u, v)$.
Then $c$ is sector-between $u$ and $v$. \qed
\end{corollary}

\subsection{Spanning rectangles}

\begin{definition}
\label{Def:RectangleTowards}
Suppose that $x$ is a singleton and $c$ is a cusp.
We define the \emph{rectangle at $c$ towards $x$}, denoted $Q(c, x)$, to be the maximal rectangle so that
\begin{itemize}
\item $Q(c, x) \subset S(c, x)$,
\item $c$ is an ideal corner of $Q(c, x)$, and 
\item each point of $Q(c, x)$ is sector-between $c$ and $x$. \qedhere
\end{itemize}
\end{definition}

\begin{lemma}
\label{Lem:RectangleTowards}
$Q(c, x)$ exists and is unique.
\end{lemma}

\begin{proof}
Let $\ell_c \in F^\calV$ and $m_c \in F_\calV$ be the cusp leaves at $c$ bounding $S(c, x)$.
We extend their orientations (namely, away from $c$) to obtain orientations of $F^\calV$ and $F_\calV$.
With these orientations, $c$ is the south-west corner of $S(c, x)$.

By \reflem{IVT} there are cusps $a$ and $b$ with cusp leaves $\ell$, $\ell'$, $\ell''$, $m$, $m'$, and $m''$ so that
\begin{itemize}
\item $\ell$ crosses $m_c$, 
\item $m$ crosses $\ell_c$, 
\item $\ell'$ and $\ell''$ form bi-leaves with $\ell$, and 
\item $m'$ and $m''$ form bi-leaves with $m$.
\end{itemize}
Let $\hat{\ell}$ be the bi-leaf (formed by $\ell$ and either $\ell'$ or $\ell''$) so that $c$ is contained in the sector with boundary consisting of $\ell$ and the divider of $\hat{\ell}$.
We define $\hat{m}$ similarly.

Then $Q(c, x)$ is the rectangle with boundary contained in the union of $\ell_c$, $m_c$, $\hat{\ell}$, and $\hat{m}$.
\end{proof}

The bi-leaves $\hat{\ell}$ and $\hat{m}$ separate $c$ from $x$ and are not nested, so there are four cases, as shown in \reffig{FindSpanningEdgeRectangle}.

\begin{figure}[htbp] 
\subfloat[]{
\labellist
\small\hair 2pt
\pinlabel {$c$} [tr] at 2 2
\pinlabel {$\ell_c$} [r] at 2 60
\pinlabel {$m_c$} [t] at 60 2
\pinlabel {$\ell$} [l] at 168 60
\pinlabel {$m$} [b] at 60 230
\pinlabel {$a=b$} [b] at 166 230
\endlabellist
\includegraphics[width = 0.35 \textwidth]{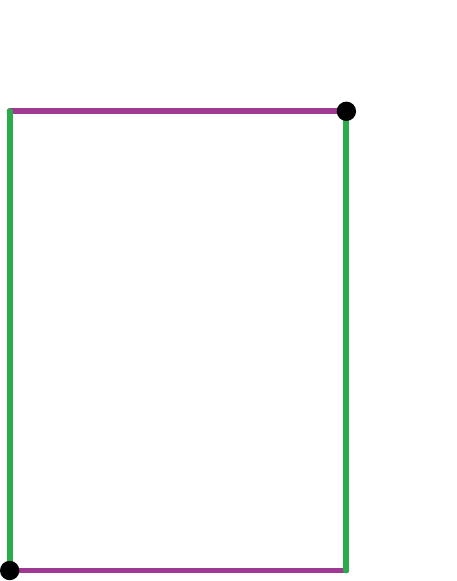}
\label{Fig:Case0}
}
\subfloat[]{
\labellist
\small\hair 2pt
\pinlabel {$c$} [tr] at 2 2
\pinlabel {$\ell_c$} [r] at 2 60
\pinlabel {$m_c$} [t] at 60 2
\pinlabel {$\hat{\ell}$} [l] at 168 60
\pinlabel {$\hat{m}$} [b] at 60 230
\pinlabel {$a$} [t] at 95 221
\pinlabel {$b$} [r] at 163 265
\endlabellist
\includegraphics[width = 0.35 \textwidth]{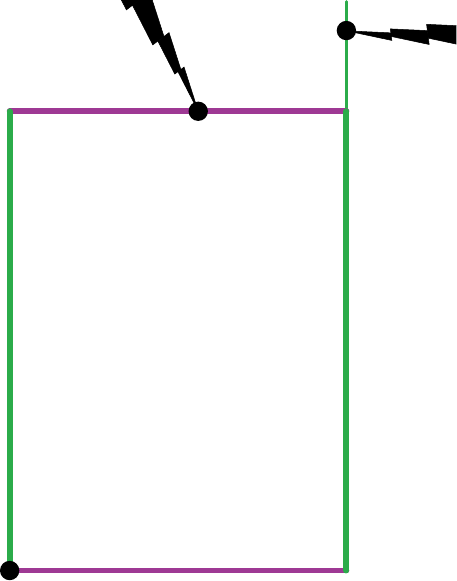}
\label{Fig:Case1}
}

\subfloat[]{
\labellist
\small\hair 2pt
\pinlabel {$c$} [tr] at 2 2
\pinlabel {$\ell_c$} [r] at 2 60
\pinlabel {$m_c$} [t] at 60 2
\pinlabel {$\hat{\ell}$} [l] at 168 60
\pinlabel {$\hat{m}$} [b] at 60 230
\pinlabel {$a$} [t] at 195 221
\pinlabel {$b$} [r] at 163 111
\endlabellist
\includegraphics[width = 0.35 \textwidth]{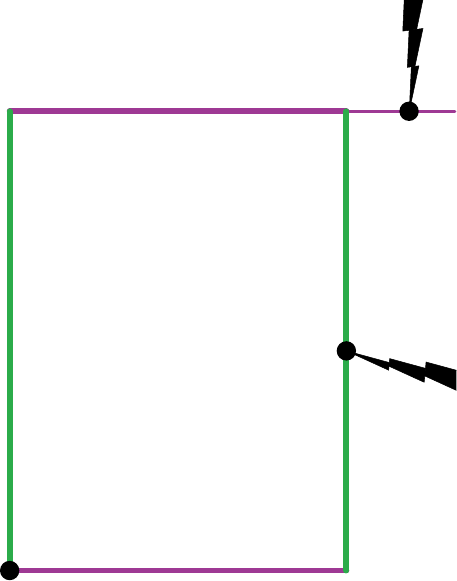}
\label{Fig:Case2}
}
\subfloat[]{
\labellist
\small\hair 2pt
\pinlabel {$c$} [tr] at 2 2
\pinlabel {$\ell_c$} [r] at 2 60
\pinlabel {$m_c$} [t] at 60 2
\pinlabel {$\hat{\ell}$} [l] at 168 60
\pinlabel {$\hat{m}$} [b] at 60 230
\pinlabel {$a$} [t] at 95 221
\pinlabel {$b$} [r] at 163 111
\endlabellist
\includegraphics[width = 0.35 \textwidth]{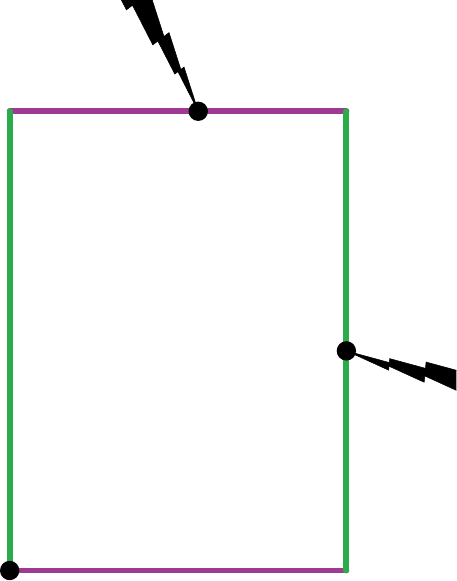}
\label{Fig:Case3}
}
\caption{The rectangles $Q(c, x)$.}
\label{Fig:FindSpanningEdgeRectangle}
\end{figure}
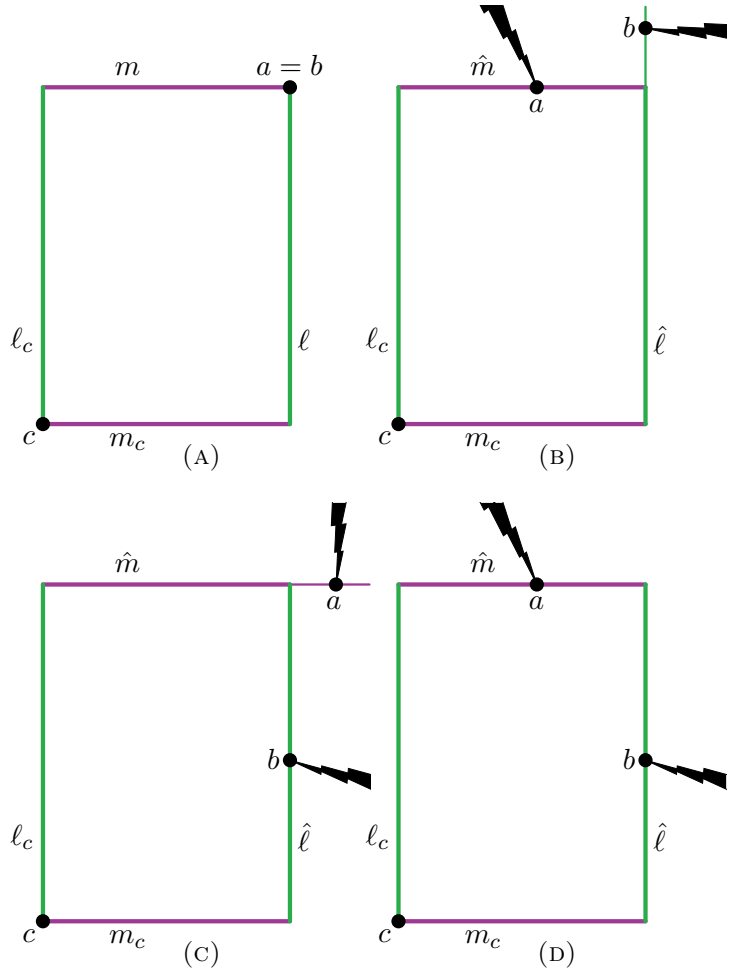

\begin{lemma}
\label{Lem:InHull}
Suppose that $x \neq y$ are singletons in $\Circle$.
Suppose that $c$ is a cusp in $\Hull(y, x)$.
Then $Q(c, x) \subset \Hull(y, x)$.
\end{lemma}

\begin{proof}
Suppose that $p = \ell_p \cap m_p$ is a point of the interior of $Q(c, x)$.
With notation as in the proof of \reflem{RectangleTowards}, we have that $m_p$ separates $m_c$ from $\hat{m}$.
Thus $m_p$ separates $c$ from $x$.
Similarly $\ell_p$ separates $c$ from $x$.
Thus $p$ is sector-between $c$ and $x$.

By \refcor{CuspBetween}, we have that $c$ is sector-between $y$ and $x$. 
Thus \reflem{BetweenTransitive} implies that $p$ is sector-between $y$ and $x$.
Thus $Q(c, x)$ is contained in $\Hull(y, x)$.
\end{proof}

We next pick out specific rectangles in $\Hull(y, x)$.

\begin{definition}
An edge rectangle $R$ contained in $\Hull(u, v)$ is \emph{south-north spanning} if its south and north sides lie in the boundary of $\Hull(u, v)$. 
\end{definition}

\begin{lemma}
\label{Lem:ConstructSpanning}
Suppose that $x \neq y$ are singletons.
Suppose that $c$ is a cusp in $\Hull(y, x)$.
Then there are cusps $d$ and $e$ of $\Hull(y, x)$ as follows.
\begin{itemize}
\item $d$ is sector-between $c$ and $x$,
\item $e$ is sector-between $d$ and $x$,
\item $R = S(d, x) \cap S(e, y)$ is an edge rectangle with cusps $d$ and $e$, and
\item $R$ is south-north spanning.
\end{itemize}
\end{lemma}

\begin{proof}
From $c$ we will generate a cusp $d$ so that $Q(d, x)$ contains $e$ and thus the desired edge rectangle $R$.
Here we use the same orientation conventions as in the proof of \reflem{RectangleTowards}.

To start, if $Q(c, x)$ is as in Figures \ref{Fig:Case0}, \ref{Fig:Case1}, and \ref{Fig:Case3} then we take $d = c$, take $e = a$, and take $R$ to be the edge rectangle with those cusps.
The case in \reffig{Case2} is more delicate; 
we proceed as follows.

Set $c^0 = c$ and $a^0 = a$.
In general, let $c^{k+1}$ be the cusp labelled $b$ and let $a^{k+1}$ be the cusp labelled $a$ in \reffig{Case2} as applied to $Q(c^k, x)$.
In all cases, $c^{k+1}$ is sector-between $c^k$ and $x$.
Induction and \reflem{BetweenTransitive} imply that $c^{k+1}$ is sector-between $c$ and $x$.

As long as the rectangles $Q(c^k, x)$ are as in case \reffig{Case2} their northern sides all lie in the leaf $m$ running through the north-west corner of $Q(c, x)$.
Thus the sequence $a^k$ is constant.
The astroid lemma (\refsec{Astroid}) implies that this process eventually terminates with a rectangle $Q(c^N, x)$ in one of the other three cases.
At that point we take $d = c^N$, take $e = a^N = a$ and set $R \subset Q(c^N, x)$ to be the edge rectangle with cusps $d$ and $e$.

In all of the above cases, the first property follows by construction and the second property follows because $e$ lies in the closure of $Q(d, x)$ (\reflem{CuspBetween}).
The construction also gives $R \subset S(d, x)$.
Since $d$ is sector-between $y$ and $x$, we have that $S(d, y)$ intersects $S(d, x)$ only at $d$.
Note that $e$ lies in the interior of $S(d, x)$.
Thus, $d$ is sector-between $y$ and $e$.
By \reflem{NestedSectors} we have that $S(d, y)$ is a proper subset of $S(e, y)$.
In particular, $d$ lies in $S(e, y)$.
Thus $R \subset S(e, y)$.
The cusp leaves on the boundaries of $S(d, x)$ and $S(e, y)$ run along the sides of $R$ and thus intersect in pairs.
Thus $R = S(d, x) \cap S(e, y)$.

Finally, we show that the rectangle $R$ is a south-north spanning edge rectangle for $\Hull(y, x)$.
By \reflem{InHull} we have that $R \subset \Hull(y, x)$.
We show that the north and south sides of $R$ lie in the boundary of $\Hull(y, x)$.
Suppose that $p$ is north of $m$.
Note that $y$ is south of $m$ because if it was north, it would be contained in the union of the two sectors at $d$ to either side of $\ell_d$.
This would mean that $y$ and $x$ would be in the same or adjacent sectors, contradicting the fact that $d$ is sector-between $y$ and $x$.
Moreover, $x$ is (strictly) south of $p$.
(In cases \ref{Fig:Case1} and \ref{Fig:Case3}, the singleton $x$ is south of $m$ thus south of $p$.
In case \ref{Fig:Case0}, the singleton $x$ may be either south or north of $m$, but in both cases $x$ is south of $p$.)
Thus $p$ is not sector-between $y$ and $x$.
A similar argument applies for points $p$ south of $m_d$.
\end{proof}

\begin{lemma}
\label{Lem:EdgeSpanningInHalfDisc}
Suppose that $\ell$ is a non-cusp leaf so that the half-disc $H(\ell)$ contains $x$ and does not contain $y$.
Then there is a south-north spanning edge rectangle for $\Hull(y, x)$ that is contained in $H(\ell)$.
\end{lemma}

\begin{proof}
\reflem{IVT} applied to $\ell$ and $x$ gives us a cusp $c$ with cusp leaves $m$, $m'$, and $m''$ so that 
\begin{itemize}
\item $m$ crosses $\ell$,
\item $m$ and $m'$ form a bi-leaf, as do $m$ and $m''$, and
\item exactly one of $m'$ or $m''$ separates $\ell$ from $x$. 
\end{itemize}

See \reffig{IVT}.

\begin{figure}[htbp]
\labellist
\footnotesize\hair 2pt
\pinlabel $y$ [r] at 11 115
\pinlabel $x$ [l] at 367 118
\pinlabel $c$ [l] at 379 187
\pinlabel $\ell$ [r] at 127 123
\pinlabel $m$ [b] at 217 190
\pinlabel $m'$ [l] at 313 129
\pinlabel $m''$ [l] at 313 245
\pinlabel $\ell'$ [r] at 236 79
\endlabellist
\includegraphics[width=0.45\textwidth]{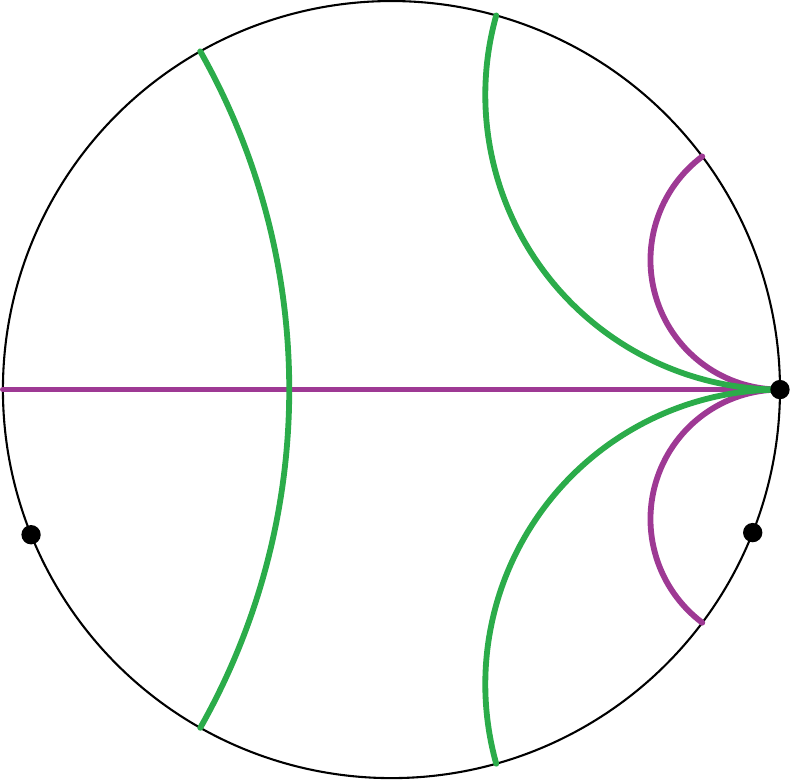}
\caption{The cusp and cusp leaves found in \reflem{IVT}. }
\label{Fig:IVT}
\end{figure}

\begin{claim*} 
$S(c, x) \subset H(\ell)$ and $c \in \Hull(y, x)$.
\end{claim*}

\begin{proof}
Breaking symmetry, suppose that $m'$ separates $\ell$ from $x$, and therefore separates $y$ from $x$.
It follows that $m'$ separates $S(c, x)$ from $\ell$.
Thus $S(c, x) \subset H(\ell)$.

Furthermore, $m'$ separates $S(c, x)$ from $y$.
Thus $S(c, y)$ is not equal to $S(c, x)$.
Let $\ell'$ be the divider of the bi-leaf containing $m$ and $m'$.
Since $\ell$ and $\ell'$ are in the same foliation, they do not cross.
Since $m$ crosses $\ell$, we deduce that $\ell'$ separates $y$ from $x$.
So the sector at $c$ bounded by $\ell'$ and $m'$ also separates $y$ from $x$.
Thus the sectors at $c$ containing $y$ and $x$ are not adjacent.
\end{proof}

From \reflem{ConstructSpanning}, we obtain the cusps $d$ and $e$ and their edge rectangle $R$. 
Since $R$ is contained in the sector $S(c, x)$ it is contained in the half-disc $H(\ell)$.
\end{proof}

\begin{lemma}
\label{Lem:SpanningEdgeRectangles}
Suppose that $x \neq y$ are singletons.
The hull $\Hull(y, x)$ contains a sequence of south-north spanning edge rectangles $(R_n)$ whose closures Hausdorff converge to $x$.
\end{lemma}
 
\begin{proof}
\refprop{SingleLeaf} provides a sequence of non-cusp leaves $(\ell_n)$ so that the half-discs $H(\ell_n)$ give a neighbourhood basis for $x$ in $\Disc$.
Pass to a tail so that for all $n$, the leaf $\ell_n$ separates $y$ from $x$.
Applying \reflem{EdgeSpanningInHalfDisc} we obtain south-north spanning edge rectangles $R_n$ contained in $H(\ell_n)$.
\end{proof}

\begin{remark}
In fact the collection of all edge rectangles south-north spanning $\Hull(y, x)$ is linearly ordered, from $y$ to $x$, ``along the hull''.
The order has exactly two accumulation points, at $y$ and at $x$.
Curiously, these accumulations may be from one side only.
\end{remark}

\subsection{Separating sectors}

By \refcor{CuspBetween}, the cusp $d$ is sector-between $y$ and $x$.
By \reflem{ConstructSpanning}, the rectangle $R$ is contained in $S(d, x)$.
Thus the cusp $y$ is not inside the sector $S(d, x)$.
Thus the following is well-defined.

\begin{definition}
\label{Def:SeparatingSectors}
Suppose that the cusps $d$ and $e$, and the rectangle $R$ are as given in \reflem{ConstructSpanning}.
We define the \emph{separating sectors} $S$ and $T$, as follows.
Let $\ell_d$ and $m_d$ be the boundary leaves of $S(d, x)$.
If $\ell_d$ separates $y$ from $R$ then let $S$ be the sector at $d$ on the opposite side of $\ell_d$ from $R$.
If instead $m_d$ separates $y$ from $R$ then let $S$ be the sector at $d$ one step further away from $R$ than the sector on the opposite side of $m_d$ from $R$.
We similarly define the sector $T$ at $e$ relative to $x$. 
\end{definition}
Assuming the orientation conventions as in the proof of \reflem{RectangleTowards}, there are four possible cases. 
These are shown in \reffig{Rainbows}.

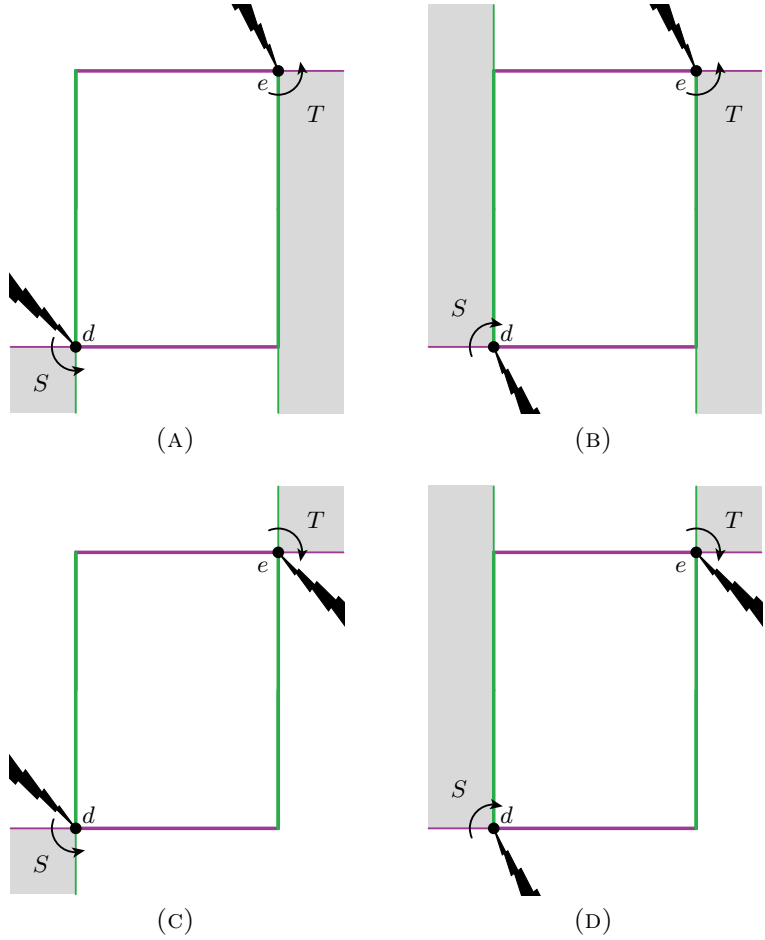
\begin{figure}[htbp]
\subfloat[]{
\labellist
\scriptsize\hair 2pt
\pinlabel {$d$} [bl] at 55 55 
\pinlabel {$e$} [tr] at 210 270
\pinlabel {$S$} at 25 25
\pinlabel {$T$} at 245 240
\endlabellist
\includegraphics[width = 0.35 \textwidth]{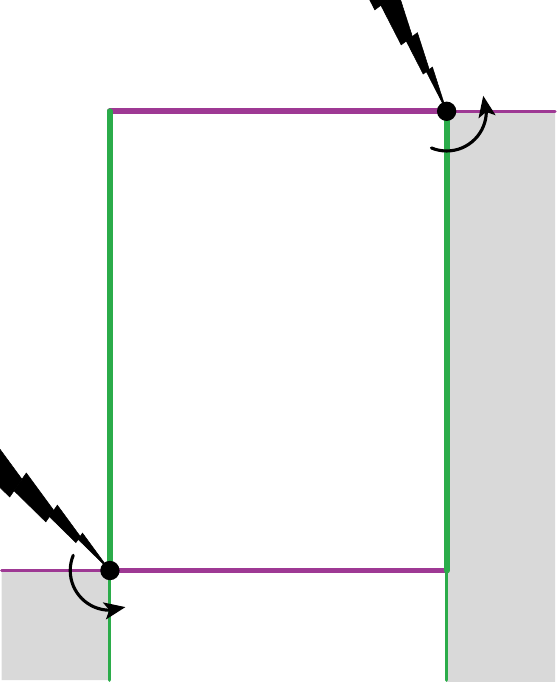}
\label{Fig:RainbowAA}
}
\qquad
\subfloat[]{
\labellist
\scriptsize\hair 2pt
\pinlabel {$d$} [bl] at 55 55 
\pinlabel {$e$} [tr] at 210 270
\pinlabel {$S$} at 25 85
\pinlabel {$T$} at 245 240
\endlabellist
\includegraphics[width = 0.35 \textwidth]{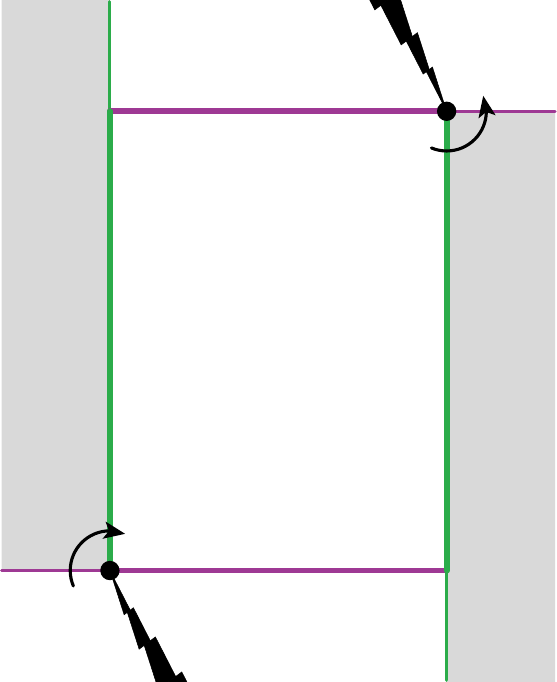}
\label{Fig:RainbowCA}
}

\subfloat[]{
\labellist
\scriptsize\hair 2pt
\pinlabel {$d$} [bl] at 55 55 
\pinlabel {$e$} [tr] at 210 270
\pinlabel {$S$} at 25 25
\pinlabel {$T$} at 245 300
\endlabellist
\includegraphics[width = 0.35 \textwidth]{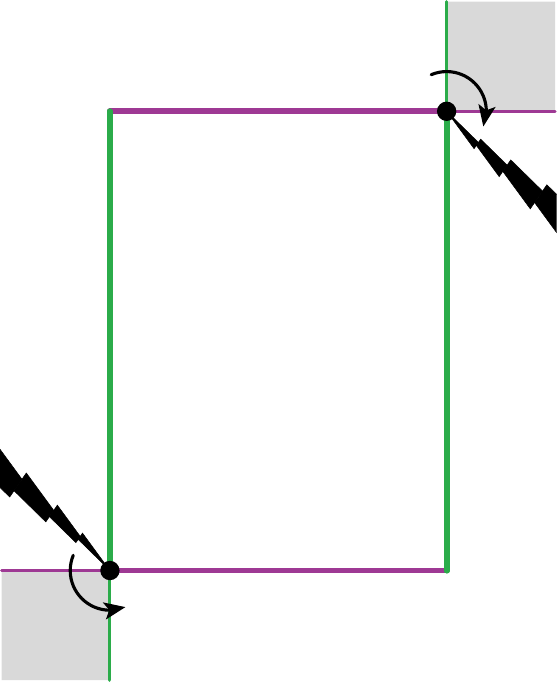}
\label{Fig:RainbowAC}
}
\qquad
\subfloat[]{
\labellist
\scriptsize\hair 2pt
\pinlabel {$d$} [bl] at 55 55 
\pinlabel {$e$} [tr] at 210 270
\pinlabel {$S$} at 25 85
\pinlabel {$T$} at 245 300
\endlabellist
\includegraphics[width = 0.35 \textwidth]{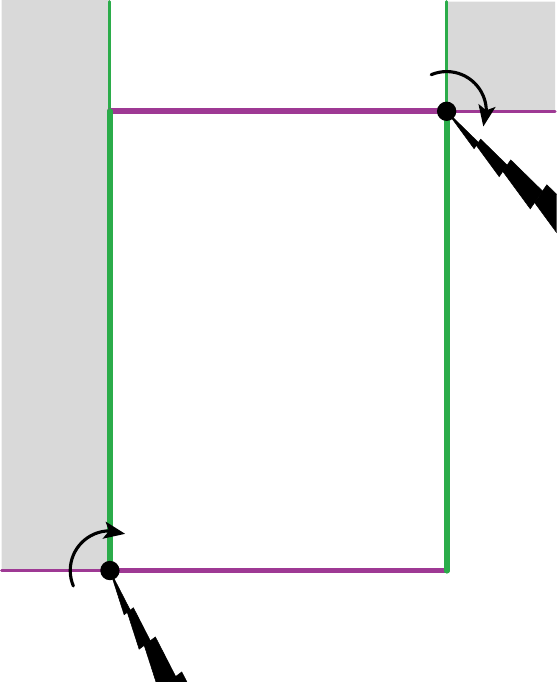}
\label{Fig:RainbowCC}
}
\caption{The four ways in which $y$ and $x$ can be arranged relative to $d$ and $e$. The arrows point from $y$ towards $x$.}
\label{Fig:Rainbows}
\end{figure}

\begin{lemma}
\label{Lem:SeparatingSectorProperties}
The separating sectors $S$ and $T$ satisfy the following.
\begin{itemize}
\item $S$ and $T$ are disjoint.
\item $S$ separates $y$ from $T$, while $T$ separates $S$ from $x$.
\end{itemize}
\end{lemma}

\begin{proof}
Suppose first that $\ell_d$ separates $y$ from $R$ (as in Figures~\ref{Fig:RainbowCA} and~\ref{Fig:RainbowCC}).
Let $m$ be the bi-leaf at $d$ with divider $\ell_d$.
Suppose for a contradiction that $y$ is north of $m$.
Then there is a non-cusp leaf $m'$ close to $m$ so that both $y$ and $x$ are north of $m'$ and $d$ is south.
But then $d$ is not sector-between $y$ and $x$, a contradiction.
It follows that in this case, $S$ separates $y$ from $R$.

Suppose instead that $m_d$ separates $y$ from $R$.
(See Figures~\ref{Fig:RainbowAA} and~\ref{Fig:RainbowAC}.)
Let $S'$ be the sector at $d$ on the opposite side of $m_d$ from $R$.
We saw in the proof of \reflem{ConstructSpanning} that no points south of $m_d$ are sector-between $y$ and $x$.
Thus $y$ cannot be in either $S$ or $S'$.

Similar arguments show that $T$ separates $R$ from $x$. 
Finally, any leaf passing through the interiors of the southern and northern sides of $R$ separates $S$ from $T$.
\end{proof}
 
\begin{proof}[Proof of \refprop{SingletonConical}]
\reflem{SpanningEdgeRectangles} gives us a sequence of south-north spanning edge rectangles $(R_n)$ that Hausdorff converge to $x$.

Recall that there are only finitely many edges in the veering triangulation $\calV$.
So pass to a subsequence $(R_n)$ so that the $R_n$ belong to a single orbit under the action of $\pi_1(M)$.
Pass to a further subsequence and reorient the foliations if necessary to ensure that each $R_n$ has its cusps $d_n$ and $e_n$ (provided in \reflem{ConstructSpanning}) at its south-west and north-east corners, respectively.
\refdef{SeparatingSectors} gives separating sectors $S_n$ based at $d_n$ and $T_n$ based at $e_n$.

Pass to a further subsequence and so find group elements $g_n$ so that, for each, we have $g_n(R_0) = R_n$ and $g_n$ preserves the orientations of the foliations.
(Note that we may need to choose a new $R_0$ to arrange this.)
Applying \refprop{Convergence}, we pass to a further subsequence so that $(g_n^{-1})$ is a collapsing sequence for the action of $\pi_1(M)$ on $\Sphere$.

We now set $y_n = g_n^{-1}(y)$ and $x_n = g_n^{-1}(x)$.
Pass to a further subsequence so that $y_n \to y_\infty$ and $x_n \to x_\infty$ both converge in $\Circle$.
We now verify the properties of \refdef{CLP} for $\pi_1(M)$ acting on $\Sphere$ for the three points $[x]_{S^1}$,  $A = [y_\infty]_{S^1}$, and $B = [x_\infty]_{S^1}$ (in that order).

Let $Y$ be the closure of the component of $\Disc - S_0$ that contains $y$.
Let $X$ be the closure of the component of $\Disc - T_0$ that contains $x$. 
Note that $y_n$, and thus $y_\infty$ lies in $Y$.
Similarly, $x_n$, and thus $x_\infty$ lies in $X$.

\begin{claim}
No decomposition element of $\Circle$ meets both $\bdyi Y$ and $\bdyi X$.
\end{claim}

\begin{proof}
By \reflem{SeparatingSectorProperties} we have that $Y$ and $X$ are separated by both $S_0$ and $T_0$. 

First, suppose that $p$ is a singleton meeting $\bdyi Y$.
Since $\bdyi Y$ is disjoint from $\bdyi X$, we deduce that $p$ does not meet $\bdyi X$.

Next, suppose that $\ell$ is a non-cusp leaf of $F^\calV$ so that $\bdyi \ell$ meets $\bdyi Y$.
(The case when $\ell$ is a non-cusp leaf of $F_\calV$ is similar.)
The leaf $\ell$ can cross at most one boundary leaf of the sector $S_0$, namely the leaf that $S_0$ shares with $Y$.
Thus $\bdyi \ell$ is contained in $\bdyi Y \cup \bdyi S_0$.
So again, $\bdyi \ell$ does not meet $\bdyi X$.

Finally, suppose that $C$ is the clamshell boundary associated to a cusp $c$. 
Let $V$ be the closure of the component of $\Disc - S_0$ that does not contain $y$.
Let $W$ be the closure of the component of $\Disc - T_0$ that does not contain $x$.
Suppose that some point of $\bdyi C$ meets $\bdyi Y$.  
There are two (overlapping) cases.
\begin{itemize}
\item 
Suppose that $c$ lies in $\bdyi Y$.
In particular, by \reflem{SeparatingSectorProperties}, the cusp $c$ is disjoint from $\bdyi T_0$.
Applying \reflem{NoneShallPass}\refitm{MeetsOne}, we have that $\bdyi C$ meets precisely one of the two components of $\Circle - \bdyi T_0$.
Since $c$ lies in $\bdyi Y$, we deduce that $\bdyi C$ meets $W$.
Thus $\bdyi C$ does not meet $\bdyi X - \bdyi T_0$.
Applying \reflem{NoneShallPass}\refitm{NoBoundary}, we are done in this case.
\item
Suppose that a thorn of $C$ lies in $\bdyi Y$. 
If no point of $C$ meets $\bdyi V$ then we are done.
Otherwise, by the contrapositive of \reflem{NoneShallPass}\refitm{MeetsOne}, we have that $c$ lies in $\bdyi S_0$. 
Thus by \reflem{SeparatingSectorProperties} we have that $c$ does not lie in $\bdyi T_0$ and we are done as before. \qedhere
\end{itemize}
\end{proof}

We conclude that $A \neq B$.

Suppose that $y' \neq y$ is another singleton separated from $x$ by $S_0$.
The hulls $\Hull(y, x)$ and $\Hull(y', x)$ are not equal; 
however they coincide on the component of $\Disc - S_0$ containing $x$.
Thus $\Hull(y, x)$ and $\Hull(y', x)$ are both spanned by the edge rectangle $R_n$, for all $n$.
Thus the construction above, applied to $y'$, produces the same collapsing sequence $(g^{-1}_n)$.
Setting $A' = [y'_\infty]_{S^1}$, we find as above that $A' \neq B$.
As in the proof of \refprop{LeafConical}, we are done.
\end{proof}

\section{Proving the theorem}
\label{Sec:MainProof}

We gather the propositions together to prove our main result. 

\begin{theorem}
\label{Thm:Convergence}
Suppose that $\calV$ is a finite locally veering triangulation of a three-manifold $M$. 
Then the action of $\pi_1(M)$ on the veering two-sphere $\Sphere$ is a geometrically finite convergence action.  
Furthermore, its parabolic points are exactly the cusps $\Delta_\calV$.
\end{theorem}

\begin{proof}
By \refprop{Convergence} the action is a convergence action.
Suppose that $[x]_{S^1}$ is a point of $\Sphere$. 
By \refprop{Parabolic} if $x = c$ is a cusp, then $[c]_{S^1}$ is a bounded parabolic point for the action.
By Propositions~\ref{Prop:LeafConical} and~\ref{Prop:SingletonConical} if $x$ is the end of a leaf, or a singleton, then $[x]_{S^1}$ is a conical limit point for the action.
By \reflem{Laminations} these are all of the points of $\Sphere$. 
Thus the assumptions of \refdef{GeoFinConv} are satisfied.
\end{proof}

Combining Theorems~\ref{Thm:HyperbolicGeomFinite},~\ref{Thm:Yaman}, and~\refthm{Convergence}, and recalling \refrem{VeeringHyperbolic}, we have the following. 

\begin{corollary}
\label{Cor:Yaman}
Suppose that $\calV$ is a finite locally veering triangulation of a three-manifold $M$. 
Then the veering two-sphere $\Sphere$ is equivariantly homeomorphic to $\bdyi \HH^3$ (as equipped with the action of $\pi_1(M)$ coming from the discrete and faithful representation). \qed
\end{corollary}

\section{From veering triangulations to Cannon--Thurston maps, and back again}
\label{Sec:CannonThurston}

Here we discuss the application of our work to the theory of Cannon--Thurston maps~\cite{CannonThurston07}.
We begin with a generalisation of the original definition.

\begin{definition}
\label{Def:CT}
Suppose that $M$ is a finite-volume hyperbolic three-manifold.
Let  $S^2_M \homeo \bdyi \HH^3$ be the associated two-sphere at infinity, equipped with its canonical action of $\pi_1(M)$.
Suppose that $S^1_M$ is a circle, equipped with an action of $\pi_1(M)$.
A \emph{Cannon--Thurston map} for $M$ is a continuous equivariant map $\Phi \from S^1_M \to S^2_M$.  
\end{definition}



Recall that all orbits of the $\pi_1(M)$--action on $S^2_M$ are dense.
Thus the Cannon--Thurston map $\Phi$ is surjective.

\begin{definition}
\label{Def:Equivalence}
Suppose that $\Phi_M$ and $\Phi_N$ are Cannon--Thurston maps for $M$ and $N$, respectively.
We say that $\Phi_M$ and $\Phi_N$ are \emph{equivalent} if there are homeomorphisms
$f \from S^1_M \to S^1_N$ and $g \from S^2_M \to S^2_N$ so that $g \circ \Phi_M = \Phi_N \circ f$.
\end{definition}

\begin{remark}
\label{Rem:CommensurableEquivalent}
In particular, suppose that $(M, \calV)$ and $(N, \calW)$ are commensurable:
they have a common cover.
Then $\Psi_\calV$ and $\Psi_\calW$ are equivalent.
\end{remark}

We obtain Cannon--Thurston maps from locally veering triangulations (see \refrem{LocallyVeering}) by applying \refcor{Yaman}.

\begin{theorem}
\label{Thm:VeeringCT}
Suppose that $\calV$ is a finite locally veering triangulation of a three-manifold $M$. 
Then composing the surjection $\Circle \to \Sphere$ and the equivariant homeomorphism given by \refcor{Yaman} gives a Cannon--Thurston map $\Psi_\calV \from \Circle \to \bdyi \HH^3$.
\end{theorem}

\begin{proof}
Taking a finite cover if needed, we obtain a transverse veering triangulation $(M',\calV')$.
Since our constructions are insensitive to taking covers, we obtain 
$S^1(\calV)$, the veering circle, and $S^2(\calV)$, the veering two-sphere.

Recall from \refdef{VeeringSphere} that $\Psi_\calV$ is an equivariant surjection from $S^1(\calV)$ to $S^2(\calV)$.
By \refcor{Yaman}, we have that $S^2(\calV)$ is equivariantly homeomorphic to $S^2_M \homeo \bdyi \HH^3$.
Composing these maps gives the Cannon--Thurston map $\Psi_{\calV} \from S^1(\calV) \to S^2_M$ \emph{induced by} $\calV$.
\end{proof}

In the context of locally veering triangulations, we obtain a converse to \refrem{CommensurableEquivalent}.

\begin{theorem}
\label{Thm:Equivalence}
Suppose that $M$ and $N$ are finite-volume hyperbolic three-manifolds.
Suppose that $\calV$ is a locally veering triangulation of $M$.
Let $\Psi_\calV$ be the Cannon--Thurston map for $M$ induced by $\calV$.
Suppose that~$\Phi$ is a Cannon--Thurston map for $N$.
Suppose that $\Psi_\calV$ is equivalent to~$\Phi$.
Then $N$ is commensurable with $M$.
Furthermore, there is a locally veering triangulation $\calW$ on $N$ so that $\Phi$ is equivalent to $\Psi_\calW$ and $(N, \calW)$ is commensurable with $(M, \calV)$.
\end{theorem}

\begin{remark}
That is, up to commensurability, a locally veering triangulation can be recovered from its Cannon--Thurston map.
\end{remark}

\begin{proof}[Proof of \refthm{Equivalence}]
As in \refthm{VeeringCT}, our constructions are insensitive to taking covers.
Thus we may assume that $M$ and $N$ are oriented.
We can also assume that $\calV$ is transverse veering.
The fibres of $\Psi_\calV$ are the decomposition elements set out in \refdef{DecompositionElementsCircle}.
From the infinite fibres, and applying \reflem{DecompositionElementsRecover}, we recover the upper and lower laminations $\Lambda^\calV$ and $\Lambda_\calV$ (up to swapping them).

In~\cite[Chapter~10]{FSS22} we construct the link space $\Link$ from the laminations $\Lambda^\calV$ and $\Lambda_\calV$.
By \refthm{LinkIsLoom}, the link space is a loom space.
Moreover, the action of $\pi_1(M)$ is by homeomorphisms preserving the loom space structure. 
Since the Cannon--Thurston maps $\Psi_\calV$ and $\Phi$ are equivalent, after transporting the action of $\pi_1(N)$, 
we find that it also acts by homeomorphisms preserving the loom space structure on $\Link$.

Recall that $\cover{\calV}$ is the induced ideal triangulation of the universal cover.

\begin{claim}
\label{Clm:Faithful}
$\pi_1(N)$ acts faithfully on $\cover{\calV}$.  
\end{claim}

\begin{proof}
Consider $\Delta_\calV$, the cusps of $\cover{\calV}$.  
These give a countable dense subset of $S^1(\calV)$.  
We claim that these can be recovered from $\Psi_\calV$ (and thus from $\Phi$).  
To see this, note that the decomposition elements of cusps are exactly the infinite fibres of the map $\Psi_\calV$. 
Also, the cusp $c$ is the unique accumulation point in $S^1(\calV)$ of the decomposition element~$[c]_{S^1}$.

Thus $\Delta_\calV \subset S^1(\calV)$ is also preserved by the $\pi_1(N)$ action.  
Suppose that $\gamma \in \pi_1(N)$ lies in the kernel of the action on $\cover{\calV}$.  
Thus $\gamma$ fixes all of $\Delta_\calV$, a dense subset of $S^1(\calV)$.  
Thus $\gamma$ fixes $S^1(\calV)$ pointwise.  
We deduce that $\gamma$ fixes $S^2(\calV)$ and thus $S^2_N$ pointwise.  
By \refdef{CT} we have that $S^2_N = \bdyi \cover{N}$.
However, the only element of $\pi_1(N)$ that acts trivially on $\bdyi \cover{N}$ is the identity.
\end{proof}

Let $\Gamma$ be the group of all combinatorial automorphisms of $\cover{\calV}$.
This contains $\pi_1(M)$ as a finite-index subgroup.
By \refclm{Faithful}, we have an embedding of $\pi_1(N)$ into $\Gamma$.
Let $\hat\Gamma$ be the intersection of $\pi_1(M)$ and $\pi_1(N)$ inside of $\Gamma$.
Thus $\hat\Gamma$ is finite index in $\pi_1(N)$.
Let $\hat{M}$ and $\hat{N}$ be the induced covers of $M$ and $N$, respectively. 
These are hyperbolic manifolds with identical fundamental groups; 
so they are homotopy equivalent.
Note that $\hat{N}$ is a finite cover of $N$.

\begin{claim}
\label{Clm:Finite}
$\hat{M}$ is a finite cover of $M$.
\end{claim}

Our proof is essentially a combination of~\cite[page~92]{Kapovich10} and~\cite[Section~1]{BieriEckmann78}. 


\begin{proof}[Proof of \refclm{Finite}]
It suffices to prove that the hyperbolic volume of $\hat{M}$ is finite.

Since $M$ is cusped, it deformation retracts to a two-complex.
Lifting, we find that $\hat{M}$ deformation retracts to a two-complex.
So $H_3(\hat{M}) \isom 0$. 
Thus $H_3(\hat{N}) \isom 0$.
We deduce that $\hat{N}$ is not closed.
Since the hyperbolic volume of $\hat{N}$ is finite, we deduce that $\hat{N}$ has (finitely many) cusps.

Let $\calB$ be a collection of disjoint closed cusp neighbourhoods $B_i$ in $\hat{N}$, one per cusp of $\hat{N}$.
Set $Q_i = \pi_1(B_i)$.
Since $\hat{N}$ is oriented each $Q_i$ is isomorphic to $\ZZ^2$.
Let $R_i < \hat\Gamma$ be the images of the $Q_i$. 
Set $\calR = (R_i)$. 
We note that 
\[
H_3(\hat{\Gamma}, \calR) \isom H_3(\hat{N}, \calB) \isom \ZZ
\]
where the first isomorphism follows from~\cite[Theorem~1.3]{BieriEckmann78}. 


We transport the groups $R_i$ to $\pi_1(\hat{M})$; 
take $P_i$ to be the image of $R_i$.
Each $P_i$ gives us a parabolic subgroup;
no pair of these is conjugate. 
So we obtain (finitely many) distinct cusps and thus disjoint cusp neighbourhoods $A_i \subset \hat{M}$.
Set $\calA = (A_i)$. 

Finally, we have 
\[
H_3(\hat{M}, \calA) \isom H_3(\hat{\Gamma}, \calR) \isom \ZZ
\]
Thus $\hat{M}$ is compact relative to a finite set of cusps.
So $\hat{M}$ has finite volume. 
\end{proof}

We apply Mostow-Prasad rigidity~\cite[Theorem~B]{Prasad73} to deduce that $\hat{M}$ is isometric to, thus homeomorphic to, $\hat{N}$. 
(Alternatively, we could use the proof of \refclm{Finite} to obtain a homotopy equivalence of space pairs; 
then apply~\cite[Theorem~6.1]{Waldhausen68}.)
Thus $M$ and $N$ are commensurable.

Finally, we define $\calW = \cover{\calV} / \pi_1(N)$.
\end{proof}

\subsection{New Cannon--Thurston maps}

From \reflem{CommensurablePreservesLayered} and \refthm{Equivalence} we deduce the following.

\begin{corollary}
\label{Cor:NotMadeThere}
Suppose that $M$ and $N$ are finite-volume hyperbolic three-manifolds.
Suppose that $\calV$ is a locally veering triangulation for $M$.
Let $\Psi_\calV$ be the Cannon--Thurston map for $M$ induced by $\calV$.
Suppose that~$\Phi$ is a Cannon--Thurston map for $N$.
Suppose that $\Psi_\calV$ is equivalent to~$\Phi$.
Then $\calV$ is layered if and only if $\Phi$ is induced by a fibring of $N$. \qed
\end{corollary}

\begin{example}
\label{Exa:s227}
Suppose that $M$ is the SnapPy~\cite{snappy} census manifold \usebox{\NonfibredSnappy}.
By~\cite[Section~4]{HRST11}, the manifold $M$ is non-fibred. 
Suppose that $\calV$ is the veering triangulation \usebox{\NonfibredVeer} of $M$.  
By \refcor{NotMadeThere}, the Cannon--Thurston map $\Psi_\calV$ is not equivalent to any Cannon--Thurston map arising from a cusped fibred manifold.
\end{example}

Thus we have given new Cannon--Thurston maps in the cusped non-fibred case.

\appendix
\section{The action on the veering circle \\ recovers the veering triangulation}
\label{App:Action}

In \refthm{Equivalence} we showed that the veering triangulation can be recovered from the Cannon--Thurston map. 
Here we show that it can be recovered from just the action of the fundamental group on the veering circle.
This is more delicate: we can no longer recover leaves from fibres of the Cannon--Thurston map. 
Instead we recover leaves from the dynamics of the action. 

\begin{theorem}
\label{Thm:Isomorphism}
Suppose that $\calV$ and $\calW$ are finite transverse veering triangulations of a three-manifold $M$.
Suppose that the actions of $\pi_1(M)$ on $S^1(\calV)$ and $S^1(\calW)$ are topologically conjugate.
Then $\calV$ and $\calW$ are isomorphic as taut triangulations.
Furthermore, if the topological conjugacy is the identity on $\Delta_M$ then $\calV$ is isotopic to $\calW$.
\end{theorem}


We begin with a bit of necessary dynamical background.

\begin{definition}
Suppose that $g \from S^1 \to S^1$ is a homeomorphism. 
We say that $g$ has \emph{north-south dynamics} if $g$ has exactly two fixed points, with one attracting and the other repelling.
We say that $g$ has \emph{north-south-north-south dynamics} if $g$ has exactly four fixed points, with two attracting and two repelling. 
\end{definition}

Note that, when $g$ has north-south-north-south dynamics, its attractors and repellers are necessarily linked.
We now give a consequence of our main result, \refthm{Convergence}.

\begin{proposition} 
\label{Prop:FixedPoints}
Suppose that $\calV$ is a finite locally veering triangulation of a three-manifold $M$.
Let $S^1(\calV)$ be the veering circle.  
Every non-trivial element $\gamma$ of $\pi_1(M)$ satisfies exactly one of the following.
\begin{enumerate}
\item 
\label{Itm:Peripheral}
$\gamma$ is peripheral.
\item 
\label{Itm:NS}
$\gamma$ has north-south dynamics on $S^1(\calV)$.
Here the fixed points of $\gamma$ are singletons.
\item 
\label{Itm:NSNS}
$\gamma^2$ has north-south-north-south dynamics.
Here the attractors of $\gamma^2$ are the ends of a leaf (of $\Lambda^\calV$ or $\Lambda_\calV$) and likewise for the repellers.
\end{enumerate}
\end{proposition}

\begin{proof}
Suppose first that $\gamma$ is peripheral (so satisfies \refitm{Peripheral}) and thus fixes a cusp.
By \refprop{ParabolicFacts} the action of $\gamma$ on $S^1(\calV)$ has either infinitely many fixed points or has exactly one. 
Thus we do not satisfy \refitm{NS} or \refitm{NSNS}.

Suppose now that $\gamma$ is not peripheral.
By \refcor{Yaman}, the action of $\pi_1(M)$ on the veering two-sphere $\Sphere$ is topologically conjugate to its action on $\bdyi \HH^3$. 
Let $\rho \from \pi_1(M) \to \Isom(\HH^3)$ be the holonomy representation.
Thus $\rho(\gamma)$ is loxodromic. 
(It is possibly a glide reflection, but this does not affect the arguments.)
Thus the action of $\gamma$ on $\Sphere$ has exactly two fixed points;
one, say $A$, is attracting and one, say $B$, is repelling. 
Applying \refthm{Convergence} we have that neither $A$ nor $B$ is a cusp. 

We now have various cases as $A$ is a singleton or the pair of ends of a leaf, and similarly for $B$.
If $A$ is a singleton, it contributes one fixed point to the action of $\gamma$ on $\Circle$.
If $A$ is the pair of ends of a leaf, then the ends are either both fixed or are interchanged by the action of $\gamma$ on $\Circle$.

First suppose that $A = \{x\}$ and $B = \{y\}$ are singletons.  
Since $A$ is attracting for the action of $\gamma$ on $S^2(\calV)$, the continuity of the Cannon--Thurston map $\Psi_\calV$ implies that $x$ is attracting for the action of $\gamma$ on $S^1(\calV)$.
Similarly, $y$ is repelling.
Thus we satisfy \refitm{NS}.

Replacing $\gamma$ by $\gamma^{-1}$ if needed, we next suppose that $A = \{x, y\}$ where $x$ and $y$ are the ends of a leaf in the veering circle.  
Squaring $\gamma$ if necessary, we can assume that $\gamma$ (as it acts on $S^1(\calV)$) fixes $x$ and $y$ individually.  
Similar to the above,
we deduce that $x$ and $y$ are attracting fixed points for the action of $\gamma$ on $S^1(\calV)$.  
We deduce that $\gamma$, acting on $S^1(\calV)$, has (at least)
two more fixed points: one in $(x, y)^\acw$ and one in $(y, x)^\acw$.  
Thus by the above it has precisely four fixed points.
So $B$ is also the ends of a leaf, and we are done.
\end{proof}


We are now equipped to prove \refthm{Isomorphism}.

\begin{proof}[Proof of \refthm{Isomorphism}] 
By \refprop{FixedPoints} and \refprop{ParabolicFacts}, an element of $\pi_1(M)$ has infinitely many fixed points as it acts on $S^1(\calV)$ if and only if it is a non-trivial power of $\beta$ for some cusp $c$.
Applying \reflem{DecompositionElementsRecover}, we recover the upper and lower laminations $\Lambda^\calV$ and $\Lambda_\calV$.
In~\cite[Chapter~10]{FSS22} we construct the link space $\Link$ from the laminations $\Lambda^\calV$ and $\Lambda_\calV$.
We do the same for $\calW$.
We deduce that the link spaces of $\calV$ and $\calW$ are isomorphic.
Furthermore, the conjugacy takes the action on the link space $\link(\calV)$ to that on $\link(\calW)$.
This gives a bijection of orbits of tetrahedron rectangles as well as their adjacencies. 

Suppose that the conjugacy is the identity on $\Delta_M$.
In this case, the isomorphism of link spaces built in the previous paragraph is in fact the identity.
The resulting isomorphism of triangulations induces a homeomorphism $h \from M \to M$.
The isomorphism preserves all cusps of $\Delta_M$.
Thus $h$ sends homotopy classes of edges to themselves.
Using the hyperbolic geometry on $M$, we straighten $h$ and find that it is homotopic to the identity.
Applying~\cite[Theorem~7.1]{Waldhausen68}, we find that $h$ is isotopic to the identity.
\end{proof}

\section{Pseudo-Anosov homeomorphisms \\ with internal singularities}
\label{App:InternalSing}

We record the following for future work on classifying veering triangulations.

Suppose that $M$ is a cusped hyperbolic three-manifold.
Suppose that $\calT$ is a layered triangulation of $M$.
Suppose that $\calO_\calT$ is the induced circular order on $\Delta_\calT$~\cite[Lemmas~2.17 and~2.20]{FSS22}.
Suppose that $f \from S \to S$ is any pseudo-Anosov homeomorphism carried by the layered triangulation $\calT$.
Note that $S$ is of finite type and has punctures.
Let $\cover{S} \homeo \HH^2$ be its universal cover.

Let $L^f$ and $L_f$ be the full preimages (in $\cover{S}$) of the stable and unstable geodesic laminations for $f$.
In fact, the laminations $L^f$ and $L_f$ are independent of the choice of $f$; 
accordingly we denote them by $L^\calT$ and $L_\calT$.
We say that $L^\calT$ has \emph{internal singularities} if some component of $\HH^2 - L^\calT$ is a finite sided ideal polygon.


\begin{lemma}
Suppose that $M$ is a cusped hyperbolic three-manifold.
Suppose that $\calT$ is a layered triangulation of $M$.
Suppose that $L^\calT$ has internal singularities.
Then $\calO_\calT$ is not equal to $\calO_\calV$ for any veering triangulation $\calV$ of $M$.
\end{lemma}

\begin{proof}
We prove the contrapositive.
Suppose that $\calO_\calT = \calO_\calV$ for some veering triangulation $\calV$ of $M$.
Thus their order closures agree, giving $\Circle$.
By \refprop{FixedPoints}, the action of $\pi_1(M)$ on $\Circle$ recovers the decomposition elements of cusps.
Applying \reflem{DecompositionElementsRecover}, we recover the upper and lower laminations $\Lambda^\calV$ and $\Lambda_\calV$.

By~\cite[Theorem~4]{Thurston88}
we can recover the boundary $\bdy L^\calT$ from the action of (the various lifts of) $f$ on $\bdy \cover{S}$.
Since $\bdy \cover{S}$ is naturally homeomorphic to the order completion of $\calO_\calT$, and since $f$ appears as an element of $\pi_1(M)$, we have that 
\begin{itemize}
\item $\bdy L^\calT = \Lambda^\calV$ and $\bdy L_\calT = \Lambda_\calV$ or 
\item $\bdy L^\calT = \Lambda_\calV$ and $\bdy L_\calT = \Lambda^\calV$.
\end{itemize}
By \refcor{CircleDecomposition}, we find that $L^\calT$ has no finite-sided ideal polygons.
Thus $L^\calT$ has no internal singularities.
\end{proof}

\section{Drawing approximations \\ of Cannon--Thurston maps}
\label{App:Drawing}

\begin{center}
\textsc{Saul Schleimer and Henry Segerman}
\end{center}

Thurston generated the first approximations to Cannon--Thurston maps
(for certain punctured torus bundles over the circle).
See the last frame of~\cite[Figure~8]{Thurston82} and also~\cite[Figure~12]{Thurston82}.
The first of these is reproduced in \reffig{Thurston}.
Other approximations in the literature include work by Dicks, McMullen (in the closed case), Wada, and Wright.
For an extensive discussion and references, see~\cite[Sections~4 and~6.4]{BachmanGoernerSchleimerSegerman22}.

\begin{figure}[htbp] 
\includegraphics[width = 0.6 \textwidth]{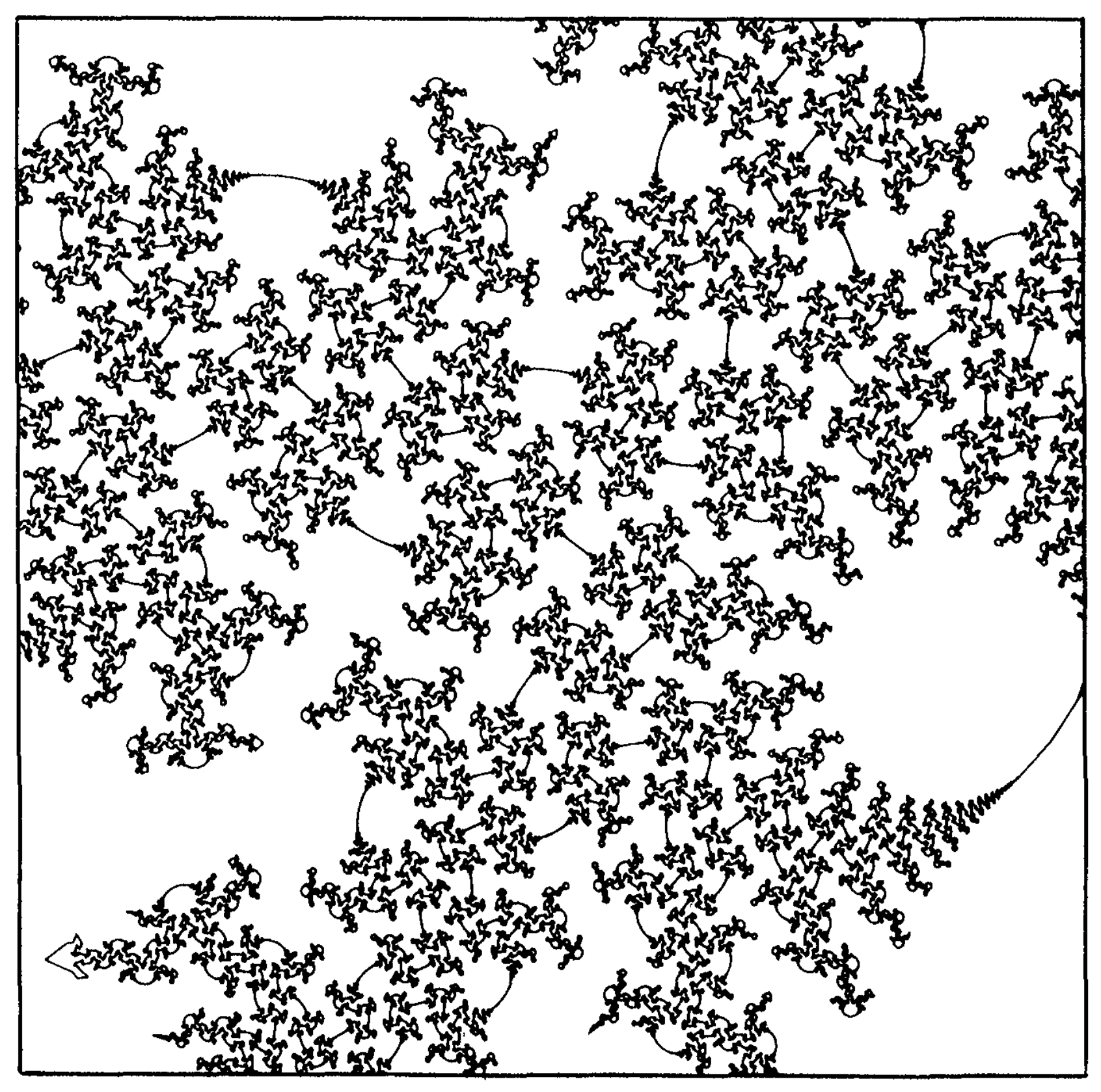}
\caption{The last frame of~\cite[Figure~8]{Thurston82}.
An approximation to (a part of) the Cannon--Thurston map for the complement of the figure-eight knot.}
\label{Fig:Thurston}
\end{figure}

As a running example, let $M$ be the complement of the figure-eight knot in the three-sphere. Let $F$ be a Seifert surface of the knot (a punctured torus in $M$). 
Fix an elevation $\cover{F} \subset \cover{M}$.
Then the Cannon--Thurston map is the map
\[
\Phi \from S^1 \homeo \bdyi \cover{F} \to \bdyi \cover{M} \homeo S^2
\]
Since $\Phi$ is a surjection onto $S^2$, its image does not illustrate its structure.
Instead, we draw finite approximations.

\subsection{The disc method}

Our understanding of Thurston's algorithm is as follows.
To approximate $\Phi$, Thurston builds a large disc $D$ in $\HH^3$.
The disc $D$ is a union of finitely many ideal hyperbolic triangles, with boundary a union of geodesics in $\HH^3$.
Their ideal endpoints lie in $\bdyi \HH^3$.
The \emph{bending angles} between adjacent triangles are carefully chosen so that $D$ is isometric to a subdisc of $\cover{F}$.
(Recall that $\cover{F}$ is the elevation of $F$, defined above.
The surface $F$ is realised geometrically as a subsurface of the two-skeleton of the canonical geometric triangulation of the figure-eight knot complement.)
Thurston's approximation is the result of projecting $\bdy D$ into $\bdyi \HH^3$.

\subsection{The continent method}

The disc method does not apply to veering Cannon--Thurston maps in general;
this is because the underlying manifold $M$ may not be fibred.
So, we replace the disc $D$ with a \emph{continent} in $\cover{M}$, in the sense of~\cite[Definition~4.15]{FSS22}. 
A continent $C$ is a finite face-connected collection of tetrahedra of $\cover{\calT}$ which admits a layering.
We refer to the layers of a continent as its \emph{landscapes};
these are all discs.
They share a common boundary, which we call the \emph{coast} of the continent.
The continent $C$ is a closed three-ball, minus finitely many points on its coast, corresponding to its cusps~\cite[Corollary~4.23]{FSS22}. 
The circle of geodesics making up the coast replaces the boundary of the disc $D$ in Thurston's procedure.

We have implemented the construction of continents and production of approximations to Cannon--Thurston maps in~\cite{VeeringCodebase}.
Example images of approximations to Cannon--Thurston maps for veering triangulations with up to 12 tetrahedra are available on the webpage for the census of veering triangulations~\cite{GSS19}.

\begin{remark}
Suppose that the manifold $M$ has a layered veering triangulation $\calV$.
Suppose that $D$ is any disc carried by the horizontal branched surface (as is the case in Thurston's examples of punctured torus bundles).
Then there is a continent $C$ whose cusps include the ideal points of $D$, in the same circular order.
Thus if we choose the continent appropriately, we exactly recover Thurston's examples.
\end{remark}

\subsection{Filling in the details}
\label{Sec:DrawingAlgorithm}

When making a line drawing, there are (at least) two relevant parameters:
\begin{itemize}
\item 
a \emph{feature size} $\epsilon$ such that details smaller than $\epsilon$ cannot be discerned, and
\item 
an amount of work we are willing to do as measured by, say, the number $N$ of line segments we are willing to draw.
\end{itemize} 

When drawing an approximation to a Cannon--Thurston map, using the disc method, one must build the disc $D$.
The naive algorithm has inputs
\begin{itemize}
\item
a maximal depth $d_{\max}$ and 
\item
an initial triangle $f_0$ in $\cover{F}$.
\end{itemize}
If $e$ is an edge of $\cover{F}$ then we define its \emph{depth} to be its combinatorial distance (in $\cover{F}$) from $f_0$.

The algorithm starts with $D = f_0$.
We map the vertices of $f_0$ to $\infty$, $0$, and $1$ in $\bdyi \HH^3$.
(This fixes a developing map, and thus the position of all vertices.)
Next, for every edge $e$ in $\bdy D$ we do the following.
Suppose that $f$ is the triangle in $\cover{F}$ adjacent to $e$ and not in $D$.
We replace $D$ by $D \cup f$ if
\begin{itemize}
\item
the depth of $e$ is less than $d_{\max}$.
\end{itemize}
(Thus the amount of work we do is proportional to $\exp(d_{\max})$.)

Note that the naive algorithm does not produce the same level of detail everywhere along its curve. 
As the maximal depth increases, more and more work is done inside of single pixels, and is thus invisible.

\subsection{Thurston's algorithm}
\label{Sec:ThurstonAlgorithm}

From our experiments, we understand Thurston's algorithm to be the following. 
It is a modification of the naive algorithm.
There is an additional input of
\begin{itemize}
\item
a feature size $\epsilon$. 
\end{itemize}
Let $p_0$ be the projection of $1$ to the edge $e_0$ between $0$ and $\infty$.
Using the ball model, this equips $\bdyi \HH^3$ with a spherical metric.
If $e$ is any edge of $\cover{F}$ then we define its \emph{length} to be the spherical distance between its endpoints.\footnote{One can instead use the upper-half space model here, and measure euclidean distance in $\CC$.
It is very difficult to tell which Thurston used.}

In the recursive step, when considering an edge $e$, we now replace $D$ by $D \cup f$ if
\begin{itemize}
\item
the depth of $e$ is less than $d_{\max}$ and
\item
the length of $e$ is greater than the feature size $\epsilon$.
\end{itemize}

Thurston~\cite[page~373]{Thurston82} says that the resulting
\begin{quotation}
\emph{curve actually fills $\CC$, but to make the computer actually fill $\CC$ would
require an amount of computer money (and foolishness) approximating the
``defense'' budget.\footnote{We believe that here by ``$\CC$'' Thurston in fact means $\CP^1$.
Since $\CC$ is not compact it can never be filled by any algorithm,
regardless of funding.}}
\end{quotation}
Indeed, we can see large blank regions in \reffig{Thurston}.

One measure of how well a curve $\phi$ fills $\bdyi \HH^3$ is the \emph{injectivity radius} of $\bdyi \HH^3 - \phi$: 
the spherical radius of the largest round disc with interior disjoint from the curve.
In \reffig{Thurston} we see that the injectivity radius is many times larger than the feature size $\epsilon$.

\begin{remark}
We can estimate $\epsilon$ by the maximal distance between consecutive vertices.
Thus $\epsilon$ in \reffig{Thurston} seems to be smaller than a pixel.
Note that there are some visible line segments near the lower left corner.
These happen at leaves of the search tree that hit maximal depth before hitting minimal feature size.
Thurston seems to have removed this part of the stopping criterion in his more recent image of the Cannon--Thurston map for the figure eight knot~\cite[Figure~2]{Thurston98}.
\end{remark}

Under the Cannon--Thurston map $\Phi$, each parabolic fixed point has countably many preimages: 
one cusp and two countable collections of interleaving thorns (see \reflem{Laminations}\refitm{Interleave} and \reffig{Laminations}).
We take a small interval (in the circle) about a thorn.
The image of this interval, under $\Phi$, is a neighbourhood (with fractal boundary) of an arc of a circle (in $\bdyi \HH^3$).
The blue shaded region in \reffig{CT_fig8_neck} is an approximation of such a neighbourhood.
Clearly, Thurston's approximation $\phi$ has trouble entering these neighbourhoods.

This is due to the \emph{problem of necks} raised by Mumford, Series, and Wright~\cite[page~184, 248--252]{MSW02}.
A \emph{thin neck} is an edge $e$ with endpoints at cusps $c$ and $d$ so that
\begin{itemize}
\item $e$ has length less than $\epsilon$, but
\item the image of the arc $[c,d]^\acw$ under $\Phi$ has large diameter.
\end{itemize}
\reffig{CT_fig8_neck} illustrates the problem of thin necks (for Thurston's algorithm). 
In fact, for every cusp, and for every associated thorn, there will eventually be such a thin neck.
This is why Thurston's algorithm yields curves with relatively large injectivity radius.

\begin{figure}[htbp]
\includegraphics[width = 0.8\textwidth]{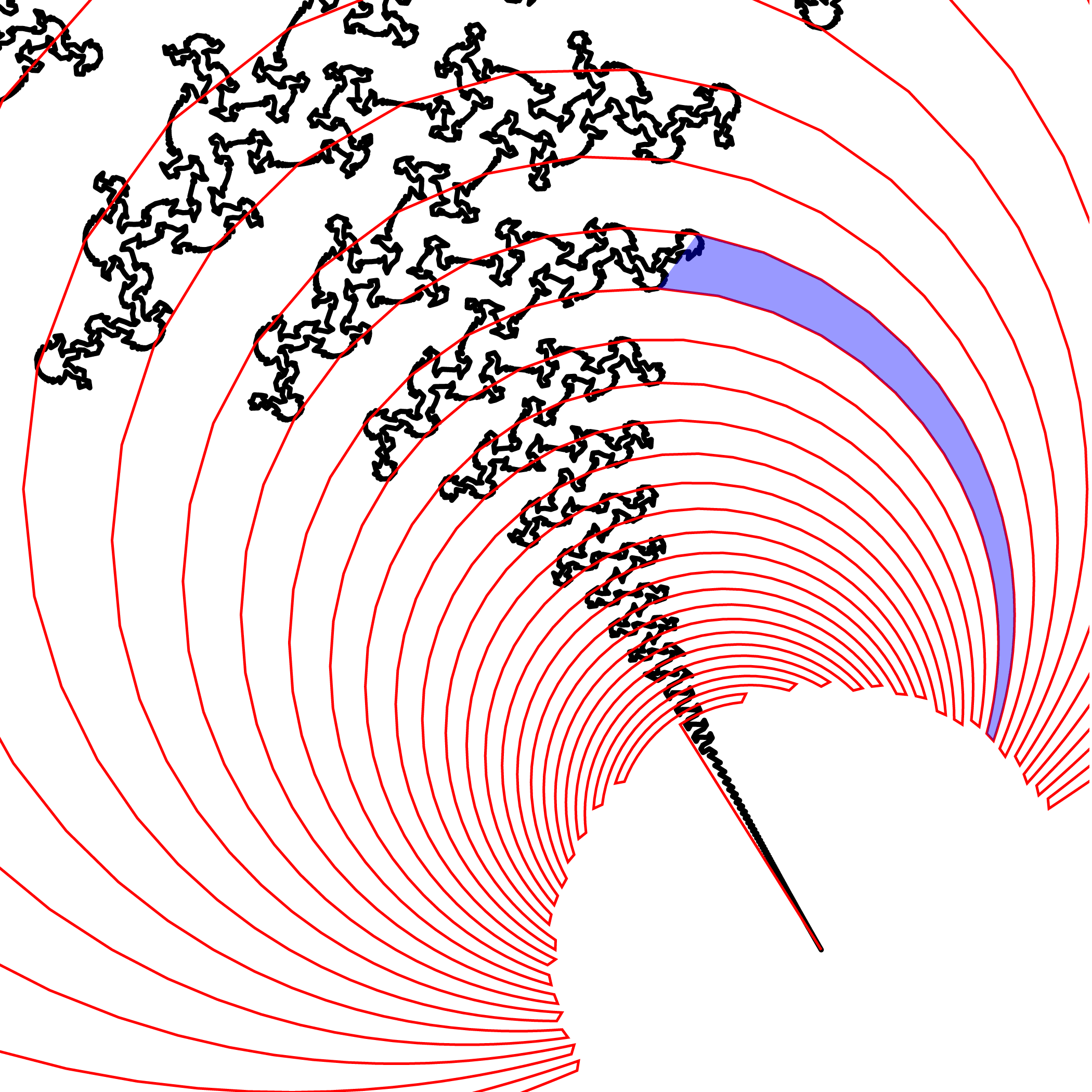}
\caption{A close-up of our recreation of Thurston's algorithm (in black) overlaid with a curve (in red) produced by our new algorithm.
The ratio of feature sizes is about 36.
The thin neck is near the intersection of the black curve and the blue shaded region. 
The ``head'', on the other side of the neck, is approximated by the blue shaded region.}
\label{Fig:CT_fig8_neck}
\end{figure}

\subsection{Our algorithm}
\label{Sec:OurAlgorithm}

As discussed above, we produce a continent instead of a disc.
We choose our continent to ensure that all edges on the boundary of the continent have length less than $\epsilon$.

Our algorithm requires the following inputs:
\begin{itemize}
\item
an initial tetrahedron $t_0$ in $\cover{\calV}$ and
\item 
a feature size $\epsilon$.
\end{itemize}

The algorithm starts with $C = t_0$.
We map the vertices of $t_0$ to $\infty$, $0$, $1$, and $z$, the shape parameter for $t_0$.
(As before, this fixes a developing map.)
We again use the ball model, centred on the same point $p_0$.
Next, for each face $f$ in $\bdy C$ we do the following.
Suppose that $t$ is a tetrahedron in $\cover{\calV}$ adjacent to $f$ and not in $C$.
We replace $C$ by the smallest \emph{convex} continent that contains $C \cup t$ if
\begin{itemize}
\item the length of some edge of $f$ is greater than the feature size $\epsilon$.
\end{itemize}  

We define convexity in~\cite[Definition~5.17]{FSS22}.
We detail the algorithm to produce the convex continent in the proof of~\cite[Proposition~5.4]{FSS22}.

\begin{theorem}
\label{Thm:Terminates}
The continent algorithm terminates.
\end{theorem}

\begin{proof}
Suppose for contradiction that it does not terminate.
Then there is a sequence $(e_n)$ of distinct hyperbolic edges whose projections to $\bdyi \HH^3$ have length greater than or equal to $\epsilon$. 
Since $\bdyi \HH^3$ is compact we may pass to a subsequence, again denoted $(e_n)$, where the projections converge.
Thus for any $\delta$ we may find a point $p$ of $\HH^3$ so that infinitely many of the $e_n$ intersect the $\delta$--ball about $p$.
This contradicts the local finiteness of the (geometric) triangulation $\cover{\calV}$.
\end{proof}

\begin{figure}[htbp]
\includegraphics[width = \textwidth]{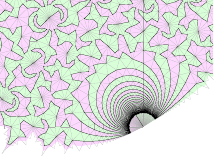}
\caption{The upper (green) and lower (purple) landscapes with common boundary the approximation to the Cannon--Thurston map.}
\label{Fig:two_landscapes}
\end{figure}

\reffig{two_landscapes} shows an approximation to the Cannon--Thurston map as the image of the coast of the continent. 
We also draw the edges of the landscapes on the boundary of the continent.
In particular, there are landscape edges that connect the coast to the large cusp in the lower right of the figure. 
Requiring that these edges be short forces the continent algorithm to explore the pointy ``heads'' atop the thin necks (that Thurston's algorithm cannot see).

\begin{remark}
\label{Rem:DiscVsContinent}
\reffig{CT_fig8_intro} shows two approximations to the Cannon--Thurston map associated to the figure-eight knot complement.
The curve produced by Thurston's algorithm is in black;
it consists of 835,774 arcs and uses a feature size of $0.00265752$.
The curve produced by our algorithm is the common boundary of the green and purple regions;
it consists of 83,576 arcs and uses a feature size of $0.04$.
Note that our curve has one-tenth the arcs and uses a feature size that is (a bit more than) fifteen times larger.
Nonetheless, our curve fills the sphere much more evenly.
\end{remark}

\begin{lemma}
\label{Lem:InjectivityRadius}
Suppose that $\calV$ is geometric.
Suppose that $0 < \epsilon \leq \pi/2$. 
Suppose that $\psi_\calV$ is the approximation produced by the continent method.
Then the complement of $\psi_\calV$ in $\bdyi \HH^3$ has injectivity radius at most $\epsilon$.
\end{lemma}

\begin{proof}
Since $\calV$ is geometric the ideal hyperbolic tetrahedra of $\cover{\calV}$ have shapes with positive imaginary part.
Set $\closure{\HH}^3 = \HH^3 \cup \bdyi \HH^3$;
this is a closed three-ball in $\RR^3$. 
Positivity of the shapes ensures that $C$, the continent, is a three-ball embedded in $\closure{\HH}^3$. 
(The embedding is almost proper, in the sense that the preimage of $\bdyi \HH^3$ is only a finite subset of $\bdy C$:
the ideal vertices of the tetrahedra in $C$.)
Thus $\bdy C$, the boundary of the continent, is a topological two-sphere embedded in $\closure{\HH}^3$. 

Recall that $p_0$, our basepoint, lies in the edge $e_0$ connecting $0$ to $\infty$. 
The spherical distance between $0$ and $\infty$ is $\pi$.
Since $\epsilon \leq \pi/2$ all triangles meeting $e_0$ lie in the interior of $C$.
Thus the basepoint $p_0$ also lies in the interior of $C$.
It follows that the outward radial projection of $\bdy C$ to $\bdyi \HH^3$ is a degree-one map.
In particular, every point of $\bdyi \HH^3$ lies in the image of some ideal hyperbolic triangle of $\bdy C$.
Also, the approximate Cannon--Thurston map $\psi_\calV$ meets every ideal vertex of $C$.

Now, all faces of $\bdy C$ are disjoint from $p_0$.
Fix one such.  
Its outward radial projection is a spherical triangle $t$ with edges of length at most $\epsilon$.
Since $\epsilon \leq \pi/2$, we find that $t$ is contained in some hemisphere.
We deduce that the (closed) $\epsilon$--disc about any one vertex of $t$ contains $t$.
This implies that the vertices of the approximation $\psi_\calV$ are $\epsilon$--dense in $\bdyi \HH^3$, as desired.
\end{proof}

\subsection{Non-layered example}

\begin{figure}[htbp]
\vspace{-10 pt}
\subfloat[A layered example. ]{
\includegraphics[width = 0.46 \textwidth]{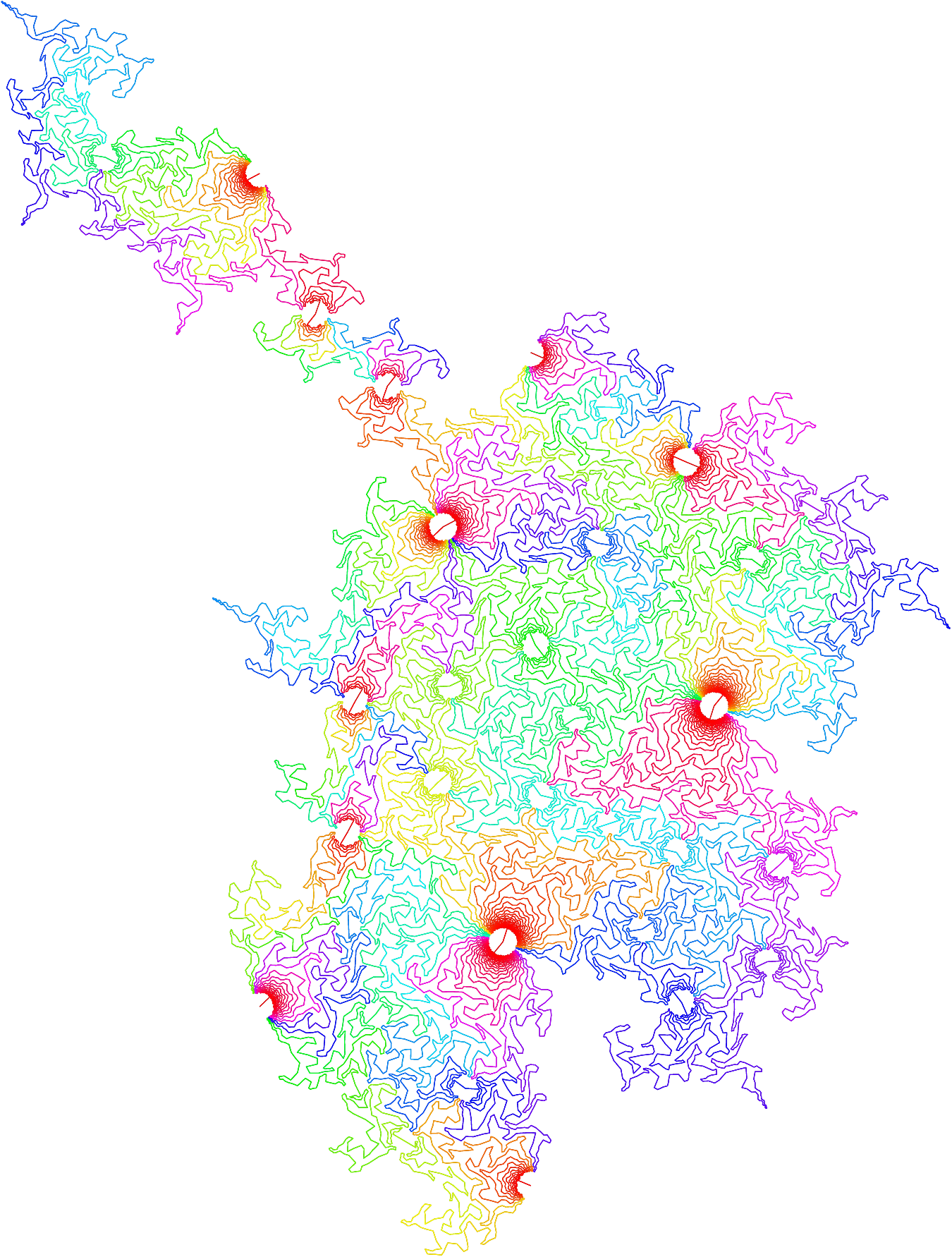}
\label{Fig:CT_example_fibred}
}
\subfloat[A non-layered example.]{
\includegraphics[width = 0.46 \textwidth]{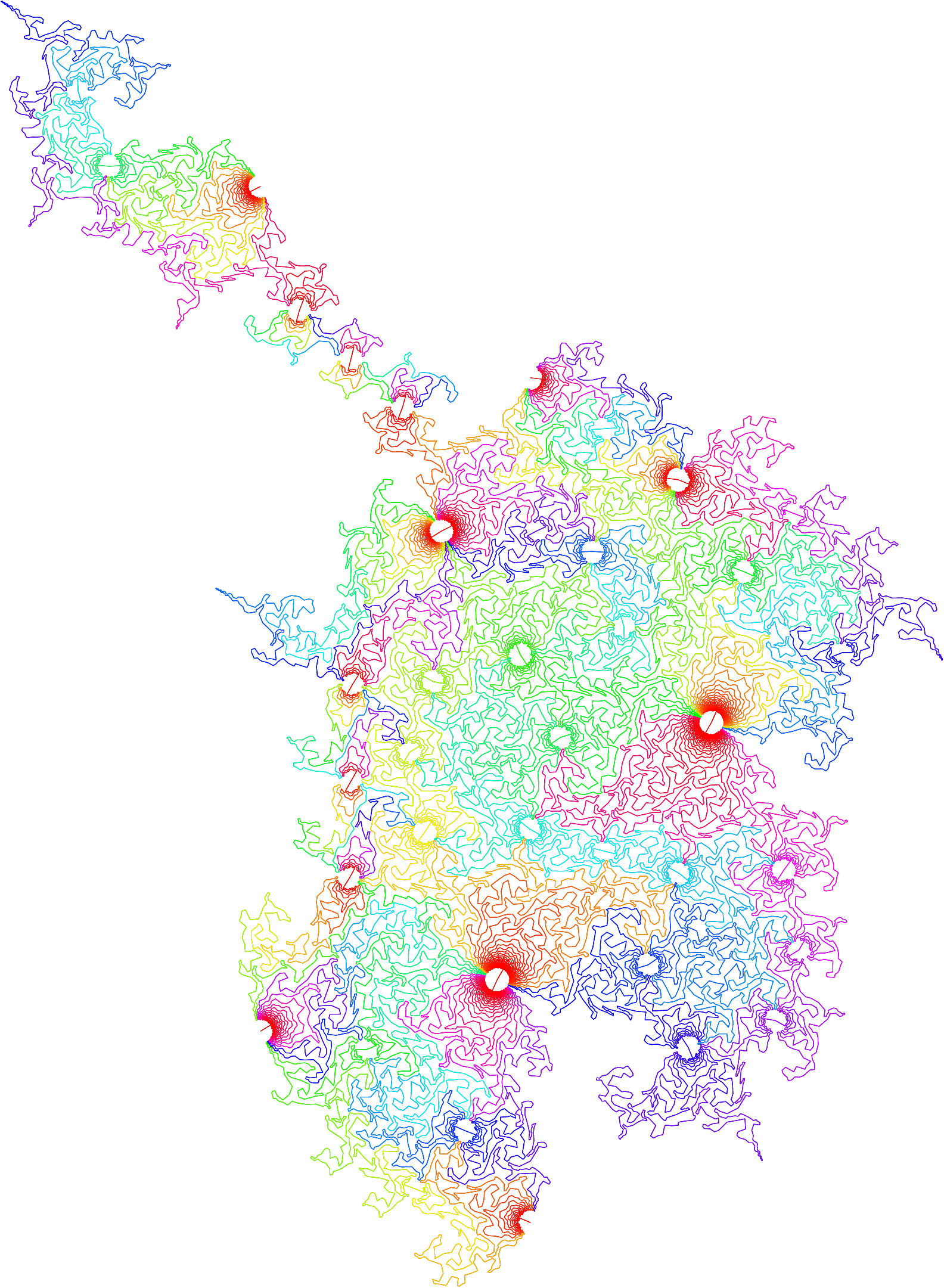}
\label{Fig:CT_example_non_fibred}
}
\caption{Fundamental domains for the approximations to the Cannon--Thurston maps in both a layered example (\usebox{\fibredVeer}) and in a non-layered example (\usebox{\NonfibredVeer}). We cycle colours around the spectrum as we draw each curve in an attempt to illustrate the proximity of points along the curve.}
\label{Fig:CT_examples}
\end{figure}

\reffig{CT_examples} shows repeating units for approximations to the Cannon--Thurston maps for two veering triangulations.
In the notation of the veering census~\cite{GSS19} these are \usebox{\fibredVeer}  and \usebox{\NonfibredVeer}.
The latter is a triangulation of the manifold \usebox{\NonfibredSnappy} from the SnapPy census~\cite{snappy}, which is non-fibred~\cite[Section~4]{HRST11} (and thus the veering triangulation is not layered). 
As mentioned in \refexa{s227}, \reffig{CT_example_non_fibred} could not have been generated by previous techniques.

On the other hand, \usebox{\fibredVeer} is a layered triangulation of the (therefore) fibred manifold \usebox{\fibredSnappy}.
The maps shown in \reffig{CT_examples} illustrate that the difference between layered and non-layered Cannon--Thurston maps is somewhat subtle.
Indeed, the triangulations \usebox{\NonfibredVeer} and \usebox{\fibredVeer} are related by a single application of ``veering Dehn surgery''.

\subsection{Questions}

Calegari and Loukidou~\cite[page~26]{CalegariLoukidou26} point out that, in the closed case, CaTherine wheels are approximated by embeddings.
The same argument applies to our surjection $\Circle \to \Sphere$.
However, this does not answer the following.

\begin{question}
In all of the examples we have produced (using the algorithm of \refsec{OurAlgorithm}), our approximate Cannon--Thurston maps are embedded. 
In particular, this happens even when the veering triangulation is not geometric. 
Are these approximate maps embeddings?
\end{question}


At feature size $\epsilon$, by \reflem{InjectivityRadius} the continent algorithm (\refsec{OurAlgorithm}) has injectivity radius at most $\epsilon$.
In \refsec{ThurstonAlgorithm}, for the example in \reffig{Thurston} we visually estimated the injectivity radius of Thurston's algorithm.

\begin{question}
Give useful bounds for the injectivity radius of Thurston's algorithm at feature size $\epsilon$.
\end{question}

In \refrem{DiscVsContinent} we observe that in at least one example, our algorithm does about a tenth the work (measured in number of arcs) at a larger feature size and nonetheless produces an approximation with far smaller injectivity radius.
We suspect that this holds more generally.

\begin{question}
\label{Que:ThurstonWork}
Estimate the amount of work done by Thurston's algorithm at feature size $\epsilon$. 
(And so justify the quote given in \refsec{ThurstonAlgorithm}.)
\end{question}

\begin{question}
\label{Que:OurWork}
Estimate the amount of work done by the continent algorithm at feature size~$\epsilon$.
\end{question}

Instead of using injectivity radius to measure the quality of an approximation to a Cannon--Thurston map, 
Jeremy Kahn suggests using $C^0$ distance on the space of maps from the circle to the sphere.
Since $C^0$ distance is defined via a supremum, giving lower bounds is straightforward:
one looks for arcs of the approximation which stray far from the corresponding ``true'' position.
In the case of Thurston's approximation, this amounts to finding thin necks; see \reffig{CT_fig8_neck}.

Providing upper bounds is much more difficult, as it requires understanding the shape of the discs bounded by \emph{lightning curves}~\cites{DicksWright12}{Gueritaud16}.

\begin{question}
Estimate how close the disc method and continent method come (as measured by the work done) to the true Cannon--Thurston map in $C^0$.
\end{question}

\section{Curves, continua, discs, and spheres}
\label{App:Topology}

In this appendix we gather the necessary point-set topology needed for our constructions of the veering two-sphere~(\refsec{Sphere}) and veering disc~(\refsec{Disc}).
We rely on Wilder's book~\cite{Wilder49}, and the somewhat more recent text of Steen and Seebach~\cite{SteenSeebach95}, as our standard references for topological matters.

\subsection{Jordan curves}

If $\sigma \from S^1 \to \RR^2$ is an embedding,
then we call the image $\Sigma = \sigma(S^1)$ a \emph{Jordan curve}~\cite[page~107]{Osgood03}.
If $\sigma \from \RR^1 \to \RR^2$ is a proper embedding
then we call the image $\Sigma = \sigma(\RR^1)$ a \emph{Jordan line}.

We will need a version of the Jordan curve theorem~\cite[page~593]{Jordan87} called the \emph{Schoenflies extension theorem}~\cite[page~324]{Schoenflies06}.  
This statement follows from~\cite[page~94]{Wilder49}.  


\begin{theorem}
\label{Thm:Schoenflies}
Suppose that $\Sigma \subset \RR^2$ is a Jordan curve (Jordan line).  
Then there is an orientation-preserving homeomorphism of $\RR^2$ taking $\Sigma$ to the unit circle $S^1$ (to the $x$--axis).
\qed
\end{theorem}


\subsection{Peano continua}

\begin{definition}
\label{Def:PeanoSpace}
A topological space $P$ is a \emph{Peano space}~\cite[page~76]{Wilder49} if it is non-degenerate~\cite[page~2]{Wilder49}, perfectly separable~\cite[page~70]{Wilder49}, normal~\cite[page~49]{Wilder49}, locally compact~\cite[page~29]{Wilder49}, connected~\cite[page~7]{Wilder49}, and locally connected~\cite[page~40]{Wilder49}.
\end{definition}

Note that many authors use the terminology \emph{second countable}~\cite[page~7]{SteenSeebach95} instead of ``perfectly separable'';
we will do the same.  
We also note the standard result that if a space $Q$ is Hausdorff~\cite[page~69]{Wilder49} and compact~\cite[page~34]{Wilder49}, then $Q$ is normal~\cite[Theorem~1.27, page~74]{Wilder49}.

\begin{definition}
\label{Def:PeanoContinuum}
A Peano space $P$ is a \emph{Peano continuum}~\cite[page~76]{Wilder49} if it is compact. 
\end{definition}

\subsection{Spanning arcs of the disc}

Suppose that $D$ is a topological space and $C$ is a subspace.
Suppose that $(I, \bdy I)$ is the unit interval.
Suppose that $\rho \from (I, \bdy I) \to (D, C)$ is an embedding of pairs with $\rho^{-1}(C) = \bdy I$.
Then we will call $\rho$ a \emph{spanning arc} for the pair $(D, C)$~\cite[pages~27 and~88]{Wilder49}.

We use Zippin's characterisation of the disc~\cite{Zippin33}.
As usual we follow Wilder's terminology;
see~\cite[page~92]{Wilder49}.

\begin{theorem}
\label{Thm:Zippin}
Suppose that $D$ is a Peano continuum containing a circle $C$.
Suppose that we have the following.
\begin{enumerate}
\item
\label{Itm:Exists}
$D$ contains an arc that spans $C$.
\item
\label{Itm:All}
Every arc of $D$ that spans $C$ separates $D$.
\item
\label{Itm:None}
No closed proper subset of an arc spanning $C$ separates $D$.
\end{enumerate}
Then there is a homeomorphism of pairs $(D, C) \homeo (D^2, S^1)$. \qed
\end{theorem}

\subsection{Quotients of the two-sphere}

We use Moore's theorem~\cite[Theorem~25]{Moore25}.  
We follow~\cite[Theorem~1.1]{Timorin10};
see also~\cite[Theorem~10.18]{Calegari07}.

\begin{definition}
\label{Def:ClosedEquiv}
An equivalence relation on a topological space $X$ is \emph{closed} if it is closed as a subset of $X \times X$.
\end{definition}

\begin{theorem}
\label{Thm:Moore}
Suppose that $E$ is a closed equivalence relation on $S^2$ with at least two classes.  
Suppose that, for all $x$, the class $[x]_E$ is closed, connected, and non-separating.  
Then the quotient $S^2/E$ is homeomorphic to $S^2$. \qed
\end{theorem}

\renewcommand{\UrlFont}{\tiny\ttfamily}
\renewcommand\hrefdefaultfont{\tiny\ttfamily}
\bibliographystyle{plainurl}
\bibliography{veering_to_convergence}

\begin{thebibliography}{10}

\bibitem{Agol11}
Ian Agol.
\newblock Ideal triangulations of pseudo-{A}nosov mapping tori.
\newblock In {\em Topology and geometry in dimension three}, volume 560 of {\em
  Contemp. Math.}, pages 1--17. Amer. Math. Soc., Providence, RI, 2011.
\newblock \href {https://doi.org/10.1090/conm/560/11087}
  {\path{doi:10.1090/conm/560/11087}}.

\bibitem{AlperinDicksPorti99}
Roger~C. Alperin, Warren Dicks, and Joan Porti.
\newblock The boundary of the {G}ieseking tree in hyperbolic three-space.
\newblock {\em Topology Appl.}, 93(3):219--259, 1999.
\newblock \href {https://doi.org/10.1016/S0166-8641(97)00270-8}
  {\path{doi:10.1016/S0166-8641(97)00270-8}}.

\bibitem{BachmanGoernerSchleimerSegerman22}
David Bachman, Matthias Goerner, Saul Schleimer, and Henry Segerman.
\newblock Cohomology fractals, {C}annon-{T}hurston maps, and the geodesic flow.
\newblock {\em Exp. Math.}, 31(4):1047--1085, 2022.
\newblock \href {https://doi.org/10.1080/10586458.2021.1994059}
  {\path{doi:10.1080/10586458.2021.1994059}}.

\bibitem{BakerRiley13}
O.~Baker and T.~R. Riley.
\newblock Cannon-{T}hurston maps do not always exist.
\newblock {\em Forum Math. Sigma}, 1:Paper No. e3, 11, 2013.
\newblock \href {https://doi.org/10.1017/fms.2013.4}
  {\path{doi:10.1017/fms.2013.4}}.

\bibitem{BarthelmeMann26}
Thomas Barthelmé and Kathryn Mann.
\newblock Pseudo-{A}nosov flows: a plane approach, 2026.
\newblock \href {https://arxiv.org/abs/2509.15375} {\path{arXiv:2509.15375}}.

\bibitem{BeardonMaskit74}
Alan~F. Beardon and Bernard Maskit.
\newblock Limit points of {K}leinian groups and finite sided fundamental
  polyhedra.
\newblock {\em Acta Math.}, 132:1--12, 1974.
\newblock \href {https://doi.org/10.1007/BF02392106}
  {\path{doi:10.1007/BF02392106}}.

\bibitem{BieriEckmann78}
Robert Bieri and Beno Eckmann.
\newblock Relative homology and {P}oincar\'e{} duality for group pairs.
\newblock {\em J. Pure Appl. Algebra}, 13(3):277--319, 1978.
\newblock \href {https://doi.org/10.1016/0022-4049(78)90012-9}
  {\path{doi:10.1016/0022-4049(78)90012-9}}.

\bibitem{Bonatti26}
Christian Bonatti.
\newblock Action on the circle at infinity of foliations of {$\Bbb R^2$}.
\newblock {\em Enseign. Math.}, 72(1-2):85--146, 2026.
\newblock \href {https://doi.org/10.4171/lem/1086}
  {\path{doi:10.4171/lem/1086}}.

\bibitem{Bowditch93}
B.~H. Bowditch.
\newblock Geometrical finiteness for hyperbolic groups.
\newblock {\em J. Funct. Anal.}, 113(2):245--317, 1993.
\newblock \href {https://doi.org/10.1006/jfan.1993.1052}
  {\path{doi:10.1006/jfan.1993.1052}}.

\bibitem{Bowditch99}
B.~H. Bowditch.
\newblock Convergence groups and configuration spaces.
\newblock In {\em Geometric group theory down under ({C}anberra, 1996)}, pages
  23--54. de Gruyter, Berlin, 1999.

\bibitem{Bowditch98}
Brian~H. Bowditch.
\newblock A topological characterisation of hyperbolic groups.
\newblock {\em J. Amer. Math. Soc.}, 11(3):643--667, 1998.
\newblock \href {https://doi.org/10.1090/S0894-0347-98-00264-1}
  {\path{doi:10.1090/S0894-0347-98-00264-1}}.

\bibitem{Bowditch07}
Brian~H. Bowditch.
\newblock The {C}annon-{T}hurston map for punctured-surface groups.
\newblock {\em Math. Z.}, 255(1):35--76, 2007.
\newblock \href {https://doi.org/10.1007/s00209-006-0012-4}
  {\path{doi:10.1007/s00209-006-0012-4}}.

\bibitem{Buckminster26}
Ellis Buckminster.
\newblock {C}annon--{T}hurston maps for {A}nosov foliations, 2026.
\newblock \href {https://arxiv.org/abs/2604.21201} {\path{arXiv:2604.21201}}.

\bibitem{regina}
Benjamin~A. Burton, Ryan Budney, William Pettersson, et~al.
\newblock Regina: Software for low-dimensional topology, 1999--2026.
\newblock \url{http://regina-normal.github.io/}.

\bibitem{Calegari07}
Danny Calegari.
\newblock {\em Foliations and the geometry of 3-manifolds}.
\newblock Oxford Mathematical Monographs. Oxford University Press, Oxford,
  2007.

\bibitem{CalegariDunfield03}
Danny Calegari and Nathan~M. Dunfield.
\newblock Laminations and groups of homeomorphisms of the circle.
\newblock {\em Invent. Math.}, 152(1):149--204, 2003.
\newblock \href {https://doi.org/10.1007/s00222-002-0271-6}
  {\path{doi:10.1007/s00222-002-0271-6}}.

\bibitem{CalegariLoukidou26}
Danny Calegari and Ino Loukidou.
\newblock Ca{T}herine wheels, 2026.
\newblock \href {https://arxiv.org/abs/2604.24619} {\path{arXiv:2604.24619}}.

\bibitem{CannonThurston07}
James~W. Cannon and William~P. Thurston.
\newblock Group invariant {P}eano curves.
\newblock {\em Geom. Topol.}, 11:1315--1355, 2007.
\newblock \href {https://doi.org/10.2140/gt.2007.11.1315}
  {\path{doi:10.2140/gt.2007.11.1315}}.

\bibitem{snappy}
Marc Culler, Nathan~M. Dunfield, Matthias Goerner, and Jeffrey~R. Weeks.
\newblock Snap{P}y, a computer program for studying the geometry and topology
  of three-manifolds.
\newblock \url{http://snappy.computop.org}.

\bibitem{DicksWright12}
Warren Dicks and David~J. Wright.
\newblock On hyperbolic once-punctured-torus bundles {IV}: {A}utomata for
  lightning curves.
\newblock {\em Topology Appl.}, 159(1):98--132, 2012.
\newblock \href {https://doi.org/10.1016/j.topol.2011.08.018}
  {\path{doi:10.1016/j.topol.2011.08.018}}.

\bibitem{Fenley12}
S\'ergio Fenley.
\newblock Ideal boundaries of pseudo-{A}nosov flows and uniform convergence
  groups with connections and applications to large scale geometry.
\newblock {\em Geom. Topol.}, 16(1):1--110, 2012.
\newblock \href {https://doi.org/10.2140/gt.2012.16.1}
  {\path{doi:10.2140/gt.2012.16.1}}.

\bibitem{Fenley16}
S\'{e}rgio~R. Fenley.
\newblock Quasigeodesic pseudo-{A}nosov flows in hyperbolic 3-manifolds and
  connections with large scale geometry.
\newblock {\em Adv. Math.}, 303:192--278, 2016.
\newblock \href {https://doi.org/10.1016/j.aim.2016.05.015}
  {\path{doi:10.1016/j.aim.2016.05.015}}.

\bibitem{FLT26}
Sérgio~R. Fenley, Michael~P. Landry, and Samuel~J. Taylor.
\newblock Simultaneous universal circles and continuous extension, 2026.
\newblock \href {https://arxiv.org/abs/2607.05703} {\path{arXiv:2607.05703}}.

\bibitem{Frankel15}
Steven Frankel.
\newblock Quasigeodesic flows and sphere-filling curves.
\newblock {\em Geom. Topol.}, 19(3):1249--1262, 2015.
\newblock \href {https://doi.org/10.2140/gt.2015.19.1249}
  {\path{doi:10.2140/gt.2015.19.1249}}.

\bibitem{FSS22}
Steven Frankel, Saul Schleimer, and Henry Segerman.
\newblock From veering triangulations to link spaces and back again, 2022.
\newblock \href {https://arxiv.org/abs/1911.00006} {\path{arXiv:1911.00006}}.

\bibitem{Freden95}
Eric~M. Freden.
\newblock Negatively curved groups have the convergence property. {I}.
\newblock {\em Ann. Acad. Sci. Fenn. Ser. A I Math.}, 20(2):333--348, 1995.

\bibitem{Freden97}
Eric~M. Freden.
\newblock Properties of convergence groups and spaces.
\newblock {\em Conform. Geom. Dyn.}, 1:13--23 (electronic), 1997.
\newblock \href {https://doi.org/10.1090/S1088-4173-97-00011-8}
  {\path{doi:10.1090/S1088-4173-97-00011-8}}.

\bibitem{FuterGueritaud13}
David Futer and Fran\c{c}ois Gu\'{e}ritaud.
\newblock Explicit angle structures for veering triangulations.
\newblock {\em Algebr. Geom. Topol.}, 13(1):205--235, 2013.
\newblock \href {https://arxiv.org/abs/1012.5134} {\path{arXiv:1012.5134}},
  \href {https://doi.org/10.2140/agt.2013.13.205}
  {\path{doi:10.2140/agt.2013.13.205}}.

\bibitem{GehringMartin}
F.~W. Gehring and G.~J. Martin.
\newblock Discrete convergence groups.
\newblock In {\em Complex analysis, {I} ({C}ollege {P}ark, {M}d., 1985--86)},
  volume 1275 of {\em Lecture Notes in Math.}, pages 158--167. Springer,
  Berlin, 1987.
\newblock \href {https://doi.org/10.1007/BFb0078350}
  {\path{doi:10.1007/BFb0078350}}.

\bibitem{GSS19}
Andreas Giannopolous, Saul Schleimer, and Henry Segerman.
\newblock A census of veering structures, 2025.
\newblock \url{https://math.okstate.edu/people/segerman/veering.html}.

\bibitem{Gueritaud13}
Fran\c{c}ois Gu\'{e}ritaud, 2013-05-28.
\newblock Personal communication.

\bibitem{Gueritaud16}
Fran\c{c}ois Gu\'{e}ritaud.
\newblock Veering triangulations and {C}annon-{T}hurston maps.
\newblock {\em J. Topol.}, 9(3):957--983, 2016.
\newblock \href {https://doi.org/10.1112/jtopol/jtw016}
  {\path{doi:10.1112/jtopol/jtw016}}.

\bibitem{HRST11}
Craig~D. Hodgson, J.~Hyam Rubinstein, Henry Segerman, and Stephan Tillmann.
\newblock Veering triangulations admit strict angle structures.
\newblock {\em Geom. Topol.}, 15(4):2073--2089, 2011.
\newblock \href {https://arxiv.org/abs/1011.3695} {\path{arXiv:1011.3695}},
  \href {https://doi.org/10.2140/gt.2011.15.2073}
  {\path{doi:10.2140/gt.2011.15.2073}}.

\bibitem{Jordan87}
Camille Jordan.
\newblock {\em Cours d'analyse de l'\'{E}cole polytechnique. {T}ome {III}}.
\newblock Gauthier-Villars, 1887.
\newblock [Calcul Int\'{e}gral. \'{E}quations diff\'{e}rentielles.].

\bibitem{Kapovich10}
Michael Kapovich.
\newblock {\em Hyperbolic manifolds and discrete groups}.
\newblock Modern Birkh\"auser Classics. Birkh\"auser Boston, Ltd., Boston, MA,
  2009.
\newblock Reprint of the 2001 edition.
\newblock \href {https://doi.org/10.1007/978-0-8176-4913-5}
  {\path{doi:10.1007/978-0-8176-4913-5}}.

\bibitem{Lackenby00}
Marc Lackenby.
\newblock Taut ideal triangulations of 3-manifolds.
\newblock {\em Geom. Topol.}, 4:369--395, 2000.
\newblock \href {https://doi.org/10.2140/gt.2000.4.369}
  {\path{doi:10.2140/gt.2000.4.369}}.

\bibitem{pyx}
J{\"o}rg Lehmann, Michael Schindler, and Andr{\'e} Wobst.
\newblock {PyX}: {A} {Python} package for the generation of {PostScript},
  {PDF}, and {SVG} graphics, 2025.
\newblock Version 0.17.
\newblock \url{https://pyx-project.org}.

\bibitem{Mather82}
John~N. Mather.
\newblock Foliations of surfaces. {I}. {A}n ideal boundary.
\newblock {\em Ann. Inst. Fourier (Grenoble)}, 32(1):viii, 235--261, 1982.
\newblock \href {https://doi.org/10.5802/aif.867} {\path{doi:10.5802/aif.867}}.

\bibitem{Mj98top}
Mahan Mitra.
\newblock Cannon--{T}hurston maps for hyperbolic group extensions.
\newblock {\em Topology}, 37(3):527--538, 1998.
\newblock \href {https://doi.org/10.1016/S0040-9383(97)00036-0}
  {\path{doi:10.1016/S0040-9383(97)00036-0}}.

\bibitem{Mj98jdg}
Mahan Mitra.
\newblock Cannon--{T}hurston maps for trees of hyperbolic metric spaces.
\newblock {\em J. Differential Geom.}, 48(1):135--164, 1998.

\bibitem{Mj09}
Mahan Mj.
\newblock Cannon--{T}hurston maps for pared manifolds of bounded geometry.
\newblock {\em Geom. Topol.}, 13(1):189--245, 2009.
\newblock \href {https://doi.org/10.2140/gt.2009.13.189}
  {\path{doi:10.2140/gt.2009.13.189}}.

\bibitem{Mj17}
Mahan Mj.
\newblock Cannon--{T}hurston maps for {K}leinian groups.
\newblock {\em Forum Math. Pi}, 5:e1, 49, 2017.
\newblock \href {https://doi.org/10.1017/fmp.2017.2}
  {\path{doi:10.1017/fmp.2017.2}}.

\bibitem{Mj18}
Mahan Mj.
\newblock Cannon--{T}hurston maps.
\newblock pages 885--917, 2018.

\bibitem{MjPal11}
Mahan Mj and Abhijit Pal.
\newblock Relative hyperbolicity, trees of spaces and {C}annon-{T}hurston maps.
\newblock {\em Geom. Dedicata}, 151:59--78, 2011.
\newblock \href {https://doi.org/10.1007/s10711-010-9519-2}
  {\path{doi:10.1007/s10711-010-9519-2}}.

\bibitem{MjSardar12}
Mahan Mj and Pranab Sardar.
\newblock A combination theorem for metric bundles.
\newblock {\em Geom. Funct. Anal.}, 22(6):1636--1707, 2012.
\newblock \href {https://doi.org/10.1007/s00039-012-0196-1}
  {\path{doi:10.1007/s00039-012-0196-1}}.

\bibitem{Moore25}
Robert~L. Moore.
\newblock Concerning upper semi-continuous collections of continua.
\newblock {\em Trans. Amer. Math. Soc.}, 27(4):416--428, 1925.
\newblock \href {https://doi.org/10.2307/1989234} {\path{doi:10.2307/1989234}}.

\bibitem{MSW02}
David Mumford, Caroline Series, and David Wright.
\newblock {\em Indra's pearls}.
\newblock Cambridge University Press, New York, 2002.
\newblock The vision of Felix Klein.
\newblock \href {https://doi.org/10.1017/CBO9781107050051.024}
  {\path{doi:10.1017/CBO9781107050051.024}}.

\bibitem{Osgood03}
William~F. Osgood.
\newblock A {J}ordan curve of positive area.
\newblock {\em Trans. Amer. Math. Soc.}, 4(1):107--112, 1903.
\newblock \href {https://doi.org/10.2307/1986455} {\path{doi:10.2307/1986455}}.

\bibitem{VeeringCodebase}
Anna Parlak, Saul Schleimer, and Henry Segerman.
\newblock veering 0.3, code for studying taut and veering ideal triangulations,
  2025.
\newblock \url{https://github.com/henryseg/Veering}.

\bibitem{Prasad73}
Gopal Prasad.
\newblock Strong rigidity of {${\bf Q}$}-rank {$1$} lattices.
\newblock {\em Invent. Math.}, 21:255--286, 1973.
\newblock \href {https://doi.org/10.1007/BF01418789}
  {\path{doi:10.1007/BF01418789}}.

\bibitem{rhino}
{Robert McNeel \& Associates}.
\newblock {\em Rhinoceros, Version 8.0}.
\newblock Robert McNeel \& Associates, Seattle, WA, 2024.
\newblock \url{https://rhino3d.com}.

\bibitem{SchleimerSegerman20}
Saul Schleimer and Henry Segerman.
\newblock Essential loops in taut ideal triangulations.
\newblock {\em Algebr. Geom. Topol.}, 20(1):487--501, 2020.
\newblock \href {https://doi.org/10.2140/agt.2020.20.487}
  {\path{doi:10.2140/agt.2020.20.487}}.

\bibitem{SchleimerSegerman24}
Saul Schleimer and Henry Segerman.
\newblock From loom spaces to veering triangulations.
\newblock {\em Groups Geom. Dyn.}, 18(2):419--462, 2024.
\newblock \href {https://doi.org/10.4171/ggd/742} {\path{doi:10.4171/ggd/742}}.

\bibitem{Schoenflies06}
A.~Schoenflies.
\newblock Beitr\"{a}ge zur {T}heorie der {P}unktmengen. {III}.
\newblock {\em Math. Ann.}, 62(2):286--328, 1906.
\newblock \href {https://doi.org/10.1007/BF01449982}
  {\path{doi:10.1007/BF01449982}}.

\bibitem{SteenSeebach95}
Lynn~Arthur Steen and J.~Arthur Seebach, Jr.
\newblock {\em Counterexamples in topology}.
\newblock Dover Publications, Inc., Mineola, NY, 1995.
\newblock Reprint of the second (1978) edition.

\bibitem{sage}
{The Sage Developers}.
\newblock {\em {S}age{M}ath, the {S}age {M}athematics {S}oftware {S}ystem},
  2025.
\newblock \url{https://www.sagemath.org}, \href
  {https://doi.org/10.5281/zenodo.6259615} {\path{doi:10.5281/zenodo.6259615}}.

\bibitem{Thurston80}
William~P. Thurston.
\newblock Geometry and topology of three-manifolds.
\newblock Lecture notes, 1978.
\newblock \url{https://library.slmath.org/nonmsri/gt3m/}.

\bibitem{Thurston82}
William~P. Thurston.
\newblock Three-dimensional manifolds, {K}leinian groups and hyperbolic
  geometry.
\newblock {\em Bull. Amer. Math. Soc. (N.S.)}, 6(3):357--381, 1982.
\newblock \href {https://doi.org/10.1090/S0273-0979-1982-15003-0}
  {\path{doi:10.1090/S0273-0979-1982-15003-0}}.

\bibitem{Thurston88}
William~P. Thurston.
\newblock On the geometry and dynamics of diffeomorphisms of surfaces.
\newblock {\em Bull. Amer. Math. Soc. (N.S.)}, 19(2):417--431, 1988.
\newblock \href {https://doi.org/10.1090/S0273-0979-1988-15685-6}
  {\path{doi:10.1090/S0273-0979-1988-15685-6}}.

\bibitem{Thurston97}
William~P. Thurston.
\newblock Three-manifolds, foliations and circles, {I}, 1997.
\newblock \href {https://arxiv.org/abs/math/9712268}
  {\path{arXiv:math/9712268}}.

\bibitem{Thurston98}
William~P. Thurston.
\newblock Hyperbolic structures on 3-manifolds, {II}: {S}urface groups and
  3-manifolds which fiber over the circle, 1998.
\newblock \href {https://arxiv.org/abs/math/9801045}
  {\path{arXiv:math/9801045}}.

\bibitem{Timorin10}
Vladlen Timorin.
\newblock Moore's theorem, 2010.
\newblock \href {https://arxiv.org/abs/arXiv:1001.5140}
  {\path{arXiv:arXiv:1001.5140}}.

\bibitem{Tukia94}
Pekka Tukia.
\newblock Convergence groups and {G}romov's metric hyperbolic spaces.
\newblock {\em New Zealand J. Math.}, 23(2):157--187, 1994.

\bibitem{Waldhausen68}
Friedhelm Waldhausen.
\newblock On irreducible {$3$}-manifolds which are sufficiently large.
\newblock {\em Ann. of Math. (2)}, 87:56--88, 1968.
\newblock \href {https://doi.org/10.2307/1970594} {\path{doi:10.2307/1970594}}.

\bibitem{Wilder49}
Raymond~Louis Wilder.
\newblock {\em Topology of {M}anifolds}.
\newblock American Mathematical Society Colloquium Publications, vol. 32.
  American Mathematical Society, New York, N. Y., 1949.

\bibitem{Yaman04}
Asl\i Yaman.
\newblock A topological characterisation of relatively hyperbolic groups.
\newblock {\em J. Reine Angew. Math.}, 566:41--89, 2004.
\newblock \href {https://doi.org/10.1515/crll.2004.007}
  {\path{doi:10.1515/crll.2004.007}}.

\bibitem{Zippin33}
Leo Zippin.
\newblock A {C}haracterisation of the {C}losed 2-{C}ell.
\newblock {\em Amer. J. Math.}, 55(1-4):207--217, 1933.
\newblock \href {https://doi.org/10.2307/2371123} {\path{doi:10.2307/2371123}}.

\end{thebibliography}
\end{document}